\documentclass[10pt]{amsart}

\usepackage{amsmath,amssymb}
\usepackage{amsthm}
\usepackage{stmaryrd}
\usepackage{enumerate}
\usepackage{url}
\usepackage{wasysym}
\usepackage{color}
\usepackage[normalem]{ulem}
\usepackage{hyperref}
\usepackage[T1]{fontenc}
\usepackage{lmodern}
\usepackage{mathtools}
\usepackage{comment}
\usepackage{tikz-cd}

\usepackage[utf8]{inputenc}

\newtheorem{theorem}{Theorem}[section]
\newtheorem{fact}[theorem]{Fact}
\newtheorem*{fact*}{Fact}
\newtheorem{lemma}[theorem]{Lemma}

\newtheorem{corollary}[theorem]{Corollary}
\newtheorem{proposition}[theorem]{Proposition}
\newtheorem{conjecture}[theorem]{Conjecture}
\newtheorem*{theorem*}{Theorem}

\theoremstyle{definition}
\newtheorem{definition}[theorem]{Definition}
\newtheorem{example}[theorem]{Example}

\newtheorem{remark}[theorem]{Remark}
\newtheorem{question}[theorem]{Question}

\newcommand{\abar}{\bar{a}}

\newcommand{\dbar}{\bar{d}}

\newcommand{\fG}{{\mathfrak G}}
\newcommand{\ext}{\textrm{ext}}

\def\eq{\textit{eq}}

\def\tp{\operatorname{tp}}

\def\Av{\operatorname{Av}}

\def\fs{\operatorname{fs}}

\def\cl{\operatorname{cl}}
\def\acl{\operatorname{acl}}
\def\dcl{\operatorname{dcl}}
\def\Stab{\operatorname{Stab}}
\def\bdn{\operatorname{bdn}}

\def\C{\mathbb{C}}

\def\cL{\mathcal{L}}
\def\cU{\mathcal{U}}
\def\cG{\mathcal{G}}
\def\cM{\mathcal{M}}

\def\Aut{\operatorname{Aut}}
\def\Av{\operatorname{Av}}

\def\supp{\operatorname{supp}}
\def\inva{\operatorname{inv}}

\def\im{\operatorname{im}}
\def\CB{\operatorname{CB}}
\def\id{\operatorname{id}}
\def\Gal{\operatorname{Gal}}
\def\ima{\operatorname{Im}}
\def\Sym{\operatorname{Sym}}

\newcommand{\fim}{\textit{fim}}

\def\Ind{\setbox0=\hbox{$x$}\kern\wd0\hbox to 0pt{\hss$\mid$\hss}
\lower.9\ht0\hbox to 0pt{\hss$\smile$\hss}\kern\wd0}

\def\Notind{\setbox0=\hbox{$x$}\kern\wd0\hbox to 0pt{\mathchardef
\nn=12854\hss$\nn$\kern1.4\wd0\hss}\hbox to
0pt{\hss$\mid$\hss}\lower.9\ht0 \hbox to 0pt{\hss$\smile$\hss}\kern\wd0}

\def\ind{\mathop{\mathpalette\Ind{}}}
\def\nind{\mathop{\mathpalette\Notind{}}}

\allowdisplaybreaks %Fix horrible vertical spacing issues.

\DeclareMathOperator{\alt}{{Alt}}
\DeclareMathOperator{\ord}{ord}

\newcommand{\rii}{\textrm{r}}
\newcommand{\lee}{\ell}

\newcommand{\nref}[2]{\hyperref[#1]{\ref*{#1}$_{#2}$}}

\DeclareMathOperator{\Th}{{Th}}
\DeclareMathOperator{\SL}{{SL}}
\DeclareMathOperator{\mlt}{{Mlt}}

\def\forkindep{\mathrel{\raise0.2ex\hbox{\ooalign{\hidewidth$\vert$\hidewidth\cr\raise-0.9ex\hbox{$\smile$}}}}}

\numberwithin{theorem}{section}
\newtheorem{problem}[theorem]{Problem}

\newtheorem{clm}{Claim}
\newtheorem*{clm*}{Claim}

\theoremstyle{definition}

\theoremstyle{remark}

\AtEndEnvironment{proof}{\setcounter{clm}{0}}
\newenvironment{clmproof}[1][\proofname]{\proof[#1]}{\endproof}

\newcommand{\orcidlogo}{\includegraphics[height=\fontcharht\font`\B]{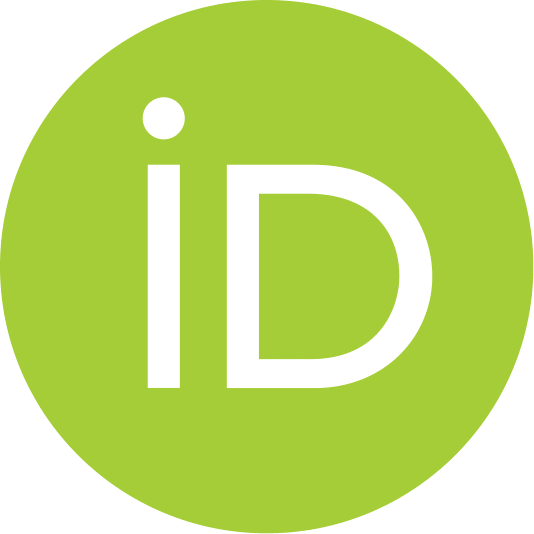}}
\newcommand{\orcid}[1]{\href{#1}{\orcidlogo #1}}

\usepackage[hyphenbreaks]{breakurl}

\hypersetup{
	colorlinks,
	linkcolor={red!50!black},
	citecolor={blue!50!black},
	urlcolor={blue!80!black}
}

\usepackage[
backend=bibtex,
style=alphabetic,
]{biblatex}
\title[Definable convolution and idempotent Keisler measures III]{Definable convolution and idempotent Keisler measures III. Generic stability, generic transitivity, and revised Newelski's conjecture}
\author{Artem Chernikov}
\address{Artem Chernikov \orcid{https://orcid.org/0000-0002-9136-8737} \\ 
University of Maryland, College Park and University of California, Los Angeles, USA}
\email{artem@umd.edu}

\author{Kyle Gannon}
\address{Kyle Gannon \orcid{https://orcid.org/0000-0001-5951-5269}\\ 
Beijing International Center for Mathematical Research (BICMR)\\
Peking University\\
Beijing, China}
\email{kgannon@bicmr.pku.edu.cn}

\author{Krzysztof Krupi\'nski}
\address{Krzysztof Krupi\'{n}ski \orcid{https://orcid.org/0000-0002-2243-4411} \\ 
Instytut Matematyczny Uniwersytet Wroc{\l}awski, pl. Grunwaldzki 2, 50-384 Wroc{\l}aw, Poland}
\email{Krzysztof.Krupinski@math.uni.wroc.pl}

\begin{document}

\maketitle
\begin{abstract}
We study idempotent measures and the structure of the convolution semigroups  of  measures over definable groups.

We isolate the property of \emph{generic transitivity} and demonstrate that it is sufficient to develop stable group theory localizing on a generically transitive type. We establish generic transitivity of generically stable idempotent types in important new cases, including abelian groups in arbitrary theories and arbitrary groups in rosy theories, and characterize  them as generics of connected type-definable subgroups.
	
	Using tools from Keisler's randomization theory, we generalize some of these results from types to generically stable Keisler measures, and classify idempotent generically stable measures in abelian groups as (unique) translation-invariant measures on type-definable fsg subgroups. This provides a partial definable counterpart to the classical work of Rudin, Cohen and Pym for locally compact topological groups.
	
	Finally, we provide an explicit construction of a minimal left ideal in the convolution  semigroup
of measures for an arbitrary countable NIP group, from a minimal left ideal in the corresponding semigroup on types  and a canonical measure constructed on its ideal subgroup. In order to achieve it, we in particular prove the revised Ellis group conjecture of Newelski for countable NIP groups.
\end{abstract}
\vspace{-15pt} 

\tableofcontents
\section{Introduction}

We study idempotent measures and the structure of the convolution semigroups on measures in definable groups, as well as some related questions about topological dynamics of definable actions (continuing \cite{chernikov2022definable, chernikov2023definable}, and entirely resolving some of the questions raised there).

We first recall the classical setting. If $G$ is a locally compact group and $\mathcal{M}(G)$ is the space of regular Borel probability measures on $G$, one  extends group multiplication on $G$ to  \emph{convolution} $\ast$ on $\mathcal{M}(G)$: if $\mu, \nu \in \mathcal{M}(G)$ and $B$ is a Borel subset of $G$, then
\begin{equation*}
    (\mu * \nu) (B) = \int_{G} \int_{G} \mathbf{1}_{B}(x \cdot y) d\mu(x)d\mu(y).  
\end{equation*}
A measure $\mu$ is \emph{idempotent} if $\mu * \mu = \mu$. A classical line of work established a correspondence between compact subgroups of $G$ and idempotent measures in $\mathcal{M}(G)$, in progressively broader contexts \cite{kawada1940probability,Wendel,Cohen,Rudin,Glicksberg2} culminating in the following:
\begin{fact}\cite[Theorem A.4.1]{pym1962idempotent}\label{fac: classical Pym} Let $G$ be a locally compact group and $\mu \in \mathcal{M}(G)$. Then the following are equivalent: 
\begin{enumerate}
    \item $\mu$ is idempotent. 
    \item The support $\supp(\mu)$ of $\mu$ is a compact subgroup of $G$ and $\mu|_{\supp(\mu)}$ is the normalized Haar measure on $\supp(\mu)$. 
\end{enumerate}
\end{fact}

We are interested in a counterpart of this phenomenon in the \emph{definable category}. In the same way as e.g.~algebraic or Lie groups are important in algebraic or
differential geometry, the understanding of groups definable in a given first-order structure (or in certain classes of first-order structures) is important for model theory and its applications. The
class of \emph{stable groups} is at the core of model theory, and the corresponding theory was developed in the 1970s-1980s borrowing many ideas from the
study of algebraic groups over algebraically closed fields (with corresponding
notions of connected components, stabilizers, generics, etc., see~\cite{poizat2001stable}). 
More recently, many of the ideas of stable group theory were extended to the  class of \emph{NIP groups}, which contains both stable groups and groups definable in $o$-minimal structures or over the $p$-adics.
This led to multiple applications, e.g.~a resolution of Pillay's conjecture for compact o-minimal groups
\cite{NIP1} or Hrushovski's work on approximate subgroups \cite{hrushovski2012stable}, and brought to light the importance of the study of invariant measures on definable subsets of the group (see e.g.~\cite{chernikov2018model} for a short survey), as well as the methods  of
topological dynamics (introduced into the picture starting with Newelski \cite{N1}). In particular, deep connections with \emph{tame} dynamical systems as studied by Glasner, Megrelishvili and others (see e.g.~\cite{glasner2007structure, Gla18}) have emerged, and play an important role in the current paper.

More precisely, we now let $G$ be a group definable in some structure $M$ (i.e.~both the
underlying set and multiplication are definable by formulas with parameters
in $M$), it comes equipped with a collection of definable subsets of
cartesian powers of $G$  closed under Boolean combinations, projection
and Cartesian products (but does not carry topology or any additional structure a priori). We let $\cU$ be a ``non-standard'' elementary extension of $M$, and we let $G(\cU)$ denote the group obtained by evaluating in $\cU$ the formulas used to
define $G$ in $M$ (which in the case of an algebraic group
corresponds to working in the universal domain, in the sense of Weil).  So
e.g.~if we start with $M = (\mathbb{R}, +, \times)$ the field of reals, and $G(M)$ its additive group, then
$G(\cU)$ is the additive group of a large real closed field extending $\mathbb{R}$ which
now contains infinitesimals — i.e., it satisfies a saturation condition: every small finitely
consistent family of definable sets has non-empty intersection.
It is classical in topological dynamics to consider the action of a discrete
group $G$ on the compact space $\beta G$ of ultrafilters on $G$, or more precisely ultrafilters on the Boolean algebra of \emph{all} subsets of $G$. In the definable setting, given
a definable group $G(M)$, we let $S_G(M)$ denote the space of ultrafilters on the
Boolean algebra of \emph{definable} subsets of $G(M)$, hence the space $S_G(M)$ (called
the space of types of $G(M)$) is a ``tame'' analogue of the Stone-\v Cech compactification of the discrete group $G$. Then $G(M)$ acts on $S_G(M)$ by homeomorphisms, and the same construction applies to $G(\cU)$ giving the space $S_G(\cU)$
of ultrafilters on the definable subsets of $G(\cU)$.
Similarly, we let $\mathfrak{M}_G(M)$ denote the space of finitely additive probability measures on the Boolean algebra of definable subsets of $G(M)$ (and $\mathfrak{M}_G(\cU)$ for $G(\cU)$), it is affinely homeomorphic to the space of all regular $\sigma$-additive Borel probability measures on $S_G(M)$ (respectively on $S_G(\cU)$), with weak$^*$-topology. The set $G(M)$ embeds into $S_G(\cU)$ as realized types, and we let $S_{G,M}(\cU)$ denote its closure (model theoretically, this corresponds to the set of global types in $S_G(\cU)$ that are \emph{finitely satisfiable} in $G(M)$). Similarly, we let $\mathfrak{M}_{G,M}(\cU)$ denote the closed convex hull of $G(M)$ in $\mathfrak{M}_G(\cU)$ (this is the set of global Keisler measures on $G(\cU)$ finitely satisfiable in $G(M)$, equivalently the set of measures supported on $S_{G,M}(\cU)$ --- see \cite[Proposition 2.11]{chernikov2022definable}).
Similarly to the classical case, in many situations (including the ones discussed in the introduction) we have a well-defined convolution operation $\ast$  on $\mathfrak{M}_{G,M}(\cU)$ (see Definition \ref{def: definable conv} and the discussion around it).

In this context, generalizing a classical fact about idempotent types in  stable groups \cite{N3}, we have the following definable counterpart of Fact \ref{fac: classical Pym} for  \emph{stable} groups:

\begin{fact}\cite[Theorem 5.8]{chernikov2022definable}\label{fac: stable corresp}
Let $G$ be a (type-)definable group in a stable structure $M$ and $\mu \in \mathfrak{M}_{G,M}(\cU)$ a measure. Then $\mu$ is idempotent if and only if $\mu$ is the unique left-invariant (and the unique right-invariant) measure on a type-definable subgroup of $G(\cU)$ (namely, the left-/right-stabilizer of $\mu$).
%	\begin{enumerate}
%		\item Assume that $T$ is NIP and $\mu \in \mathfrak{M}_{G}(\cU)$ is $\Aut(\cU/M)$-invariant, idempotent and $G^{00}$-right-invariant. Then $\Stab(\mu)$ is a type-definable subgroup of $G$ and $\mu$ is supported on it. In particular, $\mu$ is a right-invariant measure on $\Stab(\mu)$ (but possibly not a unique one).
%	\end{enumerate}
\end{fact}
\noindent This suggests a remarkable analogy between the topological and definable settings, even though Fact \ref{fac: stable corresp} is proved using rather different methods.

In the first part of the paper (Sections \ref{sec: generically stable types} and \ref{sec: idemp gen stab measures}), we study generalizations of Fact \ref{fac: stable corresp} beyond the limited context of stable groups (we note that this correspondence fails in general NIP  groups  without an appropriate tameness assumption on the idempotent measure \cite[Example 4.5]{chernikov2023definable}). An important class of groups arising in the work on Pillay's conjectures is that of groups with \emph{finitely satisfiable generics}, or \emph{fsg} groups in short \cite{NIP1}. It contains stable groups, as well as (definably) compact groups in $o$-minimal structures, and provides a natural counterpart to the role that compact groups play in Fact \ref{fac: stable corresp}. By a well-known characterization in the NIP context (see e.g.~\cite[Proposition 8.33]{Guide}), these are precisely the groups that admit a (unique) translation-invariant measure $\mu$ on their definable subsets which is moreover \emph{generically stable}: a sufficiently
long random sample of elements from the group uniformly approximates the measure of
all sets in a definable family of subsets with high probability (i.e.~$\mu$ is a \emph{frequency interpretation measure}, or \fim\ measure, satisfying a \emph{uniform} version of the weak law of large numbers --- this notion is motivated by Vapnik-Chervonenkis theory, and serves as a correct generalization of generic stability for measures outside of NIP, by analogy with generically stable types in the sense of \cite{PiTa}); see Section \ref{sec: fim measures}).
An analog of Fact \ref{fac: stable corresp} would thus amount to demonstrating that such subgroups are the  only source of idempotent generically stable measures (see Problem \ref{conj: main measures}).

First, in Section \ref{sec: generically stable types} we focus on the case of idempotent types in $S_{G,M}(\cU)$ (i.e.~$\{0,1\}$-measures, equivalently~ultrafilters on the Boolean algebra of definable subsets of $G$). After reviewing some preliminaries on generically stable types (Sections \ref{sec: gen stab types} and \ref{def: generically stable group}), we revise the case of groups in stable structures (Section \ref{subsection: stable theories}), and then resolve the question in several important cases:
\begin{theorem}
	Assume $p \in S_{G}(\cU)$ is generically stable and idempotent, and one of the following holds:
	\begin{enumerate}
	\item $p$ is  stable and $M$ is arbitrary (Proposition \ref{proposition: stable type case}, see Section \ref{sec: gen trans stable types});
		\item $G$ is abelian and $M$ is arbitrary (Proposition \ref{prop: key for types in ab}, see Section \ref{sec: abelian for types});
		\item $G$ is arbitrary and $M$ is inp-minimal (Proposition \ref{prop: key for types in dpmin}, see Section \ref{sec: inp-min});
		\item $G$ is arbitrary and $M$ is rosy (so e.g.~if $M$ has a simple theory; Proposition \ref{prop: gen trans in rosy}, see Sections \ref{sec: rosy types} and \ref{sec: gen trans simple}).
	\end{enumerate}
	Then $p$ is the unique left-/right-invariant type on a type-definable subgroup of $G(\cU)$ (namely, the left-/right-stabilizer of $p$).
\end{theorem}

The proof proceeds by establishing the crucial property of \emph{generic transitivity} (see Section \ref{sec: gen trans}) for idempotent generically stable types in these cases, namely that if $(a_1,a_2) \models p \otimes p$, then $(a_1 \cdot a_2, a_1) \models p \otimes p$ (using local weight arguments in cases (2) and (3), and the appropriate version of the theory of stratified ranks in the other cases). The question whether every generically stable idempotent type is generically transitive remains open, even for NIP groups (see Problem \ref{conjecture: main conjecture'} and discussion in Section \ref{sec: gen trans}).

We further investigate generic transitivity, and demonstrate that it is a sufficient condition for developing some crucial results of stable group theory localizing on a generically transitive type (some other elements of stable group theory for generically stable types were considered in \cite{wang2022group}). In particular, in Section \ref{sec: chain conds} we establish a chain condition for groups type-definable using parameters from a Morley sequence of a generically stable type $p$, implying that there is a smallest group of this form --- and it is equal to the stabilizer of $p$ when $p$ is generically transitive (see Corollary \ref{cor: chain cond cor} and Proposition \ref{prop: chain cond loc stab} for the precise statement).

In Section \ref{sec: idemp gen stab measures}, we generalize some of these results from types (i.e.~$\{0,1\}$-measures) to  general measures, in arbitrary structures. Our main result is a definable counterpart of Fact \ref{fac: classical Pym} for \emph{arbitrary} abelian group:

\begin{theorem}\label{thm: abel meas gen trans intro}(Theorem \ref{prop:main})
	Let $G$ be an abelian group and $\mu \in \mathfrak{M}_G(\cU)$ a generically stable measure. Then $\mu$ is idempotent if and only if $\mu$ is the unique left-invariant (and the unique right-invariant) measure on a type-definable subgroup of $G(\cU)$ (namely, its stabilizer).

\end{theorem}
\noindent Groups as in Theorem \ref{thm: abel meas gen trans intro}, i.e.~supporting an invariant generically stable measure, are called {\em fim} groups (see Section \ref{sec: fim groups}), and in the NIP case correspond precisely to fsg groups (but this is potentially a stronger condition in general).
Our proof of Theorem \ref{thm: abel meas gen trans intro} relies on several ingredients of independent interest. First, we develop some theory of \fim\  groups, generalizing from fsg groups in NIP structures (Section \ref{sec: fim groups}). 
Then, in Section \ref{sec: fim meas and randomization}, we provide a characterization  of generically stable measures of independent interest extending \cite{CGH2},  demonstrating that the usual property --- any Morley sequence determines the measure of arbitrary formulas by averaging along it --- holds even when the parameters of these formulas are allowed to be ``random''. More precisely:
\begin{theorem}[Theorem \ref{thm: unif gen stab meas}]Let $\mu \in \mathfrak{M}_{x}(\mathcal{U})$ be \fim,  $\nu \in \mathfrak{M}_{y}(\mathcal{U})$ arbitrary, $\varphi(x,y,z)$ a formula, $b \in \mathcal{U}^{z}$, and let $\mathbf{x} = (x_i)_{i \in \omega}$. Suppose that $\lambda \in \mathfrak{M}_{\mathbf{x}y}(\mathcal{U})$ is arbitrary such that $\lambda|_{\mathbf{x},M}  = \mu^{(\omega)}$ and $\lambda|_{y} = \nu$. Then 
\begin{equation*} 
\lim_{i \to \infty} \lambda(\varphi(x_i,y,b)) = \mu \otimes \nu(\varphi(x,y,b)). 
\end{equation*} Moreover for every $\varepsilon > 0$ there exists $n = n(\mu, \varphi, \varepsilon) \in \mathbb{N}$ so that for any $\nu, \lambda,b$ as above, we have $\lambda(\varphi(x_i,y,b)) \approx^{\varepsilon} \mu \otimes \nu( \varphi(x,y,b))$ for all but $n$ many $i \in \omega$.
\end{theorem} 
This is new even in the NIP case, and relies on the use of Keisler randomization theory. Namely, we use the correspondence between measures in $G(\cU)$ and types in its randomization, viewed as a structure in continuous logic, that was introduced  in \cite{BEN} (and studied further in \cite{CGH2}). It  allows us to imitate in Section \ref{sec: idemp meas in ab proof} the bounded local weight argument from Section \ref{sec: abelian for types} in a purely measure theoretic context, using an adapted version of generic transitivity (see Section \ref{sec: gen trans for meas}) and arguments with pushforwards.

Problem \ref{conjecture: main conjecture'} on whether every generically stable idempotent type is generically transitive generalizes to measures (see Problem \ref{conj: main measures}). In Section \ref{sec: support trans}, we distinguish a weaker property of a measure than being generically transitive, which we call {\em support transitivity}. It leads to a weaker conjecture saying that every generically stable idempotent measure is support transitive (see Problem \ref{intermediate:conjecture}). While this conjecture is open, it trivially holds for idempotent types, and so one can expect that if the techniques used for types in Sections \ref{subsection: stable theories}--\ref{sec: rosy types} could be adapted to measures, they would rather not prove the main conjecture that every generically stable idempotent measure is generically transitive, but reduce it to the above weakening. The idea is to pass to the randomization of the structure in question, and if this randomization happens to have a well-behaved stratified rank, then apply a continuous logic version of the arguments from Sections \ref{subsection: stable theories}--\ref{sec: rosy types}. In Section \ref{sec: stable measures}, we illustrate how it works for stable theories (recall that stability is preserved under randomization). 

In Section \ref{section: top dyn}, instead of considering an individual (idempotent) measure, we study the structure of the (left-continuous compact Hausdorff) semigroup $\left(\mathfrak{M}_{G,M}(\cU), \ast \right)$ of measures on a definable NIP group under convolution, through the lens of Ellis theory. It was demonstrated in \cite[Theorem 5.1] {chernikov2023definable} that the ideal (or Ellis) 
subgroup of any minimal left ideal is always trivial, and that when $G$ is \emph{definably amenable} (i.e.~admits a left-invariant finitely additive probability measure on its definable subsets), then any minimal left ideal itself is trivial, but has infinitely many extreme points when $G$ is not definably amenable. In the general, non-definably amenable case, a description of a minimal left ideal in $\left(\mathfrak{M}_{G,M}(\cU), \ast \right)$  was obtained  under some additional strong assumptions (see \cite[Theorem 6.11]{chernikov2023definable} and discussion at the end of  Section \ref{sec: minimal left ideal of measures}).
Here we obtain a description of a minimal left ideal of $\left(\mathfrak{M}_{G,M}(\cU), \ast \right)$ for an \emph{arbitrary} countable NIP group:
\begin{theorem}[Corollary \ref{cor: min ideal ctbl group}]\label{thm: intro min ideal}
	Assume that $G$ is group definable in a countable NIP structure $M$, and  let $\cM$ be a minimal left ideal in $(S_{G,M}(\mathcal{U}),*)$ and $u \in \cM$ an idempotent. Then the ideal group $u\cM$ carries a canonical invariant Keisler measure $\mu_{u\cM}$ (see Proposition \ref{prop: def of inv meas on ideal group} for the definition), and $\mathfrak{M}(\cM) * \mu_{u\cM}$  is a minimal left ideal of $\left(\mathfrak{M}_{G,M}(\mathcal{U}), \ast \right)$, where $\mathfrak{M}(\cM)$ denotes the space of all measures supported on $\cM$, and  $\mu_{u\cM}$ is an idempotent in $\mathfrak{M}(\cM) * \mu_{u\cM}$.
\end{theorem}

\noindent Theorem \ref{thm: intro min ideal} is deduced using a combination of two results of independent interest that we now discuss.

An important general fact from topological dynamics is that the ideal group $u\cM$  of the Ellis semigroup of any flow is always a compact $T_1$ (\emph{not} necessarily Hausdorff) semi-topological group (i.e.~multiplication is separately continuous) with respect to a canonical topology, the so called \emph{$ \tau$-topology} (which is weaker than the induced topology from the Ellis semigroup). This topology was defined by Ellis and has played an essential role in the most important structural results in abstract topological dynamics, starting from the Furstenberg structure theorem for minimal distal flows (e.g.~see \cite{Aus}) and ending with a recent theorem of Glasner on the structure of tame, metrizable, minimal flows \cite{Gla18}.
%In fact, our proof of the revised Newelski's conjecture for countable $G$ will be deduced  using this theorem of Glasner. 
In model theory, the $\tau$-topology on the ideal groups played a key role in applications to the quotients of definable groups by their model-theoretic connected components (\cite{KrPi1}) and to Lascar strong types and quotients by arbitrary bounded invariant equivalence relations \cite{KPR18,KrRz}. It also partly motivated the work of Hrushovski on definability patterns structures with spectacular applications to additive combinatorics \cite{Hr_pat1,Hr_pat2}. In \cite{KrPi2}, the $\tau$-topology was used to give a shorter and simpler proof of the main result of \cite{Hr_pat2}.
As the key result of Section  \ref{section: top dyn} we demonstrate the following: 
\begin{theorem}\label{thm: intro Borel def}
 (Lemma \ref{lemma: constructibility}) Assume that  $G$ is group definable in an arbitrary  NIP structure $M$, and that the $\tau$-topology on the ideal group $u\cM$  of the $G(M)$-flow $S_{G,M}(\cU)$ is Hausdorff. Then for any clopen subset $C$ of $S_{G}(\cU)$, the subset $C \cap u\cM$ of $u\cM$ is constructible, and so Borel, in the $\tau$-topology. 
\end{theorem}
\noindent It follows that when the $\tau$-topology on $u\cM$ is Hausdorff, the ideal group $u\cM$  is a compact topological group (Corollary \ref{corollary: T2 implies topological}), so we have the unique (normalized)  Haar measure $h_{u\cM}$ on Borel subsets, and by Theorem \ref{thm: intro Borel def} it induces the aforementioned Keisler measure $\mu_{u\cM} \in \mathfrak{M}_{G}(\cU)$ via $\mu_{u\cM}(\varphi(x)):=h_{u\cM}([\varphi(x)] \cap u\cM)$ (Proposition \ref{prop: def of inv meas on ideal group}). This reduces the question to understanding when the $\tau$-topology is Hausdorff --- which is precisely the revised version of Newelski's conjecture (see \cite[Conjecture 5.3]{KrPi2}).

\emph{The Ellis group conjecture} of Newelski \cite{N1} is an important prediction in the study of NIP groups connecting a canonical model-theoretic quotient of a definable group $G(M)$ and a dynamical invariant of its natural action on $S_{G,M}(\cU)$.
Let $G$ be a group definable in a structure $M$, and let $u \cM$ be the ideal (Ellis) group of the $G(M)$-flow $(G(M), S_{G,M}(\cU))$. We let $G^{00}_{M}$ be the smallest type-definable over $M$ subgroup of $G(\cU)$ of bounded index. The \emph{Ellis group conjecture} of Newelski says that the group epimorphism $\theta: u \cM \to G(\cU)/G^{00}_{M}$ given by $\theta(p) := a/G^{00}_{M}$, for a type $p \in u \cM$ and $a$ a realization of $p$, is an isomorphism under suitable tameness assumptions on the ambient theory. This conjecture was established for definably amenable groups definable in $o$-minimal structures in \cite{chernikov2014external}, and for definably amenable groups in arbitrary NIP structures in \cite{CS}. On the other hand, it was refuted for $\SL_2(\mathbb{R})$ in \cite{gismatullin2015some}. Newelski's epimorphism $\theta$ was refined in \cite{KrPi1} to a sequence of epimorphisms
 $$u \cM \to u \cM / H(u \cM) \to G(\cU)/G^{000}_{M} \to G(\cU)/G^{00}_{M},$$
 where $G^{000}_{M}$ is the smallest bounded index subgroup of $G(\cU)$ invariant under the action of $\Aut(\cU/M)$, and $H(u \cM)$ is the subgroup of $u \cM$ given by the intersection of the $\tau$-closures of all $\tau$-neighborhoods of $u$.
With this refinement,  Newelski's conjecture fails when $G^{000}_M \neq G^{00}_{M}$, in which case also $u\cM/H(u \cM) \to G(\cU)/G^{000}_M$ is not an isomorphism. The first such example, the universal cover  $\widetilde{\SL_2(\mathbb{R})}$ of $\SL_2(\mathbb{R})$, was found in \cite{conversano2012connected}, and further examples were given in \cite{gismatullin2015model}. On the other hand, no examples of NIP groups with non-trivial $H(u \cM)$  (equivalently, with $u\cM$ not Hausdorff in the $\tau$-topology) were known. This motivated the following weakening of Newelski's conjecture: 
\begin{conjecture}\label{conjecture: revised Newelski's conjecture}\cite[Conjecture 5.3]{KrPi2}
If $M$ is NIP, then the $\tau$-topology on $u\cM$  is Hausdorff.
\end{conjecture}

\noindent It clearly holds whenever $u\mathcal{M}$  is finite. It is also known to hold for definably amenable groups in NIP theories, as the full Newelski's conjecture holds in this context.  Besides those two general situations,  it was confirmed only for $\widetilde{\SL_2(\mathbb{R})}$ (we refer to \cite[Section 5]{KrPi2} for a proof and a more detailed discussion).

In Section \ref{sec: rev Newelski conj} we establish the revised Newelski's conjecture for \emph{countable} NIP groups:
\begin{theorem}[Theorem \ref{thm: revised New conj}]
	The revised Newelski's conjecture holds when $G$ is a definable group in a countable NIP structure.
\end{theorem}

\noindent  This relies on the fundamental theorem of Glasner describing the structure of minimal tame metrizable flows  \cite{Gla18} and a presentation of the $G(M)$-flow $S_{G,M}(\cU)$ as the inverse limit of all $S_{G,\Delta}(M)$ (the Stone space of the $G(M)$-algebra generated by the finite set $\Delta$), where $\Delta$ ranges over all finite collections of externally definable subsets of $G(M)$.

\section{Idempotent generically stable types}\label{sec: generically stable types}

Throughout this section, we let $T$ be a complete theory, $M \models T$, and $\cU \succ M$ a monster model.

\subsection{Generically stable types}\label{sec: gen stab types}

\noindent We will need some basic facts about generically stable types. As usual, given a global type $p \in S_x(\cU)$ (automorphism-)invariant over a small set $A \subseteq \cU$ and  $\cU' \succ\cU $ a bigger monster model with respect to $\cU$, we let $p |_{\cU'}$ be the unique extension of $p$ to a type in $S_x(\cU')$ which is invariant over $A$. Given another $A$-invariant type $q \in S_y(\cU)$, we define the $A$-invariant type $p \otimes q \in S_{xy}(\cU)$ via $p \otimes q := \tp(ab/\cU)$ for some/any $a,b$ in $\cU'$ such that $b \models q$ and $a \models p|_{\cU b}$. Given an arbitrary linear order $(I,<)$, a sequence $\bar{a} = (a_i : i \in I)$ in $\cU$ is a \emph{Morley sequence} in $p$ over $A$ if $a_i \models p|_{A a_{<i}}$ for all $i \in I$. Then the sequence $\bar{a}$ is indiscernible over $A$, and for any other Morley sequence $\bar{a}' = (a'_i:i\in I)$ in $p$ over $A$ we have $\tp(\bar{a}/A) = \tp(\bar{a}'/A)$.  We can then define a global $A$-invariant type $p^{(I)}((x_i : i \in I)) \in S_{\bar{x}}(\cU)$ as
$$\bigcup \left\{ \tp(\bar{a}/B) : A \subseteq B \subseteq \cU \textrm{ small}, \bar{a} = (a_i : i \in I) \textrm{ a Morley sequence in }p \textrm{ over }B \right\}.$$
Equivalently, $p^{(I)} = \tp(\bar{a}/\cU)$ for $\bar{a} = (a_i : i \in I)$ a Morley sequence in $p|_{\cU'}$ over $\cU$. For any $k \in \omega$ (viewed as an ordinal) we have $p^{(k)}(x_1, \ldots, x_k) = p(x_k) \otimes \ldots \otimes p(x_1)$.

We do not  assume NIP unless explicitly stated, and use the standard definition from  \cite{PiTa}: a global type $p \in S_x(\cU)$ is \emph{generically stable} if it is $A$-invariant for some small $A \subset \cU$, and for any ordinal $\alpha$ (or just for $\alpha = \omega + \omega$), $(a_i : i \in \alpha)$ a Morley sequence in $p$ over $A$ and formula $\varphi(x) \in \cL(\cU)$, the set $\{i \in \alpha : \models \varphi(a_i) \}$ is either finite or co-finite. 

In the following fact, items (1)--(4) can be found in \cite[Section 9]{casanovas2011more}, (5)  in \cite[Theorem 2.4]{garcia2013generic}, (6) is an immediate consequence of stationarity in (2), and (7) is \cite[Proposition 2.1]{PiTa}. We let $\ind$ denote forking independence. We will freely  use some of the basic properties of forking in arbitrary theories, e.g.~extension and left transitivity, see \cite[Section 2]{chernikov2012forking} for a reference.

%We write $a \ind_c b$ to denote that $\tp(a/bc)$ does not fork over $c$. We recall some properties of forking in arbitrary theories (see  e.g.~\cite[Section 2]{chernikov2012forking}).
%\begin{fact}\label{fac: basic forking}
%	The following properties of  forking hold in an arbitrary theory (below $a,b,c$, etc.~are  arbitrary small tuples in $\cU$).
%	\begin{enumerate}
%		\item \label{it: bf right ext} Right extension: if $a \ind_c b$ then  for any $d$ there exists  some  $d' \equiv_{bc} d$ such that $a \ind_{c} bd'$.
%		\item Monotonicity: if $aa' \ind_{c} bb'$ then $a \ind_c b$.
%		\item \label{it: bf bas mon}Base monotonicity: if $a \ind_{c} bb'$ then $a \ind_{cb'} b$.
%		\item Left transitivity: if $a \ind_{cc'} b$ and $c' \ind_c b$ then $ac' \ind_{c} b$.
%		\item  \label{it: bf inv} If $p \in S_x(\cU)$ is invariant over a small model  $M \prec \cU$, then $p$ does not fork over $M$.
%		\item \label{it: fork iff acl} If $a \ind_c b$  then $\dcl(ac) \ind_c b$.
%	\end{enumerate}
%\end{fact}

\begin{fact}\label{fac: basic gen stab}
Let $p \in S_x(\cU)$ be a generically stable type, invariant over a small subset $A \subseteq \cU$. Then the following hold.
\begin{enumerate}
\item Every realization of $p^{(\omega)}{|_{A}}$ is a totally indiscernible sequence over $A$.
	\item The type $p$ is the unique global non-forking extension of $p|_A$.
	\item For any $a \models p|_A$ and $b$ in $\cU$ such that $\tp(b/A)$ does not fork over $A$, we have $a \ind_A b \iff b \ind_A a$ (this holds for any $b$ when $A$ is an extension base, e.g.~when $A \prec \cU$).
	\item In particular, if $a,b \models p|_{A}$, then $a \ind_A b \iff (a,b) \models p^{(2)}|_{A} \iff (b,a) \models p^{(2)}|_{A}$.  
	\item If $A$ is an extension base, $(a_i)_{i<\omega} \models p^{(\omega)}|_A$ and $\varphi(x,a_0)$ (where $\varphi(x,y) \in \mathcal{L}(A)$) forks/divides over $A$, then $\{\varphi(x,a_i) : i < \omega\}$ is inconsistent.
	\item Let  $a \models p|_A$ and let $b,c$ be arbitrary small tuples in $\cU$.  If $a \ind_A b$ and $a \ind_{Ab} c$, then $a \ind_{A} bc$;
	\item $p$ is definable over $A$.
\end{enumerate}
\end{fact}

\begin{remark}
	By Fact 2.1(1), for a global generically stable type $p$ invariant over $A$, we will also be using (without mentioning) an equivalent definition of a Morley sequence with the reversed order. That is, given a linear order $(I,<)$, we might say that $(a_i : i \in I^*)$ is a Morley sequence in $p$ over $A$ if $a_i \models p|_{A a_{<i}}$ for all $i \in I$, where $I^*$ is $I$ with the reversed ordering. So e.g.~we will refer to $(a_k, a_{k-1}, \ldots, a_1)$ with $a_i \models p|_{A a_1 \ldots a_{i-1}}$ as a Morley sequence in $p$ over $A$, and write $(a_k, \ldots, a_1) \models p^{(k)}$. We will frequently use this without further mention.
\end{remark}

\begin{fact} \cite[Proposition 1.2]{dobrowolski2012omega} \label{fac: dcl of gs is gs}  Let $p \in S_{x}(\cU)$ be a generically stable type, invariant over a small set of parameters $A \subseteq \mathcal{U}$. Suppose that $ \cU \prec \cU'$, $a$ is an element of $\cU'$ such that $a \models p$, and  $b \in \dcl(a, A)$. Then $\tp(b/\mathcal{U})$ is generically stable over $A$. 
\end{fact}

\subsection{Setting} \label{sec: setting types}
 Let $G = G(x)$ be an $\emptyset$-type-definable group (in the sort of $\cU$ corresponding to the tuple of variables $x$) and $\bar G:=G(\cU)$. By $\cdot$ we mean an $\emptyset$-definable function (from $\left(\cU^x\right)^2$ to $\cU^x$) whose restriction to $G$ is the group operation on $G$. Similarly, by $^{-1}$ we mean an $\emptyset$-definable function (from $\cU^x$ to $\cU^x$) whose restriction to $G$ is the inverse in $G$.  By compactness, we can fix a formula $\varphi_0(x) \in \cL$ implied by the partial type $G(x)$ such that: $\cdot$ is defined and associative on $\varphi_0(\cU)$; $a\cdot e=a=e\cdot a$ and $a \cdot a^{-1} = a^{-1} \cdot a = e$ for all $a \in \varphi_0(\cU)$; if $b_1 \neq b_2$ then $a \cdot b_1 \neq a \cdot b_2, b_1 \cdot a \neq b_2 \cdot a$ for all $a,b_1,b_2 \in \varphi_0(\cU)$   (but $\varphi_0(\cU)$ is not necessarily closed under $\cdot$). As usual, for $A \subseteq \cU$, $S_{G}(A)$ denotes the set of types $p \in S(A)$ \emph{concentrated on $G$}, i.e.~such that $p(x) \vdash G(x)$.
 
 Given $p,q \in S_G(\cU)$ global $M$-invariant types ($M \prec \cU$), we define $p \ast q \in S_G(\cU)$ via $p \ast q =\tp(a \cdot b/\mathcal{U})$ for $(a,b) \models p \otimes q$. Together with this operation, the set of all global $M$-invariant types in $S_G(\cU)$ forms a left-continuous semigroup. We say that an invariant type $p \in S_G(\cU)$ is \emph{idempotent} if $p \ast p = p$.
 
%Let $G$ be a $0$-definable group definable in a model $M$. Let $\C \succ M$ be a monster model, and $\bar G:=G(\C)$. More generally, we will consider any 0-type-definable group $\bar G$ in $\C$. By $\cdot$ we mean a $0$-definable function on a superset of $\bar G$ whose restriction to $\bar G$ is the given group operation on $\bar G$. In the type-definable case, choose a formula $\varphi_0(x)$ implied by the type $G(x)$ such that $\cdot$ is defined and associative on $\varphi_0(\C)$ and $a\cdot e=a=e\cdot a$ for all $a \in \varphi_0(\C)$ (but $\varphi_0(\C)$ is not necessarily closed under $\cdot$).

\subsection{Generically stable groups}\label{sec: gen stab groups}

\begin{definition} \cite[Definition 2.1]{PiTa}\label{def: generically stable group}
	A type-definable group $G(x)$ is \emph{generically stable} if there is a generically stable $p \in S_{G}(\cU)$ which is left $G(\cU)$-invariant (we might use ``$G(\cU)$-invariant'' and ``$G$-invariant'' interchangeably when talking about global types).
	\end{definition}
\begin{fact}\label{fac: basic fsg}\cite[Lemma 2.1]{PiTa}
Suppose that $G$ is a generically stable type-definable group in an arbitrary theory, witnessed by a generically stable type $p \in S_{G}(\cU)$. Then we have:
	\begin{enumerate}
		\item $p$ is the unique left $G(\cU)$-invariant and also the unique right $G(\cU)$-invariant type;
		\item $p = p^{-1}$ (where $p^{-1} := \tp(g^{-1}/\cU)$ for some/any $g \models p$ in a bigger monster model $\cU' \succ \cU$).
	\end{enumerate}
\end{fact}

\noindent By Fact \ref{fac: basic fsg} and its symmetric version, we get:
\begin{corollary}
A type-definable group $G(x)$ is generically stable  if and only if there is a generically stable $p \in S_{G}(\cU)$ which is right $G(\cU)$-invariant.
\end{corollary}

\subsection{Idempotent generically stable types and generic transitivity: main conjecture}\label{sec: gen trans}

 Let $p \in S_G(\cU)$ be a generically stable type over $M$. The following is standard:

\begin{proposition}\label{proposition: intersection of relatively definable groups}
The left stabilizer $\Stab(p)$ of $p$ is an intersection of relatively $M$-definable subgroups of $\bar G$; in particular, it is $M$-type-definable.
%$\stab(p)$ is type-definable, in fact  an intersection of $M$-definable [resp. relatively $M$-definable] subgroups of $\bar G$.
\end{proposition}

\begin{proof}
By compactness, we can find a formula $\varphi_1(x) \in \mathcal{L}$ implied by $G(x)$ such that $\varphi_1(\bar G) \cdot \bar{G} \subseteq \varphi_0(\bar{G})$.

Since $p$ is generically stable over $M$, it is definable over $M$ (by Fact \ref{fac: basic gen stab}(7)). For any formula $\varphi(x,y) \in \cL$ which implies $\varphi_1(x)$ (and $y$ is an arbitrary tuple of variables) let 
%$$d_p\varphi:= \{ (h,\widetilde{A}) \in \bar G \times \C^{|\widetilde{A}|}: \varphi(hx,\bar a) \in p\}.$$ 
$$d_p\varphi:= \left\{ (h, a) \in \varphi_0(\cU) \times \cU^{y}: \varphi(h \cdot x, a) \in p\right\}.$$ 
By the definability of $p$ over $M$, $d_p \varphi$ is definable over $M$. Then 
$$\Stab_\varphi(p):=\{ g \in \bar G: g \cdot d_p \varphi =d_p \varphi\},$$
where $g \cdot (h, a):=(h \cdot g^{-1}, a)$, is an $M$-relatively definable subgroup of $\bar G$ (the fact that it is a subgroup of $G$ follows from the observation that it is the stabilizer of the set $d_p \varphi$ under the left action of $G$ on $\varphi_0(\cU) \cdot \bar G \times \cU^{y}$ given by $g \cdot (h,a):=(h \cdot g^{-1},\ a)$, which uses the choice of $\varphi_0(x)$). By the choice of $\varphi_1(x)$, we get that $\Stab(p) = \bigcap_{\varphi(x,y) \in \mathcal{L}, \varphi(x,y) \vdash \varphi_1(x)} \Stab_{\varphi}(p)$ is an intersection of relatively $M$-definable subgroups of $\bar G$.
\end{proof}
\begin{example}
	Let $G'$ be an arbitrary type-definable subgroup of $G$ which is generically stable, witnessed by a generically stable left or right $G'$-invariant type $p \in S_{G'}(\cU)$. Then $p$ is obviously idempotent.
\end{example}

Our central question in the case of types is whether this is the only source of generically stable idempotent types:
%\begin{definition}
%	For the rest of the section, we let $H:=\Stab(p)$ be the left stabilizer of $p$.
%\end{definition}

\begin{definition}
	For the rest of the section, we let $H_{\lee}:=\Stab_{\lee}(p)$  and $H_{\rii}:=\Stab_{\rii}(p)$ be the left and the right stabilizer of $p$, respectively.
\end{definition}

%\begin{problem}\label{conjecture: main conjecture}
%Assume that $p$ is generically stable and idempotent. Is it true that $p \in S_H(\cU)$ and the group $H$ is generically stable (Definition \ref{def: generically stable group})? 
%\end{problem}

\begin{problem}\label{conjecture: main conjecture}
Assume that $p$ is generically stable and idempotent, and let $H=H_{\lee}$ or $H=H_{\rii}$.  Is it true that $p \in S_H(\cU)$ and the group $H$ is generically stable (Definition \ref{def: generically stable group})? 
\end{problem}

We will note now that the second part follows from the first, and that the left and the right versions of the problem are equivalent.

\begin{remark}\label{rem: fsg follows}
Assume $p$ is generically stable and $p \in S_{H}(\cU)$, where $H=H_{\lee}$ or $H=H_{\rii}$. Then:
\begin{enumerate}
	\item $H$ is a generically stable group, witnessed by $p$  (hence $p$ is both the unique left-invariant and the unique right-invariant type of $H$, by Fact \ref{fac: basic fsg});
	\item $H$ is the smallest among all type-definable subgroups $H'$ of $G$ with $p \in S_{H'}(\cU)$;
	\item $H$ is both the left and the right stabilizer of $p$ in $G$.
\end{enumerate}
\end{remark}

\begin{proof}
We will do the case of $H=H_{\lee}$. The proof for $H=H_{\rii}$ is symmetric.

	(1) As then $p$ is a left $H$-invariant generically stable type in $S_H(\cU)$.
	
	(2) $H$ is type-definable by Proposition \ref{proposition: intersection of relatively definable groups}. For any type-definable $H' \leq G(\cU)$ with $p(x)\vdash H'(x)$, the group $H'' := H \cap H'$ is type-definable with $p(x) \vdash H''(x)$ and $H'' \leq H$. If the index is $\geq 2$, we have some $g \in H$ with $H'' \cap g \cdot H'' = \emptyset$, and $p(x) \vdash H''(x), \left(g\cdot p\right)(x) \vdash \left(g\cdot H'' \right)(x)$, so $p \neq g \cdot p$ --- a contradiction. So $H'' = H$, and $H \subseteq H'$.
	
	(3) $H_r$ is type-definable by a symmetric argument as in Proposition \ref{proposition: intersection of relatively definable groups}. By (1), $p$ is right $H$-invariant, so we have $H \subseteq H_r$, and so $p \in S_{H_r}(\cU)$. By the right version of (2) (which is obtained by a symmetric argument as in (2)), we conclude that $H = H_r$.
\end{proof}

By Remark \ref{rem: fsg follows}, we see that $p \in S_{H_{\lee}}(\cU)$ if and only if $p\in S_{H_{\rii}}(\cU)$, so the left and the right versions of Problem \ref{conjecture: main conjecture} are indeed equivalent.

%\begin{proposition}
%	Assume that $p \in S_{G}(\cU)$ is generically stable and $p(x) \vdash \Stab(p)$. Then $\Stab(p)$ is the smallest type-definable subgroup $H$ of $G$ so that $p \vdash H$.
%	\end{proposition}
%\begin{proof} 
%	
%	Then $\Stab(p)$ is type-definable over $A$ (by Fact \ref{fac: basic gen stab}).
%	For any type-definable (with parameters anywhere in $\cU$) $H \leq G(\cU)$ with $p(x)\vdash H(x)$, the group $H' := H \cap \Stab(p)$ is type-definable with $p(x) \vdash H'(x)$ and $H' \leq \Stab(p)$. If the index is $\geq 2$, we have some $g \in \Stab(p)$ with $H' \cap g \cdot H' = \emptyset$ and $p \vdash H', g\cdot p \vdash g\cdot H'$, so $p \neq g \cdot p$ --- a contradiction. So $H' = \Stab(p)$, and $\Stab(p) \subseteq H$.
%\end{proof}

%Let $p \in S_G(\cU)$ be a generically stable type. Let $\cU' \succ\cU $ be a monster model with respect to $\cU$ and $p': =p |_{\cU'}$ (the unique extension of $p$ to a type in $S_G(\cU')$ which is invariant over $M$). Then still $p'$ is definable over $M$ and $H(\cU')=\Stab(p')$. Pick an arbitrary $a \in p(\cU')$. 

Let $p \in S_G(\cU)$ be a generically stable type, and let $\cU' \succ\cU $ be a monster model with respect to $\cU$ and $p': =p |_{\cU'}$ (the unique extension of $p$ to a type in $S_G(\cU')$ which is invariant over $M$). Then still $p'$ is definable over $M$ and $H_{\lee}(\cU')=\Stab_{\lee}(p')$,  $H_{\rii}(\cU')=\Stab_{\rii}(p')$. Pick an arbitrary $a \in p(\cU')$. 
\begin{remark}\label{remark: basic}
The following conditions are equivalent for a generically stable type $p$:
\begin{enumerate}
\item $p \in S_{H_{\lee}}(\cU)$;
\item $a \in \Stab_{\lee}(p')$;
\item for any (equivalently, some) $(a_0,a_1) \models p^{(2)}$, $(a_0 \cdot a_1,a_0) \models p^{(2)}$;
\item  $p \in S_{H_{\rii}}(\cU)$;
\item $a \in \Stab_{\rii}(p')$;
\item for any (equivalently, some) $(a_0,a_1) \models p^{(2)}$, $(a_1 \cdot a_0,a_0) \models p^{(2)}$.
\end{enumerate}
\end{remark}

\begin{proof}
$(1) \Rightarrow (2)$. This is because $a \in p(\cU') \subseteq H_{\lee}(\cU') =\Stab_{\lee}(p')$, where the inclusion follows by (1).

$(2) \Rightarrow (3)$. Pick $b \models p'$. Then $(a,b) \models p^{(2)}$ and $a \cdot b \models p'$ by (2). So $(a \cdot b,a) \models p^{(2)}$.

$(3) \Rightarrow (1)$. Suppose (1) fails. Then $a \notin H_{\lee}(\cU') = \Stab_{\lee}(p')$. So for any $b \models p'$ we have that $a \cdot b$ does not realize $p'$. Using Fact \ref{fac: basic gen stab}, this implies that $a \cdot b$ does not realize $p|_{\cU a}$, because otherwise, %$ab \forkindep_{\C} \C, a$ and since $ab \forkindep_{\C, a} \C'$, we would get (by transitivity of forking independence, using that  $\tp(ab/C')$is generically stable over ???) $ab \forkindep_{\C} \C'$ which means that $\tp(ab/\C')$ is the unique non-forking extension of $p$ which is $p'$, a contradiction. Thus, $(ab,a)$ does not realize $p^{(2)}$ which contradicts (3).
since $a \cdot b \forkindep_{\cU a} \cU'$ (which follows from $b \forkindep_{\cU a} \cU'$ and left transitivity of forking), $\tp(a \cdot b/\cU')$ is the unique non-forking extension of $p|_{\cU a}$ which is exactly $p'$, a contradiction. Thus, $(a \cdot b,a)$ does not realize $p^{(2)}$ which contradicts (3).

$(4) \Rightarrow (5) \Rightarrow (6) \Rightarrow (4)$ are obtained by symmetric arguments to the above ones.

$(1) \Leftrightarrow (4)$ follows from Remark \ref{rem: fsg follows}.
\end{proof}

\begin{definition}\label{def: gen trans type}
	We will say that a generically stable type $p \in S_{G}(\cU)$ is \emph{generically transitive} if it satisfies any of the equivalent conditions in Remark \ref{remark: basic}.
\end{definition}
\begin{remark}
	We have chosen this terminology to highlight the connection of condition (3) in Remark \ref{remark: basic}(3) and the generic transitivity assumption in Hrushovski's group chunk theorem (see e.g.~\cite[Section 5.1]{bays2018geometric}).
\end{remark}
In view of Remark \ref{remark: basic}, our main Problem \ref{conjecture: main conjecture} is equivalent to the following:
\begin{problem}\label{conjecture: main conjecture'}
Assume that $p$ is generically stable and idempotent. Is it then generically transitive? We note that the converse is trivially true. 
\end{problem}

In the following sections, we will provide a positive solution under some additional assumptions on the group. We will also see that elements of stable group theory can be developed in an arbitrary theory localizing on a generically transitive type $p$, if said type $p$ is generically transitive. 

%%%%%%%%%%%s
Let $p \in S_G(\cU)$ be a generically stable type over $M$. By Proposition \ref{proposition: intersection of relatively definable groups}, we know that $\Stab_{\lee}(p)$ is an intersection of relatively $M$-definable subgroups of $\bar G$.

\begin{question}\label{question: stab is intersection}
Assuming that $p$ is generically stable and idempotent, is $\Stab_{\lee}(p)$ an intersection of $M$-definable groups?
\end{question}

An earlier draft of this article provided a positive answer to the question above for generically transitive types. It was pointed out to us by Anand Pillay that a more general  statement was already proved by Hrushovski and Rideau-Kikuchi in \cite{hrushovski2019valued}. 

\begin{proposition}\label{prop: type-def grp is intersec of def grps}
If $G$ is type-definable and $p \in S_{G}(\cU)$ is generically stable, idempotent and generically transitive, then the answer to Question \ref{question: stab is intersection} is positive.
\end{proposition}

\begin{proof} This follows directly from \cite[Proposition 3.4]{hrushovski2019valued}. 
\end{proof}
%%%%%%%%%%%%

We conclude this section with a couple of additional observations.

\begin{lemma}
	If $p \in S_G(\cU)$ is generically stable over $M$ and idempotent, then the type $p^{-1} := \tp(a^{-1}/\cU)$ for some/any $a \models p$ is also generically stable over $M$ and idempotent.
\end{lemma}
\begin{proof}
Assume the hypothesis. Then $p^{-1}$ is generically stable over $M$ by Fact \ref{fac: dcl of gs is gs}.  It follows that $p^{-1} * p^{-1}$ is definable over $M$ (see e.g.~\cite[Proposition 3.15]{chernikov2022definable}). Since $p^{-1}$ is the unique extension of $p^{-1}|_{M}$ definable over  $M$, it suffices to show that $(p^{-1} * p^{-1})|_M = p^{-1}|_{M}$. Let $b_1 \models p^{-1}|_M$ and $b_2 \models p^{-1}|_{M b_1}$. Clearly $b_1^{-1} \models p|_M$ and $b_2^{-1} \models p|_{Mb_1}$ and so $(b_{2}^{-1}, b_1^{-1}) \models p^{(2)}|_M$. Since $p$ is idempotent, we conclude that $b_{1}^{-1} \cdot b_2^{-1} \models p|_{M}$ and in particular $(b_2 \cdot b_1) \models p^{-1}|_{M}$. Hence $(p^{-1} * p^{-1})|_{M} = p^{-1}|_{M}$. 
\end{proof}

\begin{remark}
If $p$ is generically transitive, then $p = p^{-1}$. 

\noindent Indeed,  by Remarks \ref{remark: basic} and \ref{rem: fsg follows} it follows that if $p$ is generically transitive, then $p \in S_{H_{\lee}}(\cU)$ witnesses that $H_{\lee}$ is generically stable, hence $p = p^{-1}$ by Fact \ref{fac: basic fsg}.
\end{remark}

\subsection{Abelian groups}\label{sec: abelian for types}

In this section we give a positive answer to Problem \ref{conjecture: main conjecture} in the case of abelian groups (in arbitrary theories):

\begin{proposition}\label{prop: key for types in ab}
	Assume that $G$ is an abelian group and  $p \in S_G(\cU)$ is a generically stable idempotent type. Then $p$ is generically transitive.
\end{proposition}
\begin{proof}
 Let $M \prec \cU$ be any small model such that $p$ is $M$-invariant.
	Let $(a_1, a_0) \models p^{(2)}|_{M}$ be given, and assume towards contradiction that $(a_1 \cdot a_0, a_0) \not \models p^{(2)}|_{M}$, then $a_1 \cdot a_0 \nind_M a_0$ by Fact \ref{fac: basic gen stab}(4). Let $\varphi(x,y) \in \mathcal{L}(M)$ be such that $\models \varphi(a_1 \cdot a_0,a_0)$ and $\varphi(x,a_0)$ forks over $M$. We extend $(a_1,a_0)$ to a Morley sequence $(a_i)_{i< \omega} \models p^{(\omega)}|_{M}$.
	For each $k < \omega$, let $b_k := a_{k-1} \cdot a_{k-2} \cdot \ldots \cdot a_0$. Then we have:
	\begin{enumerate}
		\item $(a_{k-1}, \ldots, a_1) \models p^{(k-1)}|_{a_0 M}$, hence by idempotence of $p$ we have $a_{k-1} \cdot \ldots \cdot a_1 \models p|_{a_0 M}$, and so $(a_{k-1} \cdot \ldots \cdot a_1, a_0) \models p^{(2)}|_{M}$, and $b_k = (a_{k-1} \cdot \ldots \cdot a_1) \cdot a_0$, so $(b_k, a_0) \equiv_{M} (a_1 \cdot a_0, a_0)$, in particular $\models \varphi(b_k, a_0)$;
		\item for any permutation $\sigma$ of $\{0,1, \ldots, k-1 \}$, we have $(a_{\sigma(0)}, \ldots, a_{\sigma(k-1)}) \models p^{(k)}|_{M}$ (by Fact \ref{fac: basic gen stab}(1));
		\item in particular, for every $i < k$ we have 
		$$(a_{k-1}, a_{k-2}, \ldots, a_0) \equiv_{M}(a_{k-1}, \ldots, a_{i+1}, a_{i-1}, \ldots, a_0,a_i);$$
		\item and $b_k = a_{k-1} \cdot \ldots \cdot a_{i+1} \cdot a_{i-1} \cdot \ldots \cdot a_0 \cdot a_i$ (as $\mathcal{G}$ is abelian);
		\item hence $(b_k, a_0) \equiv_{M} (b_k, a_i)$ for every $i < k$.
	\end{enumerate}
	Thus, by (1) and (5), for every $k < \omega$ we have $b_k \models \{ \varphi(x,a_i) : i <k\}$, and by compactness the set $\{ \varphi(x,a_i) : i < \omega\}$ is consistent. But this contradicts the choice of $\varphi$ by Fact \ref{fac: basic gen stab}(5).
	Hence $(a_1 \cdot a_0, a_0) \models p^{(2)}|_{M}$, and since $M$ was an arbitrary small model (over which $p$ is invariant), we conclude the proof.		
\end{proof}

\begin{remark}
	It was pointed out to us by Martin Hils that this argument is related to \cite[Lemma 5.1]{hrushovski2019valued}, which is used there to find idempotent types in abelian groups of finite $p$-weight.
\end{remark}

\begin{remark}\label{rem: two sided arb grp}
	In the case of an arbitrary group, the proof of Proposition \ref{prop: key for types in ab} gives the following:
		\begin{itemize}
		\item If $p \in S_{G}(\cU)$ is invariant and idempotent, then for any $(a_2, a_1, a_0) \models p^{(3)}$ we have $(a_2 \cdot a_1 \cdot a_0, a_1) \models p^{(2)}$.
	\end{itemize}
	
	\noindent Indeed, this time we assume towards a contradiction that $a_2 \cdot a_1 \cdot a_0 \nind_{M} a_1$. We extend $(a_2, a_1, a_0)$ to a Morley sequence $(a_i : i \in \omega)$ in $p$ over $M$. Let $k \in \omega$ be arbitrary, and let $b_{k}:= a_k \cdot \ldots \cdot a_0$. 
	Then we get $(b_k, a_i) \equiv_{M} (a_{2} \cdot a_1 \cdot a_0, a_1)$ for all $2 \leq i \leq k-1$, which gives a contradiction as in the proof of Proposition \ref{prop: key for types in ab}. To see this, note that since $p$ is idempotent, $a_0':=a_{i-1} \cdot \ldots \cdot a_0 \models p|_M$, $a_i \models p|_{Ma_0'}$, and $a_2':=a_k \cdot \ldots \cdot a_{i+1} \models p|_{Ma_0'a_i}$. Hence $(a_2',a_i,a_0') \models p^{(3)}|_{M}$, and so $(b_k,a_i) \equiv_M (a_2 \cdot a_1 \cdot a_0,a_1)$.
	%	
%	 note that we have $b_k = a_{k} \cdot \ldots  \cdot a_{i+1} \cdot a_i \cdot a_{i-1} \cdot \ldots \cdot a_0$, so $b_k \equiv_{M a_i} a_2 \cdot a_i \cdot a_0$. 
%		By generic stability we have 
%	$$(a_{k}, \ldots, a_{i+1}, a_i, a_{i-1}, \ldots, a_0) \equiv_{M} (a_{k}, \ldots, a_{i+1}, a_{i-1}, \ldots, a_0, a_i),$$
%	so by idempotence $(a_{k} \cdot \ldots \cdot a_{i+1}, a_{i-1} \cdot \ldots \cdot a_0, a_i) \models p^{(3)}|_{M}$, so 
%	$$(a_{k} \cdot \ldots \cdot a_{i+1}, a_{i-1} \cdot \ldots \cdot a_0, a_i) \equiv_{M} (a_2, a_0, a_1),$$
%	so
%	$$(b_k, a_i) \equiv_{M} (a_2 \cdot a_1 \cdot a_0, a_1).$$
\end{remark}

\begin{remark}
	Note that in Remark \ref{rem: two sided arb grp} we only assumed that the idempotent type $p$ is invariant. However, the assumption of generic stability is necessary in Proposition \ref{prop: key for types in ab}. Indeed, let $M := (\mathbb{R}, +, <)$, $G(M) := (\mathbb{R}, +)$, and let $p_{0^+} \in S_{G}(\cU)$ be the unique global definable (over $\mathbb{R}$) type extending $\{x < a: a \in \cU, a > 0 \} \cup \{x > a: a \in \cU, a \leq 0\}$. Then $p_{0^+}$ is idempotent (see \cite[Example 4.5(1)]{chernikov2023definable}). But if $(a_1, a_0) \models p_{0^+} \otimes p_{0^+}$ we have $0 < a_1 < a_0 < \cU$, hence $a_1 + a_0 > a_0$, so $(a_1+a_0, a_0) \not \models p_{0^+} \otimes p_{0^+}$.
\end{remark}

\subsection{Inp-minimal groups}\label{sec: inp-min}

A similar argument with (local) weight applies to arbitrary inp-minimal groups. Recall that a type-definable group $G$ is \emph{inp-minimal} if $\bdn(G(x)) \leq 1$ (we refer to \cite{adler2007strong} and \cite[Section 2]{chernikov2014theories} for the definition and basic properties of \emph{burden}). In particular, every dp-minimal group is inp-minimal. Note that there exist  dp-minimal groups which are not
virtually abelian \cite{simonetta2003example}. Answering  \cite[Problem 5.9]{chernikov2014external}, it was recently proved in \cite{stonestrom2023torsion} and \cite{wagner2024dp} that all dp-minimal groups are virtually nilpotent. 

We will use the following fact:
\begin{fact}\cite[Theorem 2.9]{garcia2013generic}\label{fac: burden gen stab}
	Let $p(x) \in S(M)$ be generically stable over a small set $A$ and $(a_{i} : i < \kappa ) \models p^{(\kappa)}|_{A}$ for some cardinal $\kappa$. Let $b \in \cU^{y}$ be such that $b \nind_{A} a_{i}$ for all $i < \kappa$. Then $\bdn(\tp(b/A)) \geq \kappa$.
\end{fact}

\begin{proposition}\label{prop: key for types in dpmin}
	Assume  that $G$ is an arbitrary type-definable group which is inp-minimal. If $p \in S_G(\cU)$ is idempotent and generically stable, then $p$ is generically transitive.
\end{proposition}
\begin{proof}
	Assume that $G$ is inp-minimal and $p \in S_G(\cU)$ is a generically stable idempotent type. Let $\cU' \succ\cU $ be a monster model with respect to $\cU$ and $p': =p |_{\cU'}$ (the unique extension of $p$ to a type in $S_G(\cU')$ which is invariant over $M$). Then $p' \in S_{G}(\cU')$ is idempotent and  generically stable over $\cU$.

	Let $(a_1,a_0) \models p^{(2)} = (p')^{(2)}|_{\cU}$. Let $b := a_1 \cdot a_0 \in G(\cU')$, by assumption $\bdn \left( \tp(b/\cU) \right) \leq \bdn(G(x)) \leq 1$. As $a_1 \ind_{\cU} a_0$, by Fact \ref{fac: burden gen stab} we must have at least one of the following (by idempotence and Fact \ref{fac: basic gen stab}(4)):
	\begin{enumerate}
		\item $b \ind_{\cU} a_0$, hence $(a_1 \cdot a_0, a_0) \models p^{(2)}$;
		\item or $b \ind_{\cU} a_1$, hence $(a_1 \cdot a_0, a_1) \models p^{(2)}$.
	\end{enumerate}
	
%	Since for any small $M_1, M_2 \prec \cU$ there exists some small $M_1, M_2 \prec M \prec \cU$, the same case has to occur over all small $M \prec \cU$. So for $(a_1, a_0) \models p^{(2)}$ we have 
%		\begin{enumerate}
%		\item either  $(a_1 \cdot a_0, a_0) \models p^{(2)}$;
%		\item or $(a_1 \cdot a_0, a_1) \models p^{(2)}$.
%			\end{enumerate}
%	In the first case we get that $p$ is generically transitive, hence concentrates on its left stabilizer $p \vdash \Stab(p)$ by Remark \ref{remark: basic}. In the second case, by generic stability  of $p$ we have: if $(a_1,a_0) \models p^{(2)}$, then $(a_0 \cdot a_1 ,a_0) \models p^{(2)}$. An argument symmetric to the proof of Remark \ref{remark: basic} but for the action of $G$ on the right shows that $p$ concentrates on its \emph{right} stabilizer $\Stab_{r}(p)$. 
%So in either case, we have shown that $p$ is a generically stable (left- or right-) translation invariant type of some type-definable group $K$.
%Which then implies by Fact \ref{fac: basic fsg} that $K$ is generically stable, and $p$ is both left and right  $K$-invariant. So $\Stab(p) = \Stab_r(p)$ (see Remark \ref{rem: fsg follows}), and a posteriori (using Remark \ref{remark: basic}) $p$ is generically transitive in the second case as well.
In both cases we obtain that $p$ is generically transitive (see Remark \ref{remark: basic}).
\end{proof}

\begin{remark}
	Fact \ref{fac: burden gen stab} also holds when $p$ a \emph{generically simple} type in an NTP$_2$ theory and $A$ is an extension base (this follows from \cite[Section 6]{chernikov2014theories} and \cite[Section 3.1]{simon2020amalgamation}).
\end{remark}
%
%%
%1) Which of course raises a question if for a generically stable type p, left stabilizer is equal to its right stabilizer (without assuming that p concentrates on either of them)?
%2) I think the argument again lifts from types to measures because dp-rank can be calculated using measures, but I will need to think a bit more about that.
%3) Maybe this suggests that in the general NIP case, or at least finite dp-rank case, we should first try to show that p lives on the stabilizer of some more general type, like p times p inverse, or some finite product of p and its inverses (all of them should be equal to p a posteriori). I think this question was already in the notes: assuming p is g.s. and idempotent, it's easy to see that p^{-1} is g.s. and idempotent, but what about their convolution?

%Regarding the statement
%
%
%
%"Assume T is dp-minimal (or just burden 1) group. Then our proof with k=2 shows that if p is gen.stable and idempotent and (a_2,a_1) realize p^{(2)}, then either (a_2 a_1, a_1) realizes p^{(2)}, or (a_2 a_1, a_2) realizes p^{(2)}"
%
%
%
%which Artem wrote, I think it follows from Theorem 2.9 of Garcia, Onshuus, and Usvyatsov. Namely, if the above conclusion fails, then a_2a_1 and a_i are forking dependent over the monster model \C (for i=1,2), whereas a_1 and a_2 are forking independent over \C, and so the generically stable pre-weight of tp(a_1a_2/\C)=p is at least 2 which by Theorem 2.9 of Garcia, Onshuus, and Usvyatsov implies that the burden of p is at least 2.

\subsection{Stable theories}\label{subsection: stable theories}
In this section we provide a proof that all idempotent types in stable groups are generically transitive. This was known from \cite{N3} (see also e.g.~\cite[Lemme 1.2]{blossier2016recherche} and references there), and recently generalized from idempotent types to idempotent Keisler measures in stable theories in \cite{chernikov2022definable} (see the discussion in Section \ref{sec: stable measures}). We provide two detailed proofs since  in the following sections we will extend them to the case when $p$ is a stable type in an arbitrary theory, and also to the case of a generically stable type $p$ in a simple or even rosy theory.

The proof uses local stratified ranks:
\begin{definition}\label{def: stab strat rank}
Let $G(x)$ be an $\emptyset$-type-definable group and $\varphi_0(x) \in \cL$ as in Section \ref{sec: setting types}. To a formula $\varphi(x,y) \in \cL$ (with $y$ an arbitrary tuple of variables) we associate a formula $\varphi'(x, y):=\varphi(x, y) \wedge \varphi_0(x)$. For $g \in \bar G$, put $\varphi_g(x, y):=g \cdot \varphi'(x, y):=(\exists z)(\varphi'(z, y) \wedge x=g \cdot z)$. Finally, let $\Delta_\varphi:=\{\varphi_g(x, y): g \in \bar G\}$. We consider the usual  notion of $\Delta_\varphi$-rank denoted by $R_{\Delta_\varphi}$ and $\Delta_\varphi$-multiplicity denoted by  $\mlt_{\Delta_\varphi}$.
\end{definition}

The proofs of all items except (3) in the following fact are standard arguments as for usual $\Delta$-ranks (see e.g.~\cite[Chapter 1]{pillay1996geometric}). The proof of (3) uses the choice of $\varphi_0(x)$ and is left as an exercise.

\begin{fact}\label{properties of local ranks}
Assume $T$ is stable. Then we have:
\begin{enumerate}
\item $R_{\Delta_\varphi}(x=x) < \omega$;
\item $R_{\Delta_\varphi}(\psi_1(x) \lor \psi_2(x))=\max (R_{\Delta_\varphi}(\psi_1), R_{\Delta_\varphi}(\psi_2))$;
\item $R_{\Delta_\varphi}$ is invariant under left translations by the elements of $\bar G$;
\item For any $A \subseteq B \subset \cU$ and $q \in S_G(B)$ we have that $q$ does not fork over $A$ if and only if $R_{\Delta_\varphi}(q)= R_{\Delta_\varphi}(q|_A)$ for every $\varphi \in \cL$;
\item $\mlt_{\Delta_\varphi}(q) =1$ for any complete type $q \in S_G(N)$ over a model $N \prec \cU$.
\end{enumerate}
\end{fact}

\begin{remark}
	Note that (5) follows from (4) and stationary of types over models in stable theories, but for $\aleph_0$-saturated $N$  it can also be shown easily directly, without using forking.
\end{remark}

In the proof below, the role of $\cU$ is played by $\cU'$.

\begin{proposition}\label{proposition: stable case}
If $T$ is stable and $p \in S_{G}(\cU)$ is idempotent, then $p$ is generically transitive.
\end{proposition}
\begin{proof}
By stability, $p$ is generically stable over some small $M \prec \cU$. Let $a \models p $ in $\cU'$. By Remark \ref{remark: basic}, it suffices to show that  $a \in \Stab_{\lee}(p')$, where $p' = p|_{\cU'}$

~

\noindent {\bf Method 1 (without using forking).}
Suppose for a contradiction that $a \cdot p' \ne p'$,  witnessed by a formula $\varphi(x, b) \in \cL(\cU')$. Since $p$ is a generically stable idempotent and $a \in p(\cU')$, we get $a \cdot p'|_\cU =p$. On the other hand, since $p'$ is invariant over $M$ and the ranks are invariant under automorphisms, a formula $\psi(x) \in p'$ with $R_{\Delta_{\varphi(x, y)}}(\psi)=R_{\Delta_{\varphi(x, y)}}(p')$ and  $\mlt_{\Delta_{\varphi(x, y)}}(\psi) = 1$ can be mapped by an automorphism over $M$ to a formula $\psi'(x) \in p'|_\cU =p$ with $R_{\Delta_{\varphi(x, y)}}(\psi')=R_{\Delta_{\varphi(x, y)}}(\psi)$ and $\mlt_{\Delta_{\varphi(x, y)}}(\psi') = 1$. By Fact \ref{properties of local ranks}(3), we also have $R_{\Delta_{\varphi(x, y)}}(a \cdot p')= R_{\Delta_{\varphi(x, y)}}(p')$. Summarizing, $p \subseteq p' \cap a \cdot p'$, $R_{\Delta_{\varphi(x, y)}}(p) =  R_{\Delta_{\varphi(x, y)}}(p')= R_{\Delta_{\varphi(x, y)}}(a \cdot p')$ and $\mlt_{\Delta_{\varphi(x, y)}}(p)=1$. Hence, $p'|_{\varphi(x, y)} = a \cdot p' |_{\varphi(x, y)}$, a contradiction.

~

\noindent {\bf Method 2 (using forking).}
Since $p'$ does not fork over $\cU$, by Fact \ref{properties of local ranks}(4), we get $R_{\Delta_\varphi}(p') =  R_{\Delta_\varphi}(p)$ for every $\varphi(x, y) \in \cL$. By Fact \ref{properties of local ranks}(3), we also have $R_{\Delta_\varphi}(a \cdot p')= R_{\Delta_\varphi}(p')$. Thus, $R_{\Delta_\varphi}(a \cdot p') = R_{\Delta_\varphi}(p)$. On the other hand, since $p$ is a generically stable idempotent and $a \in p(\cU')$, we have $p \subseteq a \cdot p'$. We conclude, using  Fact \ref{properties of local ranks}(4), that $a \cdot p'$ does not fork over $\cU$, and so it is the unique non-forking extension of $p$ which is equal to $p'$.
\end{proof}

\subsection{Stable types in arbitrary theories}\label{sec: gen trans stable types}

Method 1 from the proof of Proposition \ref{proposition: stable case} extends to the case when 
$p \in S_{G}(\cU)$ is {\em stable over $M$}, that is $p$ is $M$-invariant and $p|_M$ is {\em stable} (in a not necessarily stable theory). Recall that $p(x)|_M$ is \emph{stable} if there are no sequences $(a_i)_{i<\omega}$ in $\cU^x$ and $(b_i)_{i<\omega}$ in $\cU^y$ such that $a_i \models p|_M$ for all $i<\omega$ and for some $\varphi(x, y) \in \mathcal{L}$ we have $\models \varphi(a_i, b_j ) \iff i < j$ for all $i, j < \omega$. It follows from the definitions that if $p$ is stable over $M$, then it is generically stable over $M$.

 We refer to e.g.~\cite[Section 1]{adler2014generic} and \cite[Section 2]{hasson2010stable} for the basic properties of stable types. Method 1 applies when $p$ is  stable over $M$ because in this case $R_{\Delta_\varphi}(p) < \omega$, and we also  have item (3) of Fact \ref{properties of local ranks} (without any assumptions on $T$). Thus, Method 1 yields:

\begin{proposition}\label{proposition: stable type case}
If $p \in S_{G}(\cU)$ is idempotent and stable over some small $M \prec \cU$, then $p$ is generically transitive.
\end{proposition}

In order to prove this proposition using Method 2 from the proof of Proposition \ref{proposition: stable case}, we have to be careful with item (4) of Fact \ref{properties of local ranks}. Modifying a standard proof of item (4) (see e.g. ~\cite[Lemma 3.4]{pillay1996geometric}) gives the following weaker  variant:
\begin{fact}\label{fact: weak (4)}
%\footnote{The reason to assume  that $B$ is a set of realizations of $p|_M$ is to be able to use symmetry of forking and then stability of $\tp(B/MA)$.}
Assume $p \in S_G(\cU)$ and $p|_M$ is stable for a small model $M \prec \cU$. Let $B$ be a set of realizations of $p|_M$ and $A \subseteq B$. Then for any $q \in S_G(MB)$ extending $p|_M$ we have that $q$ does not fork over $MA$ if and only if $R_{\Delta_\varphi}(q)= R_{\Delta_\varphi}(q|_{MA})$ for every $\varphi \in \cL$.
\end{fact}
\noindent While adapting the proof of \cite[Lemma 3.4]{pillay1996geometric} for Fact \ref{fact: weak (4)}, the only essential difficulty is to show that if a formula $\varphi(x,B)$ does not fork over $MA$ (where $\varphi(x,y) \in \mathcal{L}(MA)$), then some positive Boolean combination of $MA$-conjugates of $\varphi(x,B)$ is definable over $MA$. For that one needs to use the assumption that $B$ is a set of realizations of $p|_M$ in order to have that $\tp(B/MA)$ is stable, which allows to use symmetry of forking and $\acl^{\eq}(MA)$-definability of non-forking extensions of $\tp(B/MA)$.

The following is a strengthening of the fact saying that $p$ is the unique non-forking extension of $p|_M$ to a global type, and follows by one of the standard proofs:

\begin{proposition}\label{proposition: strengthening of unique non-forking extension}
Assume that $p \in S_{G}(\cU)$ is generically stable, and $p' = p|_{\cU'}$ (where $\cU' \succ \cU$ is a monster model with respect to $\cU$).
Let $q \in S_G(\cU')$ be an extension of $p$ such that $q|_{\cU B}$ does not fork over $\cU$ for every small set $B$ of independent realizations of $p$. Then $q=p'$. 

More generally, in the assumption, it is enough to consider only the sets $B$ of independent realizations of $p$ containing a fixed realization $a$ of $p$ in $\cU'$.
\end{proposition}

\begin{proof}
We will deduce the proposition from the following claim.
\begin{clm*}
If $a_i \models q|_{\cU,a_{<i}}$ for all $i<n$, then $(a_0,\dots,a_{n-1})$ is a Morley sequence in $p$. In the more general version of the assumption, the same holds but assuming that $a_0=a$.
\end{clm*}

\begin{clmproof}
This is induction on $n$. The base step $n=0$ is trivial, as $q |_\cU =p$ by assumption.

{\em Induction step.} Consider any $a_i \models q|_{\cU,a_{<i}}$ in $\cU'$ for all $i\leq n$. By induction hypothesis  $a_i \models p|_{\cU,a_{<i}}$ for all $i< n$. The goal is to prove that $a_n \models p|_{\cU,a_{<n}}$, equivalently $q|_{\cU,a_{<n}}=p|_{\cU,a_{<n}}$.

Suppose for a contradiction that $\varphi(x,a_{<n}) \in q|_{\cU,a_{<n}}$ but  $\neg \varphi(x,a_{<n}) \in p|_{\cU,a_{<n}}$.
Extend $a_{<n}$ to a Morley sequence in $p$ by $b_n \models  p|_{\cU,a_{<n}}, b_{n+1} \models p|_{\cU,a_{<n},b_n},\dots$. Since $a_{<n}b_{\geq n}$ is a Morley sequence in $p$, we get, by generic stability of $p$, that the formula $\varphi(x,a_{<n-1},a_{n-1}) \wedge \neg \varphi(x,a_{<n-1},b_m)$ divides over $\cU$ for every $m \geq n$. As by assumption we know that $q|_{\cU,a_{<n},b_{\geq n}}$ does not fork over $\cU$, we conclude that $\varphi(x,a_{<n-1},b_m) \in q(x)$ for all $m \geq n$. 

Pick any $c \models q|_{\cU,a_{<n},b_{\geq n}}$. By the above conclusion, $\models \varphi(c,a_{<n-1},b_m)$ for all $m \geq n$, so, by generic stability of $p$, $\varphi(c,a_{<n-1},y) \in p'(y)$.

On the other hand, since  $\neg \varphi(x,a_{<n}) \in p|_{\cU,a_{<n}}$ and $b_{\geq n}$ is Morley sequence in $p$ over $\cU,a_{<n}$, we get that $\models \neg \varphi(b_i,a_{<n})$ for all $i\geq n$. Since $a_{<n}b_{\geq m}$ is a Morley sequence in $p$, and as such it is totally indiscernible, we get that $\models \neg \varphi(a_{n-1},a_{<n-1},b_i)$ for all $i \geq n$.
%On the other hand, since  $\neg \varphi(x,a_{<n}) \in p|_{\C,a_{<n}}$, we get that $\models \neg \varphi(b_i,a_{<n})$ for almost all $i\geq n$. Since $a_{<n}b_{\geq m}$ is a Morley sequence in $p$, and as such it is totally indiscernible, we get that $\models \neg \varphi(a_{n-1},a_{<n-1},b_i)$ for almost all $i \geq n$. 
Therefore, $\neg \varphi(a_{n-1},a_{<n-1},y) \in p'(y)$.

Summarizing, the last two paragraphs yield  $\varphi(c,a_{<n-1},y) \wedge \neg \varphi(a_{n-1},a_{<n-1},y) \in p'(y)$. This contradicts the $\cU$-invariance of $p'$, because $c \equiv_{\cU,a_{<n-1}} a_{n-1}$ as both these elements satisfy $q|_{\cU,a_{<n-1}}$.
\end{clmproof}

Now consider any $\varphi(x, b) \in q$. Pick $a_i \models q|_{\cU,a, b,a_{<i}}$ in $\cU'$ for all $0<i<\omega$, and put $a_0:=a$. By the claim, $(a_i)_{i<\omega}$ is a Morley sequence in $p$. Since $\models \varphi(a_i, b)$ for all $0<i<\omega$, by generic stability of $p$, we get $\varphi(x, b) \in p'$. Thus, $q=p'$.
\end{proof}

Now, to prove Proposition \ref{proposition: stable type case} via Method 2, using  Fact \ref{fact: weak (4)}, we get that $R_{\Delta_\varphi}(p'|_{\cU B})= R_{\Delta_\varphi}(p)$ for every set $B$ of realizations of $p$ (where $p':=p|_{\cU'}$). This implies (by Fact \ref{properties of local ranks}(3) ---  which does not require any assumptions) that  $R_{\Delta_\varphi}((a \cdot p')|_{\cU B})= R_{\Delta_\varphi}(a \cdot (p'|_{\cU B}))=R_{\Delta_\varphi}(p)$  for every set $B$ of realizations of $p$ containing $a$. Since $p \subseteq a \cdot p'$ (by the idempotence of $p$), we conclude, using  Fact \ref{fact: weak (4)}, that  $(a \cdot p')|_{\cU B}$ does not fork over $\cU$ for every set $B$ of realizations of $p$ containing $a$. Hence, $a \cdot p'=p'$ by Proposition \ref{proposition: strengthening of unique non-forking extension}.

\subsection{Simple theories}\label{sec: gen trans simple}

Assume that $T$ is a simple theory, and as usual that $p \in S_G(\cU)$ is idempotent and  generically stable over $M \prec \cU$. Method 2 extends from stable to simple theories using \emph{stratified Shelah degrees} in place of stratified local ranks. In simple theories, they are wrongly defined in Definition 4.1.4 of \cite{Wa} (as they are not left-invariant due to the lack of associativity outside $\bar G$). A way to fix it is to use suitable $\varphi(x, y)$ (as in Definition \ref{def: stab strat rank})   or to apply Definition 4.3.5 of \cite{Wa} (in the special case of type-definable rather than hyper-definable groups). In any case, by \cite{Wa}, stratified Shelah degrees satisfy items (1)--(4) of Fact \ref{properties of local ranks}, so Method 2 applies directly and yields the following generalization of Proposition \ref{proposition: stable case}:

\begin{proposition}\label{proposition: simple case}
Let $T$ be simple, $G$ an $\emptyset$-type-definable group, and $p \in S_{G}(\cU)$ idempotent and generically stable. Then $p$ is generically transitive.
\end{proposition}

\subsection{Rosy theories}\label{sec: rosy types}

In the case of groups in \emph{rosy theories}, again we can apply Method 2, using stratified local \emph{thorn ranks}. They were defined and studied in \cite{EKP} (see \cite[Definition 1.13]{EKP}) in the case of definable groups, and extend easily to type-definable groups (using $\varphi_0$ as in Definition \ref{def: stab strat rank}). By \cite{EKP}, stratified local thorn ranks satisfy items (1)--(4) of Fact \ref{properties of local ranks}, with thorn forking in place forking in item (4). However, \cite[Theorem 3.4]{garcia2013generic} tells us that if a type $q \in S(B)$ is generically stable and $A \subseteq B$, then $q$ forks over $A$ if and only if $q$ thorn forks over $A$. We have all the tools to prove the following  generalization of Proposition \ref{proposition: simple case} via Method 2. 

Before its statement, let us first recall local thorn ranks and define stratified local thorn ranks  $\th_{\Phi,\Theta,k}^G$:
\begin{definition}
For  a finite set $\Phi$ of partitioned formulas with object variables $x$ and parameter variables $y$, a finite set of formulas $\Theta$ in variables $y,z$, and natural number $k>0$, the \emph{$\th_{\Phi,\Theta, k}$-rank} is the unique function from the collection of all consistent formulas with parameters to $\textrm{Ord} \cup \{\infty\}$ satisfying: 
$\th_{\Phi,\Theta,k}(\psi) \geq \alpha +1$ if
and only if there is $ \varphi \in \Phi $, some $\theta(y,z) \in
\Theta$ and parameter $c$ such that:

\begin{enumerate}
\item  $\th_{\Phi,\Theta,k}(\psi(x)\land \varphi(x,a))\geq \alpha $ for infinitely many $a\models \theta(y,c) $, and
\item  $\left\{ \varphi \left( x,a\right) : a \models \theta(y,c) \right\}$ is
$k$-inconsistent.
\end{enumerate}
Given a (partial) type $\pi(x)$ closed under conjunction we define $\th_{\Phi,\Theta,k}(\pi(x))$ to be the minimum of $\th_{\Phi,\Theta,k}(\psi)$ for $\psi\in\pi(x)$. 
\end{definition}

\begin{definition}
\begin{enumerate}
\item  For a formula $\varphi(x, y) \in \mathcal{L}$, let $\tilde{\varphi}(x,t, y):= (\exists z)(\varphi(z, y) \wedge \varphi_0(z)  \wedge x=t \cdot z)$, where $\varphi_0(x)$ is chosen in Section \ref{sec: setting types}.
\item For a finite set $\Phi$ of formulas in variables $x,y$, put $\widetilde{\Phi}:=\{\tilde{\varphi}(x,t,y):\varphi(x,y)\in \Phi\}$. For a finite set of formulas $\Theta$ in variables $y,z$, put $\Theta^*(t,y;t',z) := \{\theta(y,z) \land t=t':\theta\in \Theta\}$.
\item The \emph{stratified $\th_{\Phi, \Theta, k}^G$-rank} is defined as the  unique function satisfying: \\
$\th_{\Phi,\Theta,k}^G(\psi) \geq \alpha +1$ if and only if there is a formula $\varphi \in \Phi$, some $\theta^*(t,y;t'z) \in
\Theta^{*}$ and parameters $g \in G$ and $c$ anywhere such that:

\begin{enumerate}
\item   $\th_{\Phi, \Theta, k}^G(\psi(x) \land \tilde{\varphi}(x,g,b))\geq n$ for infinitely many $(g,b) \models \theta^*(t,y;g,c)$, and
\item  $\left\{ \tilde{\varphi}( x,g,b) : (g,b) \models \theta^*(t,y;g,c) \right\}$ is $k$-inconsistent.
\end{enumerate}
Given a (partial) type $\pi(x)$ closed under conjunction we define $\th_{\Phi,\Theta,k}^G(\pi(x))$ to be the minimum of $\th_{\Phi,\Theta,k}^G(\psi)$ for $\psi\in\pi(x)$.
\end{enumerate}
\end{definition}

\begin{proposition}\label{prop: gen trans in rosy}
If $T$ is rosy and $G$ is a type-definable group, then every generically stable idempotent type $p \in S_{G}(\cU)$ is generically transitive.
\end{proposition}

\begin{proof}
Assume $p$ is generically stable over $M \prec \cU$ and let $a \models p$ in $\cU' \succ \cU$ and $p'= p|_{\cU'}$. As usual, by Remark \ref{remark: basic} it suffices to show that $a \in \Stab_{\lee}(p')$. Since $p'$ does not fork over $\cU$, it does not $\th$-fork over $\cU$, so $\th^G_{\Phi,\Theta,k}(p')=\th^G_{\Phi,\Theta,k}(p)$  for every finite $\Phi$, $\Theta$, $k$. Since $\th^G_{\Phi,\Theta,k}(a \cdot p')=\th^G_{\Phi,\Theta,k}(p')$, we conclude that $\th^G_{\Phi,\Theta,k}(a \cdot p')=\th^G_{\Phi,\Theta,k}(p)$. On the other hand, since $p$ is a generically stable idempotent and $a \in p(\cU')$, we have $p \subseteq a \cdot p'$. Therefore, $a \cdot p'$ does not $\th$-fork over $\cU$. In order to conclude that $a \cdot p'$ does not fork over $\cU$ (and so coincides with $p'$, which finishes the proof), by virtue of  \cite[Theorem 3.4]{garcia2013generic}, it remains to show the following:

\begin{clm*}
$a \cdot p'$ is generically stable over $M,a$. 
\end{clm*}

The claim follows from Fact \ref{fac: dcl of gs is gs} since $p'$ is generically stable over $M$ (and so over $M,a$) and $a \cdot p'$ is realized by $a \cdot g \in \dcl(a,g)$ for $g \models p'$.
%\begin{clmproof}
%We will show that $ap'$ is generically stable over $M,a$. First, let us check that it is invariant over $M,a$.  Consider any $\varphi(x, \bar c) \in ap'$ and $f \in \Aut(\C'/M,a)$. We need to show that $\varphi(x,f(\bar c)) \in ap'$. Take $g \models p'$. Then $ag \models ap'$, so $\models \varphi(ag,\bar c)$. By invariance of $p'$ over $M$, we get $\models \varphi((f(a)g,f(\bar c))$. Since $f(a)=a$, we conclude that $\models \varphi((ag,f(\bar c))$, so $\varphi(x,f(\bar c)) \in ap'$ as required.
%\end{clmproof}
\end{proof}

\subsection{Chain conditions for type-definable groups on a generically stable type}\label{sec: chain conds}

One approach towards Problem \ref{conjecture: main conjecture} is to adapt Newelski's ``2-step generation'' theorem \cite[Theorem 2.3]{N3} that provides a positive answer for types in stable theories.
One ingredient is the existence of a smallest type-definable group containing a given type in a stable theory. In this section we investigate this question (and corresponding chain conditions and connected components) for groups type-definable using parameters from a Morley sequence of a generically stable type.

\begin{definition}
If $p \in S_x(\cU)$ is an $A$-definable global type and $\pi(x,y)$ is a small partial type over $A$ (with $x,y$ possibly infinite tuples of variables), we let $d_{p}x \,  \pi(x,y) := \{ d_{p}x \, \varphi(x,y) : \varphi(x,y) \in \pi(x,y) \} $. Hence $d_{p}x \,  \pi(x,y)$ is a partial type over $A$ in  the variables $y$ so that for any $b \in \cU^y$, $b \models d_{p}x \,  \pi(x,y)$ if and only if $\pi(x,b) \subseteq p$.
\end{definition}

The following is a generalization of \cite[Lemma 2.1]{dobrowolski2012omega} to type-definable groups (combining it with \cite[Lemma 2.1]{haskel2018maximal}).

\begin{proposition}\label{prop: new chain cond gs}
	Let $G = G(x)$ be an $A$-type-definable group and $p(y) \in S_G(\cU)$  a global type generically stable over a small set $A \subset \cU$. Assume that $H(x,\bar{y},\bar{z})$ is a countable  partial type over $A$, with $\bar{y} = (y_i : i \in \omega)$. Let $\bar{c}$ be a tuple in $\cU^{\bar{z}}$, $(b_i : i \in I)$ a small infinite set of realizations of $p|_{A \bar{c}}$ (forking-)independent over $A$, and let $\bar{i} = (i_{\alpha} : \alpha < \omega) \in I^{\omega}$ with $i_{\alpha}$ pairwise distinct and $I \setminus \{i_{\alpha} : \alpha < \omega\}$ 	 infinite. Suppose $H(\cU; \bar{b}_{\bar{i}}, \bar{c}) \cap G(\cU) \leq G(\cU)$ is a subgroup, where $\bar{b}_{\bar{i}} = (b_{i_{\alpha}}: \alpha < \omega)$. Let $$\mathcal{J} := \left\{(j_{\alpha}: \alpha < \omega) \in I^{\omega} :  \left( \forall \alpha \neq \alpha' \in \omega, j_{\alpha} \neq j_{\alpha'} \right) \land \left( |\{\alpha \in \omega : j_{\alpha} \neq i_{\alpha}\}| < \aleph_0 \right)\right\}.$$
	Then $\bigcap_{\bar{j} \in \mathcal{J}} H(\cU; \bar{b}_{\bar{j}}, \bar{c}) \cap G(\cU)$ is type-definable over $A \bar{c}$ as $d_{p^{(\omega)}}\bar{y} \,  H(x,\bar{y}, \bar{c}) \cap G(x)$ (note that the type $p^{(\omega)}$ is definable over $A$).
	\end{proposition}

\begin{remark}

If $H(x,\bar y,\bar z)$ is uncountable, then we can write it as a conjunction of countable types defining relatively definable subgroups of $G(\mathcal{U})$, and so we conclude that Proposition 2.39 holds for uncountable types $H(x,\bar y,\bar z)$, too. 

\end{remark}

	\begin{proof}
	
	For $n < \omega$, let $I_n := \{i_{\alpha} : n < \alpha < \omega \}$ and $\bar{i}_n := (i_{\alpha} : n < \alpha < \omega)$.
		\begin{clm}\label{cla: chain cond cla 1}
			If $I$ is uncountable, then for any (infinite) countable subset $I' \subseteq I \setminus I_0$ we have 
			$$\bigcap_{i \in I \setminus I_0} H(\cU; b_i, \bar{b}_{\bar{i}_0} ; \bar{c}) \cap G(\cU)  = \bigcap_{i \in I'} H(\cU; b_i, \bar{b}_{\bar{i}_0}; \bar{c}) \cap  G(\cU).$$
			 If the claim holds for uncountable $I$, it also holds for countable $I$. Namely, if $I$ is countable, we can extend it to uncountable $I''$ and apply the claim for $I''$. 
		\end{clm}
		\begin{clmproof}
		Suppose not, so there exists an infinite countable $I' \subseteq I \setminus I_0$, $i^{\ast} \in I \setminus (I_0 \cup I')$ and 
		$$g \in \left(  G(\cU) \cap \bigcap_{i \in I'} H(\cU; b_{i}, \bar{b}_{\bar{i}_0}; \bar{c} ) \right) \setminus H(\cU; b_{i^{\ast}}, \bar{b}_{\bar{i}_0}; \bar{c}).$$
		
	\noindent	Fix any sequence $\bar{j} = (j_{\alpha} : \alpha < \omega_1) \in (I \setminus I_0)^{\omega_1}$ with  $j_{\alpha}$ pairwise distinct. As $(b_i : i \in I)$ is an $A$-independent set of realizations of $p|_{A \bar{c}}$, by generic stability of $p$ and Fact \ref{fac: basic gen stab}(2) it is a Morley sequence in $p$ over $A\bar{c}$, hence it is totally indiscernible over $A \bar{c}$. Then,  for every $\alpha < \omega_1$ there is an automorphism $\sigma_{\alpha}$ of $\cU$ fixing $A \bar{c} \bar{b}_{\bar{i}_0}$ and sending $ b_{i^{\ast}}$ to $b_{j_{\alpha}}$ and $(b_i : i \in I')$ to a sequence containing $(b_{j_{\beta}} : \beta < \alpha)$ as a subsequence. Then $g_{\alpha} := \sigma_{\alpha}(g) \in \left(  G(\cU) \cap \bigcap_{\beta < \alpha} H(\cU; b_{j_{\beta}}, \bar{b}_{\bar{i}_0}; \bar{c} ) \right) \setminus H(\cU; b_{j_{\alpha}}, \bar{b}_{\bar{i}_0}; \bar{c})$ for all $\alpha < \omega_1$. As the partial type $H(x,\bar{y}, \bar{z})$ is countable, there is a formula $\varphi(x; y_0, \bar{b}_{\bar{i}_0}; \bar{c})$  (with hidden parameters from $A$) so that $\varphi(x; b_{j_{\alpha}}, \bar{b}_{\bar{i}_0}; \bar{c}) \in  H(x; b_{j_{\alpha}}, \bar{b}_{\bar{i}_0}; \bar{c})$ and $g_{\alpha}  \in \left(  G(\cU) \cap \bigcap_{\beta < \alpha} H(\cU; b_{j_{\beta}}, \bar{b}_{\bar{i}_0}; \bar{c} ) \right) \setminus \varphi(\cU; b_{j_{\alpha}}, \bar{b}_{\bar{i}_0}; \bar{c})$ for uncountably many $\alpha < \omega_1$. Then, noting that $(b_{j_{\alpha}} : \alpha < \omega_1) \models p^{(\omega_1)} |_{A \bar{b}_{\bar{i}_0} \bar{c}}$, by compactness, Ramsey  and indiscernibility we can find  $(b'_{\alpha} : \alpha < \omega)$ in $\cU$ realizing $p^{(\omega)}|_{A \bar{b}_{\bar{i}_0} \bar{c}}$ and $(g'_{\alpha} : \alpha < \omega)$ in $G(\cU)$ so that the sequence $(g'_\alpha, b'_{\alpha})_{\alpha < \omega}$ is indiscernible over $A\bar{b}_{\bar{i}_0} \bar{c}$, and for all $\alpha < \omega$ we have   
	$g'_{\alpha} \in \left(  G(\cU) \cap \bigcap_{\beta < \alpha} H(\cU; b'_{\beta}, \bar{b}_{\bar{i}_0}; \bar{c} ) \right) \setminus \varphi(\cU; b'_{\alpha}, \bar{b}_{\bar{i}_0}; \bar{c})$, $\varphi(x; b'_{\alpha}, \bar{b}_{\bar{i}_0}; \bar{c}) \in  H(x; b'_{\alpha}, \bar{b}_{\bar{i}_0}; \bar{c})$ and $G(\cU) \cap H(\cU; b'_{\alpha}, \bar{b}_{\bar{i}_0}; \bar{c} ) \leq G(\cU)$.

	First assume  $g'_{\alpha} \notin H(\cU, b'_{\beta}, \bar{b}_{\bar{i}_{0}}, \bar{c})$ for some $\beta > \alpha$ in $\omega$, then $\models \neg \psi(g'_{\alpha}, b'_{\beta}, \bar{b}_{\bar{i}_{0}}, \bar{c})$ for some $\psi(x; b'_{\beta}, \bar{b}_{\bar{i}_0}; \bar{c}) \in  H(x; b'_{\beta}, \bar{b}_{\bar{i}_0}; \bar{c})$. Hence, by indiscernibility, we have $\models \psi(g'_{\alpha}, b'_{\beta}, \bar{b}_{\bar{i}_0}, \bar{c}) \iff \beta < \alpha$ for all $\alpha, \beta \in \omega$. Taking $\alpha \in \omega$ sufficiently large, as $(b'_{\beta} : \beta < \omega) \models p^{(\omega)}|_{A \bar{b}_{\bar{i}_0} \bar{c}}$, this contradicts generic stability of $p$. 
	
	Otherwise we have $g'_{\alpha} \in H(\cU, b'_{\beta}, \bar{b}_{\bar{i}_0}, \bar{c}) \iff \alpha \neq \beta$, for all $\alpha, \beta \in \omega$. Note that for any $h_1, h_2 \in G(\cU) \cap H(\cU, b'_{0},\bar{b}_{\bar{i}_0}, \bar{c}) $, we have $h_1 \cdot g'_0 \cdot h_2 \notin H(\cU, b'_{0}, \bar{b}_{\bar{i}_0}, \bar{c})$ (as $g'_0 \in G(\cU) \setminus H(\cU, b'_{0}, \bar{b}_{\bar{i}_0}, \bar{c})$). By compactness there is a formula $\psi(x,y_0, \bar{y}, \bar{z}) \in H(x,y_0, \bar{y}, \bar{z})$ (without loss of generality the partial type $H$ is closed under conjunctions) so that:
	 for any $h_1,h_2 \in  G(\cU) \cap H(\cU, b'_{0}, \bar{b}_{\bar{i}_0}, \bar{c})$, $\models \neg \psi(h_1 \cdot g'_0 \cdot h_2 , b'_{0}, \bar{b}_{\bar{i}_0}, \bar{c})$. By indiscernibility we then have: for any $\alpha < \omega$, for any $h_1,h_2 \in G(\cU) \cap H(\cU, b'_{\alpha}, \bar{b}_{\bar{i}_0}, \bar{c})$, $\models  \neg \psi(h_1 \cdot g'_{\alpha} \cdot h_2 , b'_{\alpha}, \bar{b}_{\bar{i}_0}, \bar{c})$. Now let $K$ be any finite subset of $\omega$, and let $g'_{K} := \prod_{\alpha \in K} g'_{\alpha}$. As $g'_{\alpha} \in H(\cU, b'_{\beta}, \bar{b}_{\bar{i}_0}, \bar{c}) \iff \alpha \neq \beta$, it follows that for all $\alpha < \omega$, $\models \psi(g'_K, b'_{\alpha}, \bar{b}_{\bar{i}_0}, \bar{c}) \iff \alpha \notin K$. Taking $K$ sufficiently large, this again contradicts generic stability of $p$ (in fact, even ``generic NIP'' of $p$).
		\end{clmproof}
		
		\begin{clm}\label{cla: chain cond cla 2}
			For $n < \omega$, let 
			$$\mathcal{J}_n := \{(j_{\alpha} : \alpha < n+1) \in (I \setminus I_n)^{n+1} :  \forall \alpha, \alpha' < n+1 \text{ with } \alpha \neq \alpha', j_{\alpha} \neq j_{\alpha'} \}.$$ 
			Then for any $g \in G(\cU)$ we have: $g \in \bigcap_{\bar{j} \in \mathcal{J}_n} H(\cU;\bar{b}_{\bar{j}}, \bar{b}_{\bar{i}_n}; \bar{c})$ if and only if $g \models d_{p^{(n+1)}}\bar{y}_n \left( H(x, \bar{y}_n, \bar{b}_{\bar{i}_{n}}, \bar{c}) \right)$, where $\bar{y}_n = (y_i : i < n+1)$.
		\end{clm}
		\begin{clmproof}
		We prove this by induction on $n < \omega$.
		
		For $n=0$, assume $g \in G(\cU) \cap \bigcap_{j \in \mathcal{J}_0} H(\cU; b_{j}, \bar{b}_{\bar{i}_0}; \bar{c})$. As $I$ is small, we can choose a sequence $(b_i : i \in I')$  of independent realizations of $p|_{g \{b_i : i \in \mathcal{J}_0\} \bar{b}_{\bar{i}_0} \bar{c} A}$ in $\cU$  with $I'$ uncountable.
		Applying Claim \ref{cla: chain cond cla 1} to the independent set $(b_i : i \in \mathcal{J}_0 \cup I')$ of realizations of $p|_{A \bar{c}}$, we have 
		 $$\bigcap_{j \in  \mathcal{J}_0} H(\cU; b_j, \bar{b}_{\bar{i}_0} ; \bar{c}) \cap G(\cU)  = \bigcap_{j \in  \mathcal{J}_0 \cup I'} H(\cU; b_j, \bar{b}_{\bar{i}_0}; \bar{c}) \cap  G(\cU),$$
		 so in particular $\models H(g; b_j, \bar{b}_{\bar{i}_0}; \bar{c})$ for any $j \in I'$, hence $g \models d_p y_0 \, H(x, y_0, \bar{b}_{\bar{i}_0}; \bar{c})$.
		 
		 Conversely, suppose $g \models d_p y_0 \, H(x, y_0, \bar{b}_{\bar{i}_0}; \bar{c})$, and we let $(b_i : i \in I')$ realize $p^{(\omega_1)}|_{g \{b_i : i \in \mathcal{J}_0\} \bar{b}_{\bar{i}_0} \bar{c} A}$. Then $g \in \bigcap_{j \in   I'} H(\cU; b_j, \bar{b}_{\bar{i}_0}; \bar{c}) \cap  G(\cU)$.  
		 By Claim \ref{cla: chain cond cla 1} again we have 
		 \begin{gather*}
		 	\bigcap_{j \in   I'} H(\cU; b_j, \bar{b}_{\bar{i}_0}; \bar{c}) \cap  G(\cU) = \bigcap_{j \in  \mathcal{J}_0 \cup I'} H(\cU; b_j, \bar{b}_{\bar{i}_0}; \bar{c}) \cap  G(\cU) \\ 
		 	= \bigcap_{j \in  \mathcal{J}_0} H(\cU; b_j, \bar{b}_{\bar{i}_0}; \bar{c}) \cap  G(\cU),
		 \end{gather*}
	hence $g \in \bigcap_{j \in  \mathcal{J}_0} H(\cU; b_j, \bar{b}_{\bar{i}_0}; \bar{c})$. This shows the case $n=0$.
		 
		 Now suppose that the claim holds for $n$, and we show it for $n+1$. For $i \in I \setminus I_{n+1}$, let $\mathcal{J}_n(i)$ be the set of sequences of $n+1$ pairwise distinct elements in $I \setminus (I_{n+1} \cup \{i\})$. By the induction hypothesis  and indiscernibility, given $g \in G(\cU)$,  
		 $$g \in \bigcap_{\bar{j} \in \mathcal{J}_n(i)} H(\cU, \bar{b}_{\bar{j}}, b_i,\bar{b}_{\bar{i}_{n+1}}, \bar{c}) \  \Leftrightarrow \ g \models d_{p^{(n+1)}} \bar{y}_{n} \, H(x, \bar{y}_{n}, b_i, \bar{b}_{\bar{i}_{n+1}}, \bar{c}). $$
		 
		\noindent Let $K(x, b_i, \bar{b}_{\bar{i}_{n+1}}, \bar{c})$ be the countable partial type $d_{p^{(n+1)}} \bar{y}_{n} \, H(x, \bar{y}_{n}, b_i, \bar{b}_{\bar{i}_{n+1}}, \bar{c})$, then $G(\cU) \cap K(\cU, b_i, \bar{b}_{\bar{i}_{n+1}}, \bar{c})$ is the corresponding subgroup. Then
		\begin{gather*}
			G( \cU) \cap \bigcap_{\bar{j} \in \mathcal{J}_{n+1}} H(\cU;\bar{b}_{\bar{j}}, \bar{b}_{\bar{i}_{n+1}}; \bar{c}) = G( \cU)  \cap \bigcap_{i \in I \setminus I_{n+1}} \bigcap_{\bar{j} \in \mathcal{J}_n(i)} H(\cU, \bar{b}_{\bar{j}}, b_i,\bar{b}_{\bar{i}_{n+1}}, \bar{c}) \\
			=G(\cU) \cap  \bigcap_{i \in I \setminus I_{n+1}} K(\cU, b_i, \bar{b}_{\bar{i}_{n+1}}, \bar{c}) \\
			= \{g \in G(\cU) : g \models d_p y_0 \, K(x, y_0, \bar{b}_{\bar{i}_{n+1}}, \bar{c})\} \\
			= \{ g \in G(\cU) : g \models d_{p^{(n+2)}} \bar{y}_{n+1} \, H(x, \bar{y}_n, \bar{b}_{\bar{i}_{n+1}}, \bar{c} )\},
		\end{gather*}
		where the equality between the second and the third lines holds because of the case $n=0$ applied to the groups relatively defined by $K$.
		\end{clmproof}

		Assume now $g \in G(\cU) \cap \bigcap_{\bar{j} \in \mathcal{J}}H(\cU, \bar{b}_{\bar{j}}, \bar{c})$, and let $\psi(x,\bar{y}, \bar{c}) \in H(x, \bar{y}, \bar{c})$ be arbitrary, since it only contains finitely many variables we have $\psi(x, \bar{y}_n, \bar{c})$ for some $n < \omega$. As in particular we have $g \in G(\cU) \cap \bigcap_{\bar{j} \in \mathcal{J}_n}H(\cU, \bar{b}_{\bar{j}}, \bar{b}_{\bar{i}_n},  \bar{c})$, by Claim \ref{cla: chain cond cla 2} this implies $g \models d_{p^{(n+1)}} \bar{y}_n \, \psi (x,\bar{y}_n, \bar{c})$, hence  $g \models d_{p^{(\omega)}}\bar{y} \,  \psi(x,\bar{y},\bar{c})$. This gives one inclusion.
		
		Conversely, suppose $g \models d_{p^{(\omega)}}\bar{y}  \, H(x,\bar{y},\bar{c})$. Let $(b'_i : i \in I')$ be an independent set of realizations of $p|_{g \{b_i : i \in I\} \bar{c} A}$ with $I'$ of the same cardinality as $I$, choose $\bar{i}' \in (I')^{\omega}$ and let $\mathcal{J}'$ be defined with respect to $I'$ and $\bar{i}'$ in the same way as $\mathcal{J}$ is defined with respect to $I$ and $\bar{i}$. Then in particular $\bar{b}'_{\bar{j}} \models p^{(\omega)}|_{g \bar{c} A}$ for any $\bar{j} \in \mathcal{J}'$, so $g \in G(\cU) \cap \bigcap_{\bar{j} \in \mathcal{J}'} H(\cU, \bar{b}'_{\bar{j}}, \bar{c})$. We emphasize that in the following, partial type are treated here as \emph{sets of formulas}. Now given an arbitrary formula $\psi(x)$ in the union $\bigcup_{\bar{j} \in \mathcal{J}}H(x, \bar{b}_{\bar{j}}, \bar{c})$, we can choose $n < \omega$ and finite $\bar{j}_1 \in (I \setminus \bar{i} )^{<\omega}$ so that all parameters of $\psi$ are in $\bar{c}A (b_{i_{\alpha}} : \alpha < n) \bar{b}_{\bar{j}_1}$. By construction $(b_i : i \in I \cup I')$ is an independent set of realizations of $p|_{g \bar{c} A}$, in particular it is totally indiscernible over $g \bar{c} A$.

		We can thus choose an automorphism $\sigma \in \Aut(\cU)$ fixing $\bar{b}_{\bar{i}_n} \bar{b}'_{\bar{i}'_n} A \bar{c}$ with  $\sigma(b_{i_{\alpha}}) = b'_{i_{\alpha}}$ for $\alpha < n$ and $\sigma(\bar{b}_{\bar{j}_1}) = \bar{b}'_{\bar{j}'_1}$ for some $\bar{j}'_1 \in (I' \setminus \bar{i}')^{< \omega}$.  Hence $\psi(x, \bar{b}'_{\bar{j}'_1}, (b'_{i_{\alpha}} : \alpha < n), \bar{c}) = \psi(x, \sigma(\bar{b}_{\bar{j}_1}), (\sigma(b_{i_{\alpha}}))_{ \alpha < n}, \bar{c})$ is a formula in the union $\bigcup_{\bar{j} \in \mathcal{J}'}H(x, \bar{b}'_{\bar{j}}, \bar{c})$, which contains $g$ by assumption. And $\sigma$ fixes $\bigcap_{\bar{j} \in \mathcal{J}'}H(\cU, \bar{b}'_{\bar{j}}, \bar{c})$ setwise, as it is type-definable over $\bar{b}'_{\bar{i}_n} \bar{c} A$ by Claim \ref{cla: chain cond cla 2}. Hence $\sigma(g) \models \psi(x, \sigma(\bar{b}_{\bar{j}_1}), (\sigma(b_{i_{\alpha}}))_{ \alpha < n}, \bar{c})$, and so $g \models \psi(x, \bar{b}_{\bar{j}_1}, (b_{i_{\alpha}})_{ \alpha < n}, \bar{c})$. We have thus shown that $g \in G(\cU) \cap \bigcap_{\bar{j} \in \mathcal{J}}H(\cU, \bar{b}_{\bar{j}}, \bar{c})$, giving the reverse inclusion. This  finishes the proof of the theorem.
	\end{proof}

\begin{question}

Do we have the analogue of Claim \ref{cla: chain cond cla 1} if instead of a single realization $b_i$ of $p|_{A \bar{c}}$ we allow to vary tuples of independent realizations of $p$? By essentially the same proof, this holds under the stronger assumption that $p^{(n)}$ is generically stable for all $n < \omega$.
  We also expect that an analogue of Proposition \ref{prop: new chain cond gs} holds for invariant instead of type-definable subgroups, but do not pursue it here. \end{question}

\begin{corollary}\label{cor: chain cond cor}
	Let $G$ be an $\emptyset$-type-definable group and $p \in S_G(\cU)$ a global generically stable type. Assume that $p$ is invariant over a small set $A \subset \cU$, and consider the family
	$$\mathcal{H}_{p,A} := \left\{ H \leq G(\cU) : H \textrm{ is type-definable over }A\bar{b}, \bar{b} \models p^{(|\bar{b}|)}|_{A}, p \vdash H(x) \right\}.$$
	Then the group $G_{p,A} := \bigcap \mathcal H_{p,A}$ is type-definable over $A$, and $p|_{A} \vdash G_{p,A}$. 
\end{corollary}

\begin{proof}
Assume we are given an arbitrary small linear order $I$ and arbitrary $\bar{b} \models p^{(I)}|_{A}$, and let $H \leq G(\cU)$ with $p \vdash H$ be type-definable over $\bar{b}A$. We can write $H = G(\cU) \cap \bigcap_{\alpha < \gamma} H_{\alpha}(\cU, \bar{b}_{\alpha})$ for some ordinal $\gamma$, with each $H_{\alpha}(x, \bar{b}_{\alpha}) \subseteq H(x)$ a countable partial type relatively defining a subgroup of $G(\cU)$, and $\bar{b}_{\alpha}$ a countable subsequence of $\bar{b}$. Fix some $\alpha < \gamma$, without loss of generality $\bar{b}_{\alpha}$ is co-infinite in $\bar{b}$. Note that, taking the intersection over all countable subsequences $\bar{b}'$ of $\bar{b}$ equal to $\bar{b}_{\alpha}$ except for finitely many elements, $H'_{\alpha}:=  G(\cU) \cap \bigcap_{\bar{b}'} H_{\alpha}(\cU, \bar{b}') \leq G(\cU) \cap H_{\alpha}(\cU, \bar{b}_{\alpha})$ is type-definable over $A$ by Proposition \ref{prop: new chain cond gs}. And, as $p$ is $A$-invariant and $\bar{b}' \equiv_{A} \bar{b}_{\alpha}$, we still have $p(x) \vdash H'_{\alpha}(x)$. Thus, for $H' := \bigcap_{\alpha < \gamma} H'_{\alpha}$, we have $p(x) \vdash H'(x)$ and $H'(\cU) \cap G(\cU) \leq H(\cU)$.

%
%We can choose some $\bar{b}'' \models p^{(\kappa)}|_{A \bar{b} \bar{b}'}$.
%Let $\mathcal{F} := \left\{ H \leq G(\cU) : p \vdash H(x) \right\}$. Then the family $\mathcal{F}$ is $\Aut(\cU/A)$-invariant by $A$-invariance of $p$, and is closed under supergroups and intersections. Applying Lemma \ref{lem: chain cond with params from gs MS}(3) with $\mathcal{F}$ and the sequence $\bar{b}' + \bar{b}'' \models p^{(I + \kappa)}|_{A}$, we find some $H'' \in \mathcal{F}$ type-definable over $\bar{b}''A$ with $H'' \subseteq H'$. Applying Lemma \ref{lem: chain cond with params from gs MS}(3) again to the sequence $\bar{b} + \bar{b}'' \models p^{(\kappa + \kappa)}|_{A}$, we find some $H \in \mathcal{F}$ type-definable over $\bar{b}A$ with $H \subseteq H''$. 

This shows that the group  $G_{p,A}$ is type-definable over $A$, and $p|_{A} \vdash G_{p,A}(x)$.
%	Now assume $A' \supseteq A$. Then the family $\mathcal H_{p,A'}$ is type-definable over $A'$.
%	Assume $A'' \equiv_{A} A'$. Then $H_{p,A'} = H_{p,A''}$? Take $\bar{b} \models p^{(\kappa)}|_{A' A''}$.
%	
%	Take $H''$ type-definable over $A'' \bar{b}''$. Then it is 	
%	
\end{proof}

\begin{question}\label{question: chain conditions}
\begin{enumerate}
	\item Does $G_{p,A}$ depend on $A$? 
	\item Is $G_{p,A}$ an intersection of definable groups?
	\item 	Let $p \in S_G(\cU)$ be a generically stable type (idempotent and such that $p^{(n)}$ is generically stable for all $n \in \omega$, if it helps). Does there exist a smallest type-definable (over a small set of parameters anywhere in $\cU$) group $H$ with $p \vdash H(x)$?
\end{enumerate}
	
\end{question}

All of these questions have a positive answer assuming $p$ is idempotent and generically transitive:

\begin{proposition}\label{prop: chain cond loc stab}
	Assume that $p \in S_{G}(\cU)$ is generically stable, idempotent and generically transitive. Then $\Stab_{\lee}(p)$ is the smallest type-definable (with parameters anywhere in $\cU$) subgroup $H$ of $G$ so that $p \vdash H$.
	
	Moreover, if $p^{(n)}$ is generically stable for all $n \in \omega$, then $\Stab_{\lee}(p) = G_{p,A}$ for any small $A \subseteq \cU$ so that $p$ is invariant over $A$, and $G_{p,A}$ is an intersection of definable groups.
\end{proposition}
\begin{proof} 
	
	The first part is by Remark \ref{rem: fsg follows}.	
	Moreover, let $A \subseteq \cU$ be such that $p$ is invariant over $A$. As $\Stab_{\lee}(p)$ is type-definable over $A$ (by Proposition \ref{proposition: intersection of relatively definable groups}) and $p \vdash \Stab_{\lee}(p)$, by the definition of $G_{p,A}$ we have $ G_{p,A} \subseteq \Stab_{\lee}(p)$, and $\Stab_{\lee}(p) \subseteq G_{p,A}$ by minimality as   $G_{p,A}$ is type-definable. Then $G_{p,A}$ is an intersection of definable  groups by Proposition \ref{prop: type-def grp is intersec of def grps}.
\end{proof}

We note however that we cannot answer Question \ref{question: chain conditions}(3) positively by proving a chain condition for groups containing $p$ and (type-)definable with arbitrary parameters:
\begin{example}\label{ex: ACVF}
	Let $\cU = (K,\Gamma,k)$ be a monster model of the theory ACVF of algebraically closed fields, and let $G := \left( K, + \right)$ be the additive group. For $b \in K, \beta \in \Gamma$ and $\square \in \{\geq, >\}$,  let $B_{\square  \beta}(b) := \{ x \in K : v(x - b) \square \beta \}$ be the closed (respectively, open) ball of radius $\beta$ with center $b$. 
	Fix $\alpha \in \Gamma$, given a ball $B_{\square  \alpha}(a)$, by quantifier elimination in ACVF, density of $\Gamma$, and the fact that for any two balls either one is contained in the other, or they are disjoint, 
	$$\left\{ x \in B_{\square  \alpha}(a)  \right\} \cup \left\{ x \notin B_{\square  \beta}(b)  : b \in K, \beta \in \Gamma, B_{\square  \beta}(b) \subsetneq B_{\square  \alpha}(a) \right\}$$ determines a complete type $p_{\square \alpha, a} (x) \in S_G(\cU)$, called the \emph{generic type} of $B_{\square  \alpha}(a)$.

 Now let $B := B_{\geq \alpha}(0)$ and $p := p_{\geq \alpha, 0}$. As for any $a,b$ we have $a + B_{\square  \beta}(b) = B_{\square  \beta}(a+ b)$, hence for any $b \in B$ we have $b + B = B$ and if $B' \subsetneq B$, then $b + B' \subseteq B$ for any ball $B'$. And if $b \in K \setminus B$ then $(b + B ) \cap B = \emptyset$,  so we have $\Stab_{\lee}(p) = B $, and $B \leq G(\cU)$. In particular $p \vdash B$, and $p$ is left-$B$-invariant. Note that $p$ generically stable over any $a \in K$ with $v(a) = \alpha$ (as it is internal to the residue field, or can see directly that it commutes with itself).
Hence $p$ is idempotent. 

Note that $ \{ B_{>\beta}(0) : \beta \in \Gamma, \beta  < \alpha\}$ is a large strictly descending chain of definable subgroups of $G(\cU)$ with $p \vdash  B_{>\beta}(0)$ for all $\beta < \alpha$.
	So we have arbitrary long descending chains of definable subgroups of $G(\cU)$ containing $p$ (but the smallest one $B$ is still definable).
\end{example}

\section{Idempotent generically stable measures}\label{sec: idemp gen stab measures}
\subsection{Overview}

The main aim of this section is to prove the following:
\begin{theorem*}(Theorem \ref{prop:main})
	Let $G(x)$ be an abelian type-definable group, and assume that $\mu \in \mathfrak{M}_G(\cU)$ is \fim\ (Definition \ref{def:fim}) and idempotent. Then $\mu$ is the unique left-invariant (and the unique right-invariant) measure on a type-definable subgroup of $G(\cU)$ (namely, its stabilizer).
\end{theorem*}

\begin{remark}
	In particular, if $T$ is NIP and $G$ is abelian, there is a one-to-one correspondence between generically stable idempotent measures and type-definable fsg subgroups of $G$ (under NIP, generically stable = \fim). 
\end{remark}

\begin{remark}
The assumption that $\mu$ is \fim\ cannot be relaxed. 

\noindent Indeed, consider $G := S^1 \times \mathbb{R}$ as a definable group in $(\mathbb{R};+,\times,0,1)$, and let $G(\cU) = \mathcal{S}^{1} \times \mathcal{R}$ be a monster model. Let $\lambda$ be the normalized Haar measure on $S^{1}$ and $p$ the type of the cut above $0$. Let $\lambda'$ be the unique smooth extension of $\lambda$ to the monster model, $p'$ the unique definable extension of $p$, and $\mu := \lambda' \times p'$. Then $\mu$ is definable and  idempotent. But  $\Stab(\mu) = \{(\alpha,0): \alpha \in \mathcal{S}^{1}\}$, in particular $\mu([\Stab(\mu)]) = 0$, so $\mu$ cannot satisfy the conclusion of the theorem by Remark \ref{rem: gen trans implies fim} (see also \cite[Example 4.5]{chernikov2023definable}).
\end{remark}

This section of the paper is organized as follows. We briefly recall the setting and some properties of Keisler measures in Section \ref{sec: measures setting}. In Section \ref{sec: fim measures} we recall some basic facts and make some new observations involving \fim\ measures --- they provide a generalization of generically stable measures from NIP to arbitrary theories in the same way as generically stable types in the sense of \cite{PiTa} provide a generalization from NIP to arbitrary theories. In Section \ref{sec: fim meas and randomization} 
 we prove that the usual characterization of generic stability --- any Morley sequence of a \fim\  measure determines the measure of arbitrary formulas by averaging along it --- holds even when the parameters of these formulas are themselves replaced by a measure, see Theorem \ref{thm: unif gen stab meas}.
In Section \ref{sec: def pushforwards} we collect some basic facts about definable pushforwards of Keisler measures. In Section \ref{sec: fim groups} we develop some theory of \fim\ groups (i.e.~groups admitting a translation invariant \fim\ measure) in arbitrary theories, simultaneously generalizing from fsg groups in NIP theories and generically stable groups in arbitrary theories (in particular, that this translation invariant measure is unique and bi-invariant, Proposition \ref{prop:unique}). In Section \ref{sec: gen trans for meas} we develop an appropriate analog of generic transitivity for \fim\ measures, generalizing some of the results for generically stable types from  Section \ref{sec: gen trans}. Finally, in Section \ref{sec: idemp meas in ab proof}, we put all of this together in order to prove the main theorem of the section, adapting the weight argument from Section \ref{sec: abelian for types} to a purely measure theoretic context. 

In Section \ref{sec: support trans} we isolate a weaker property of {\em support transitivity} and connect it to the algebraic properties of the semigroup induced by $\ast$ on the support of an idempotent measure. In Section \ref{sec: stable measures} we illustrate how the Keisler randomization can be used to reduce generic transitivity to support transitivity in stable groups, and discuss more general situations when the randomization may have an appropriate stratified rank.

\subsection{Setting and notation} \label{sec: measures setting}
We work in the same setting as in Section \ref{sec: setting types}.
For a partitioned formula $\varphi(x,y)$, $\varphi^{*}(y,x)$ is the same formula but with the roles of the variables swapped. In the group setting, if $\varphi(x,y) \vdash \varphi_0(x) \land \varphi_0(y)$, then $\varphi'(x,y) := \varphi(x \cdot y)$. 
Let $A \subseteq \mathcal{U}$. Then a Keisler measure (in variable $x$ over $A$) is a finitely additive probability measure on $\mathcal{L}_{x}(A)$ (modulo logical equivalence). We denote the collection of Keisler measures (in variable $x$ over $A$) as $\mathfrak{M}_{x}(A)$. Given $\mu \in \mathfrak{M}_{x}(A)$, we let $S(\mu)$ denote the support of $\mu$, i.e.~the (closed) set of all $p \in S_x(A)$ such that $\mu(\varphi(x))>0$ for every $\varphi(x) \in p$. We refer the reader to \cite{chernikov2022definable,chernikov2023definable} for basic definitions involving Keisler measures (e.g.~Borel-definable, definable, support of a measure, etc). 
Given a partial type $\pi(x)$ over $\cU$, we will consider the closed set $\mathfrak{M}_\pi(\cU) := \left\{ \mu \in \mathfrak{M}_{x}(\cU) : p \in S(\mu) \Rightarrow p(x) \vdash \pi(x) \right\}$ of measures supported on $\pi(x)$.

We will assume some familiarity with the basic theory of Keisler measures and their basic properties such as (automorphism-) invariance, (Borel-) definability, finite satisfiability, etc., and refer to e.g.~\cite{Guide} or earlier papers in the series \cite{chernikov2022definable, chernikov2023definable} for the details and references.

\subsection{Fim measures}\label{sec: fim measures} Throughout this section we work in an arbitrary theory $T$, unless explicitly specified otherwise.

\begin{definition} Let $\mu \in \mathfrak{M}_{x}(\mathcal{U})$, $\nu \in \mathfrak{M}_{y}(\mathcal{U})$ and suppose that $\mu$ is Borel-definable. Then we define the Morley product of $\mu$ and $\nu$, denoted $\mu \otimes \nu$, as the unique measure in $\mathfrak{M}_{xy}(\mathcal{U})$ such that for any $\varphi(x,y) \in \mathcal{L}_{xy}(\mathcal{U})$, we have 
\begin{equation*} 
(\mu \otimes \nu)(\varphi(x,y)) = \int_{S_{y}(A)} F_{\mu,A}^{\varphi} d(\widehat{\nu|_{A}}), 
\end{equation*} 
where: 
\begin{enumerate} 
\item $\mu$ is $A$-invariant and $A$ contains all the parameters from $\varphi$, 
\item $F_{\mu,A}^{\varphi}: S_{y}(A) \to [0,1]$ is defined by $F_{\mu,A}^{\varphi}(q) = \mu(\varphi(x,b))$ for some (equivalently, any) $b \models q$ in $\cU$, 
\item $\widehat{\nu|_{A}}$ is the unique regular Borel probability measure on $S_{x}(A)$ corresponding to the Keisler measure $\nu|_{A}$. 
\end{enumerate} 
See e.g.~\cite[Section 3.1]{chernikov2022definable} for an explanation why this product is well-defined and its basic properties. We will often abuse the notation slightly and replace $\widehat{\nu|_{A}}$ with either $\nu|_{A}$ or simply $\nu$ when it is clear from the context, and sometimes write $F_{\mu,A}^{\varphi}$ as $F_{\mu}^{\varphi}$. 
\end{definition} 

\begin{definition} Let $\mu \in \mathfrak{M}_{x}(\mathcal{U})$ and suppose that $\mu$ is Borel-definable. Then we define $\mu^{(1)} := \mu(x_1)$, $\mu^{(n+1)}(x_1,\ldots,x_{n+1}) := \mu(x_{n+1}) \otimes \mu^{(n)}(x_1,\ldots,x_n)$, and $\mu^{(\omega)} = \bigcup_{n < \omega} \mu^{(n)}(x_1,\ldots,x_n)$. (In general, $\otimes$ need not be commutative/associative on Borel definable measures in arbitrary theories.)
\end{definition} 

\begin{definition}\label{def:fim} \cite{NIP3}
Let $\mu \in \mathfrak{M}_x(\cU)$ and $M \prec \mathcal{U}$  a small model. A Borel-definable measure $\mu$ is \emph{fim} (a \emph{frequency interpretation measure}) over $M$ if $\mu$ is $M$-invariant and for any $\mathcal{L}$-formula $\varphi(x,y)$ there exists a sequence of formulas $(\theta_{n}(x_1,\ldots,x_n))_{1 \leq n < \omega}$ in $\mathcal{L}(M)$ such that: 
	\begin{enumerate}
		\item for any $\varepsilon > 0$, there exists some $n_{\varepsilon} \in \omega$ satisfying: for any $k \geq n_{\varepsilon}$, if $\mathcal{U} \models \theta_{k}(\abar)$ then 
		\begin{equation*} 
		\sup_{b \in \mathcal{U}^{y}} |\Av(\abar)(\varphi(x,b)) - \mu(\varphi(x,b))| < \varepsilon;
		\end{equation*} 
		\item $\lim_{n \to \infty} \mu^{(n)} \left( \theta_n \left(\bar{x} \right) \right) = 1$.
	\end{enumerate}
	
We say that $\mu$ is \fim\ if $\mu$ is \fim\ over some small $M \prec \cU$. 
\end{definition} 

\begin{remark}\label{rem: fim in NIP is gen stab}
	In NIP theories, \fim\ is equivalent to 	
	each of the following two properties for measures: \emph{dfs} (definable and finitely satisfiable) and \emph{fam} (finitely approximable \cite{chernikov2021definable5}), recovering the usual notion of generic stability for Keisler measures  \cite{NIP3}. Outside of the NIP context, \fim\ (properly) implies \emph{fam} over a model, which in turn (properly) implies \emph{dfs} (see \cite{CG, CGH}). 
\end{remark}
Generalizing from generically stable measures in NIP, one has:
\begin{fact}\cite{CGH}\label{fac: fim commutes}
	If $\mu \in \mathfrak{M}_{x}(\cU)$ is \fim\ and $\nu \in \mathfrak{M}_{y}(\cU)$ is Borel definable, then $\mu \otimes \nu = \nu \otimes \mu$.
\end{fact}

\begin{definition}\cite{CGH2} Let $\mu \in \mathfrak{M}_{x}(\mathcal{U})$ and $M \prec \cU$ a small submodel such that $\mu$ is $M$-invariant, and $\mathbf{x} = (x_i)_{i < \omega}$. We say that $\mu$ is \emph{self-averaging} over $M$ if for any measure $\lambda \in \mathfrak{M}_{\mathbf{x}}(\mathcal{U})$ with $\lambda|_{M} = \mu^{(\omega)}|_{M}$ and any formula $\varphi(x) \in \mathcal{L}_{x}(\mathcal{U})$ we have 
\begin{equation*} 
\lim_{i \to \infty} \lambda(\varphi(x_i)) = \mu(\varphi(x)). 
\end{equation*} 
\end{definition} 

The following generalizes a standard characterization of generically stable measures in NIP theories \cite{NIP3} to \fim\ measures in arbitrary theories (and demonstrates in particular that if $p \in S_x(\cU)$ is \fim \, viewed as a Keisler measure, then it is generically stable in the sense of Section \ref{sec: gen stab types}; indeed, $p$ is generically stable over $M$ if for every Morley sequence $(a_i)_{i <\omega}$ in $p$ over $M$ we have $\lim_{i \to \omega} \tp(a_i/\mathcal{U}) = p$ --- and if $\delta_{p}$ is self-averaging, this property holds, see \cite[Proposition 3.2]{CG}):
\begin{fact}\cite[Theorem 2.7]{CGH2}\label{fact:self} If $\mu \in \mathfrak{M}_{x}(\mathcal{U})$ and $\mu$ is \fim\ over $M$, then $\mu$ is self-averaging over $M$. 
\end{fact} 

The following  would be a natural generalization of stationarity for generically stable types to measures: 
\begin{conjecture}\label{conj: stat for fim measures} Let $\mu \in \mathfrak{M}_x(\cU)$ be \fim\ over $M \prec \cU$ and let $A$ be a small set with $M \subseteq A \subseteq \cU$. If $\nu \in \mathfrak{M}_{x}(\mathcal{U})$ is $A$-invariant, Borel-definable, and $\nu|_{A} = \mu|_{A}$, then $\mu = \nu$. 
\end{conjecture} 

Conjecture \ref{conj: stat for fim measures} is known to hold when $\mu = p$ is a type and $T$ is an arbitrary theory (by Fact \ref{fac: basic gen stab}(2)) and when $\mu$ is a measure but $T$ is NIP (by \cite[Proposition 3.3]{NIP3}).
The following proposition is a special case of Conjecture \ref{conj: stat for fim measures} sufficient for our purposes here.

\begin{proposition}\label{prop: FIM unique inv ext} Let $\mu \in \mathfrak{M}_x(\cU)$ be \fim\ over $M$ and $A \supseteq M$. Suppose $\nu \in \mathfrak{M}_{x}(\mathcal{U})$ is $A$-invariant and $\nu|_{A} = \mu|_{A}$. If either of the following holds:
\begin{enumerate} 
\item $\nu$ is definable, 
\item or the measures $\mu^{(n)}$ are \fim\ for each $n \geq 1$ and $\nu$ is Borel-definable;
\end{enumerate} 
then $\mu = \nu$. 
\end{proposition}
\begin{proof}

The only difference between the two cases is the justification of equation $(b)$ below. Suppose that we are given some $\nu$ with the described properties. Since $\nu$ is $A$-invariant and (Borel-)definable, it follows that $\nu$ is (Borel-)definable over $A$. 

\begin{clm*}
We have that $\nu^{(\omega)}|_{A} = \mu^{(\omega)}|_{A}$ (in either of the two cases). 
\end{clm*}

\begin{clmproof}
	By assumption $\nu^{(1)}|_{A} = \mu^{(1)}|_{A}$. Assume that we have already established $\nu^{(n)}|_{A} = \mu^{(n)}|_{A}$ for some $n \in \omega$. Fix an arbitrary formula $\theta(x_1,\ldots,x_{n+1}) \in \mathcal{L}_{x_1, \ldots, x_{n+1}}(A)$ and $\varepsilon > 0$. Let $\rho(x_{n+1};x_1,\ldots,x_n) : = \theta(x_1,\ldots,x_{n+1})$. Since the measure $\mu^{(n)}(x_1,\ldots,x_n)$ is definable over $A$, there exist formulas $\psi_{1}(x_{n+1},\abar), \ldots$, $\psi_{m}(x_{n+1},\abar) \in \mathcal{L}_{x_{n+1}}(A)$ and $r_1,\ldots,r_m \in [0,1]$ such that 
 \begin{equation*} 
 \sup_{q \in S_{x_{n+1}}(A)} \left \lvert F_{\mu^{(n)},A}^{\rho^{*}}(q) - \sum_{i=1}^{m}r_i \mathbf{1}_{\psi_i(x_{n+1},\abar)}(q) \right \rvert \leq \varepsilon. \tag{$\dagger$}
 \end{equation*} 
Then we have:
\begin{gather*}
	\nu^{(n+1)}(\theta(x_1,\ldots,x_{n+1})) = \left( \nu_{x_{n+1}} \otimes \nu^{(n)}_{x_1, \ldots, x_n} \right)\left(\theta(x_1,\ldots,x_{n+1}) \right) \\
	= \int_{S_{x_1,\ldots, x_{n}}(A)} F_{\nu_{x_{n+1}},A}^{\rho} \ \mathrm{d} \left(\nu^{(n)}|_{A} \right) 
\overset{(a)}{=}  \int_{S_{x_1, \ldots, x_{n}}(A)} F_{\nu_{x_{n+1}},A}^{\rho} \  \mathrm{d}\left(\mu^{(n)}|_{A} \right) \\
\overset{(b)}{=}  \int_{S_{x_{n+1}}(A)} F_{\mu^{(n)},A}^{\rho^{*}}  \  \textrm{d}\left(\nu|_{A} \right) 
\overset{(c)}{\approx}_{\varepsilon}  \int_{S_{x_{n+1}}(A)} \sum_{i=1}^{m} r_i \mathbf{1}_{\psi_i(x_{n+1},\bar{a})} \   \textrm{d}\left(\nu|_{A} \right)\\
= \sum_{i=1}^{m} r_i \nu(\psi_i(x_{n+1},\bar{a})) 
\overset{(d)}{=} \sum_{i=1}^{m} r_i \mu(\psi_i(x_{n+1},\bar{a}))  \\
=  \int_{S_{x_{n+1}}(A)} \sum_{i=1}^{m} r_i \mathbf{1}_{\psi_i(x_{n+1},\bar{a})}  \ \textrm{d}\left(\mu|_{A} \right) 
\overset{(c)}{\approx}_{\varepsilon} \int_{S_{x_{n+1}}(A)} F_{\mu^{(n)},A}^{\rho^{*}} \  \textrm{d}\left(\mu|_{A} \right) \\
= \left( \mu^{(n)}_{x_1,\ldots ,x_n} \otimes \mu_{x_{n+1}} \right)(\theta(x_1,\ldots, x_{n+1}))  
\overset{(e)}{=}  \left( \mu_{x_{n+1}} \otimes \mu^{(n)}_{x_1,\ldots ,x_n} \right) (\theta(x_1,\ldots, x_{n+1})) \\
= \mu^{(n+1)}(\theta(x_1,\ldots ,x_{n+1})),
\end{gather*}
with the following justifications for the corresponding steps: 
 \begin{enumerate}[(a)]
 \item induction hypothesis; 
 \item in Case (1), $\mu^{(n)}$ is fam over $A$ since $\mu^{(n)}$ is fam over $M$ and $M \subseteq A$ (and \fim\  implies fam over a model), by \cite[Proposition 2.10(b)]{CG}, $\nu$ is definable over $A$, and fam measures commute with definable measures \cite[Proposition 5.17]{CGH}; in Case (2), $\mu^{(n)}$ is \fim\ over $A$, $\nu$ is Borel-definable over $A$, and \fim\ measures commute with Borel-definable measures \cite[Proposition 5.15]{CGH}; 
 \item by $(\dagger)$;
 \item by assumption; 
 \item $\mu_{x_{n+1}}$ is \fim, and \fim\ measures commute with all Borel-definable measures \cite[Proposition 5.15]{CGH} (alternatively, fam measures commute with definable measures). 
 \end{enumerate}  
As $\theta$ and $\varepsilon$ were arbitrary, we conclude $\nu^{(n+1)}|A = \mu^{(n+1)}|_{A}$. And then $\nu^{(\omega)}|_{A} = \bigcup_{n < \omega} \nu^{(n)} = \bigcup_{n < \omega} \mu^{(n)} =  \mu^{(\omega)}|_{A} $.
\end{clmproof}

Let now $\varphi(x) \in \cL_x(\cU)$ be arbitrary.
 Let $\lambda(\mathbf{x}) := \nu^{(\omega)}(\mathbf{x})$, by the claim above we have $\lambda|_{A} = \mu^{(\omega)}|_{A}$.
 Note that for every $i \in \omega$, $\lambda(\varphi(x_i)) = \nu(\varphi(x))$. As $\mu$ is \fim\ over $M$, it is self-averaging over $M$ by Fact \ref{fact:self}, hence $\nu(\varphi(x)) = \lim_{i \to \infty} \lambda(\varphi(x_i)) = \mu(\varphi(x))$.
\end{proof}

\begin{remark}
	 Our proof of Proposition \ref{prop: FIM unique inv ext} does not apply to the general case of Conjecture \ref{conj: stat for fim measures} since it is open whether or not $\mu$ being \fim\ implies that $\mu^{(n)}$ is \fim\ for $n \geq 2$ in an arbitrary theory, even when $\mu=p$ is a generically stable type.
	 \end{remark}
	 
	 \subsection{Fim measures over ``random'' parameters}\label{sec: fim meas and randomization}
In this section we prove a generalization of Fact \ref{fact:self} of independent interest,  demonstrating that any Morley sequence of a \fim\  measure determines the measure of arbitrary formulas by averaging along it --- even when the parameters of these formulas are allowed to be ``random''. More precisely:
\begin{theorem}\label{thm: unif gen stab meas}Let $\mu \in \mathfrak{M}_{x}(\mathcal{U})$ be \fim\ over $M$, $\nu \in \mathfrak{M}_{y}(\mathcal{U})$, $\varphi(x,y,z) \in \mathcal{L}_{xyz}$, $b \in \mathcal{U}^{z}$, and $\mathbf{x} = (x_i)_{i \in \omega}$. Suppose that $\lambda \in \mathfrak{M}_{\mathbf{x}y}(\mathcal{U})$ is arbitrary such that $\lambda|_{\mathbf{x},M}  = \mu^{(\omega)}$ and $\lambda|_{y} = \nu$. Then 
\begin{equation*} 
\lim_{i \to \infty} \lambda(\varphi(x_i,y,b)) = \mu \otimes \nu(\varphi(x,y,b)). 
\end{equation*} Moreover for every $\varepsilon > 0$ there exists $n = n(\mu, \varphi, \varepsilon) \in \mathbb{N}$ so that for any $\nu, \lambda,b$ as above, we have $\lambda(\varphi(x_i,y,b)) \approx^{\varepsilon} \mu \otimes \nu( \varphi(x,y,b))$ for all but $n$ many $i \in \mathbb{N}$.
\end{theorem} 

%In this section we prove a general theorem about the randomization of \fim\ measures. In particular,  we extend one of the main results from \cite{CGH2}. There, the second author (along with Conant and Hanson) showed that if $\mu$ is \fim\ then $\mu$ is self-averaging. Here, we show that $\mu$ satisfies a more general condition; Let $\mathbf{x} = (x_i)_{i \geq 1}^{\omega}$ where $|x_i| = |x_j|$ for all $i,j \geq 1$. If $\lambda \in \mathfrak{M}_{\mathbf{x}y}(\mathcal{U})$, $\varphi(x,y) \in \mathcal{L}_{xy}(\mathcal{U})$, and $\lambda|_{\mathbf{x},M} = \mu^{(\omega)}|_{M}$, then 
%\begin{equation*}
%\lim_{i \to \infty} \lambda(\varphi(x_i,y)) = \mu \otimes \pi_{y}(\lambda)(\varphi(x,y)),
%\end{equation*} 
%where $\pi_{y}$ is the obvious restriction map from $\mathfrak{M}_{\mathbf{x}y}(\mathcal{U}) \to \mathfrak{M}_{y}(\mathcal{U})$. 
\begin{remark}
Fact \ref{fact:self} corresponds to the special case when $\nu(y) = q(y)$ is a type. We note that this result is central to the proof of our main theorem, and is new even for NIP theories.
\end{remark}

Our proof of Theorem \ref{thm: unif gen stab meas} relies on the use of Keisler randomization in continuous logic, as introduced and studied in \cite{ben2009randomizations}. We will follow the notation from \cite[Section 3.2]{CGH2}. Let $T$ be a complete first order theory. We let $T^{R}$ denote the (continuous) first order theory of its Keisler randomization (we refer to \cite[Section 2]{ben2009randomizations} for the details). Let $M$ be a model of $T$ and let $(\Omega, \mathcal{B},\mathbb{P})$ be a probability algebra. We consider the model $M^{(\Omega,\mathcal{B},\mathbb{P})}$ of $T^{R}$, which we usually denote as $M^{\Omega}$ for brevity, defined as follows. We let 
\begin{equation*} 
M'_{0} := \left\{f: \Omega \to M: f \text{ is } \mathcal{B}\text{-measurable}, |\im(f)| < \aleph_0\right\},
\end{equation*} 
equipped with the pseudo-metric $d(f,g) := \mathbb{P} \left( \{ \omega \in \Omega: f(\omega) \neq g(\omega) \} \right)$. Then $M^{\Omega}$ is constructed by taking the metric completion of $M'_0$ and then identifying random variables up to $\mathbb{P}$-measure $0$. We let $M^{\Omega}_0$ be the set of classes of elements of $M_0'$. By construction, $M^{\Omega}_0$ is a metrically dense (pre-)substructure of $M^{\Omega}$. 

Let $\mathcal{U}$ be a monster model of $T$ such that $M \prec \mathcal{U}$. The model $\mathcal{U}^{\Omega}$  is almost never saturated, so we will always think of $\mathcal{U}^{\Omega}$ as (elementarily) embedded into a monster model $\mathcal{C}$ of $T^{R}$, i.e.~$\mathcal{U}^{\Omega} \prec \mathcal{C}$. If $a \in \mathcal{U}$, we let $f_{a} \in \mathcal{U}^{\Omega}_0$ denote the constant random variable taking value $a$, i.e.~$f_{a}$ is the equivalence class of the maps which send $\Omega$ to the point $a$ (equivalence up to measure $0$).  If $A \subseteq \mathcal{U}$, we let $A^{c} := \left\{f_{a}: a \in A\right\}\subseteq \mathcal{U}^{\Omega}_0$. If $\varphi(x_1,\ldots ,x_n)$ is an $\mathcal{L}$-formula, we let $\mathbb{E}[\varphi(x_1,\ldots ,x_n)]$ denote the corresponding continuous  formula in the randomization. This formula is evaluated on tuples of elements $\bar{h} = (h_1,\ldots ,h_n)$ from $\mathcal{U}^{\Omega}_{0}$ via 

\begin{equation*} 
\mathbb{E} \left[\varphi(\bar{h}) \right] = \mathbb{P} \left(\{\omega \in \Omega: \mathcal{U} \models \varphi(h_1(\omega),\ldots ,h_n(\omega))\} \right), 
\end{equation*} 
and is extended to $\mathcal{U}^{\Omega}$ via uniform limits. For $B \subseteq \mathcal{C}$, $S^{R}_{x}(B)$ will denote the space of types in the tuple of variables $x$ over $B$ in $T^{R}$. 

\begin{remark} Note that for any $\bar{h} = (h_1,\ldots ,h_n) \in (\cU^{\Omega}_{0})^{n}$, there exists a finite $\mathcal{B}$-measurable partition $\mathcal{A}$ of $\Omega$ with the property that for each $i \leq n$ the function $h_i$ is constant on each element of $\mathcal{A}$. Given such an $\bar{h}$ and $\mathcal{A}$, we write $\bar{h}|_{A}$ for the tuple of constant values of the functions in $\bar{h}$ on the set $A$. Note that for each $A \in \mathcal{A}$, $\bar{h}|_{A}$ is an element of $\mathcal{U}^{n}$. 
\end{remark} 

\noindent The following fact can be derived from basic facts about continuous logic: 
\begin{fact}\label{random:1} Suppose that $p \in S_{x}^{R}(\mathcal{U}^{\Omega})$. Then there exists a net of tuples $(h_i)_{i \in I}$ where $h_i \in (\mathcal{U}^{\Omega}_0)^{x}$ such that 
$ \lim_{i \in I} \tp^{R}(h_i/\mathcal{U}^{\Omega}) = p$. 

\end{fact} 

The following observations were made by Ben Yaacov in an unpublished note \cite{BEN}. For a detailed verification, we refer the reader to \cite[Section 3.2]{CGH2}.  

\begin{fact}\label{fact:building1} Let $\mathcal{U}$ be a monster model of $T$, $\mu \in \mathfrak{M}_{x}(\mathcal{U})$, and $\mathcal{U}^{\Omega} \prec \mathcal{C}$. 
\begin{enumerate} 
\item There exists a unique type $p_{\mu} \in S_{x}^{R}(\mathcal{U}^{\Omega})$ such that for any $\mathcal{L}$-formula $\varphi(x,\bar{y})$, $\bar{h} = (h_1,\ldots, h_n) \in \mathcal{U}^{\Omega}_0$, and any measurable partition $\mathcal{A}$ such that each element of $\bar{h}$ is constant on each element of $\mathcal{A}$,   
\begin{equation*} 
(\mathbb{E}[\varphi(x, \bar{h})])^{p_{\mu}} = \sum_{A \in \mathcal{A}} \mathbb{P}(A)\mu(\varphi(x,\bar{h}|_{A})). 
\end{equation*} 
\item If $\mu$ is definable, then there exists a unique type $r_{\mu} \in S_{x}^{R}(\mathcal{C})$ such that:
\begin{enumerate} 
\item $r_{\mu}|_{\mathcal{U}^{\Omega}} = p_{\mu}$; 
\item $r_{\mu}$ is definable over ${\mathcal{U}^{\Omega}}$; if $\mu$ is $M$-definable, then $r_{\mu}$ is $M^{\Omega}$-definable. 
\end{enumerate} 
\end{enumerate} 
\end{fact} 

\begin{remark} The claims in Fact \ref{fact:building1} hold in the context where $x$ is an infinite tuple of variables. The infinitary results follow easily from their finite counterparts.
\end{remark} 

\begin{corollary}\label{Corollary:restriction} Suppose that $\mathbf{x} = (x_i)_{i < \alpha}$, $\mathbf{y} = (y_i)_{i < \beta}$, and $\lambda \in \mathfrak{M}_{\mathbf{xy}}(\mathcal{U})$. Then 
$$\left( p_{\lambda} \right)|_{\mathbf{x}, M^{\Omega}} = p_{ \left( \lambda|_{\mathbf{x}} \right)}|_{M^{\Omega}} \textrm{ and } \left( p_{\lambda} \right)|_{\mathbf{y}} = p_{ \left( \lambda|_{\mathbf{y}} \right)}.$$
\end{corollary} 

\begin{proof} Fix $\bar{x} = \left( x_{i_{1}},\ldots ,x_{i_{n}} \right)$ and $\overline{h} := \left( h_1,\ldots ,h_m \right) \in \left(M^{\Omega} \right)^{m}$. Fix a finite measurable partition $\mathcal{A}$ of $\Omega$ such that each element of $\overline{h}$ is constant on each element of $\mathcal{A}$. From the definitions we have: 
\begin{align*}
\left(\mathbb{E} \left[\varphi \left(\bar{x},\overline{h} \right) \right] \right)^{p_{\lambda}} &= \sum_{A \in \mathcal{A}} \mathbb{P}(A) \lambda(\varphi(\bar{x},\overline{h}|_{A})) \\
 &= \sum_{A \in \mathcal{A}} \mathbb{P}(A) \left( \lambda|_{\mathbf{x}} \right)(\varphi(\bar{x},\overline{h}|_{A})) \\
&= (\mathbb{E}[\varphi(\bar{x},\overline{h})])^{p_{ \left(\lambda|_{\mathbf{x}} \right)}}.
\end{align*} 
By quantifier elimination in $T^{R}$, we conclude that $p_{\lambda}|_{\mathbf{x}, M^{\Omega}} = p_{\left(\lambda_{\mathbf{x}} \right)}|_{M^{\Omega}}$.
\end{proof} 

We recall some results from \cite{CGH2} connecting the randomized measures, the Morley product and generic stability in $T$ and $T^{R}$. These are \cite[Proposition 3.15]{CGH2}, \cite[Corollary 3.16]{CGH2} and \cite[Corollary 3.19]{CGH2}, respectively. 

\begin{fact}\label{CGH2:Facts} Suppose $\mu \in \mathfrak{M}_{x}(\mathcal{U})$ and $\nu \in \mathfrak{M}_{y}(\mathcal{U})$. 
\begin{enumerate} 
\item If $\mu$ and $\nu$ are definable, then 
\begin{equation*}
r_{\mu \otimes \nu}(x,y) = r_{\mu}(x) \otimes r_{\nu}(y). 
\end{equation*} 
\item If $\mu$ is definable, then for every $n \geq 1$, 
\begin{equation*} 
r_{\mu^{(n)}}(\bar{x}) = (r_{\mu})^{(n)}(\bar{x}).
\end{equation*} 
\item If $\mu$ is \fim, then $r_{\mu}$ is generically stable over $M^{\Omega}$ (for generically stable types in continuous logic we refer to \cite{khanaki2022generic, CGH2, anderson2023generically}). 
\end{enumerate} 
\end{fact}

\begin{corollary}\label{Corollary:otimes} If $\mu\in \mathfrak{M}_{x}(\mathcal{U})$ is a definable measure then $r_{\mu^{(\omega)}} = (r_{\mu})^{(\omega)}$.
\end{corollary} 

\begin{proof} First note 

\begin{equation*} 
(r_{\mu})^{(\omega)} = \bigcup_{1 \leq n <  \omega} r_{\mu}^{(n)} = \bigcup_{1 \leq n <  \omega} r_{\mu^{(n)}}.  
\end{equation*} 
We want to show that $\bigcup_{n <  \omega} r_{\mu^{(n)}} = r_{\mu^{(\omega)}}$. By quantifier elimination in $T^{R}$, it suffices to show that for every $\mathcal{L}_{\bar{x},y}$-formula $\varphi(x_1,\ldots ,x_k,y)$  and $b \in \mathcal{C}^{y}$, we have that $ \mathbb{E}[\varphi(\bar{x},b)]^{r_{\mu^{(\omega)}}} = \mathbb{E}[\varphi(\bar{x},b)]^{r_{\mu^{(k)}}}$. By Fact \ref{random:1}, fix a net  $(h_i)_{i \in I}$ of elements each in $(\mathcal{U}^{\Omega}_0)^{y}$ such that $\lim_{i \in I} \tp^{R}(h_i/\mathcal{U}^{\Omega}) = \tp^{R}(b/\mathcal{U}^{\Omega})$. For each $i \in I$, choose a finite measurable partition $\mathcal{A}_i$ of $\Omega$ such that each element of $h_i$ is constant on each element of $\mathcal{A}_i$. We have the following computation (using Fact \ref{fact:building1}):
\begin{gather*}
	\mathbb{E}[\varphi(\bar{x},b)]^{r_{\mu^{(\omega)}}} = \lim_{i \in I}F_{r_{\mu^{(\omega)}}}^{\mathbb{E}[\varphi(\bar{x},y)]}\left((\tp^{R}(h_i/\mathcal{U}^{\Omega}) \right) \\
= \lim_{i \in I} \mathbb{E}[\varphi(\bar{x},h_i)]^{p_{\mu^{(\omega)}}} 
= \lim_{i \in I} \sum_{A \in \mathcal{A}_i} \mathbb{P}(A) \mu^{(\omega)}(\varphi(\bar{x},h_i|_{A}))\\
= \lim_{i \in I} \sum_{A \in \mathcal{A}_i} \mathbb{P}(A) \mu^{(k)}(\varphi(\bar{x},h_i|_{A}))
= \lim_{i \in I} (\mathbb{E}[\varphi(\bar{x},h_i)])^{p_{\mu^{(k)}}}\\
= \lim_{i \in I}F_{r_{\mu^{(k)}}}^{\mathbb{E}[\varphi(\bar{x},y)]}\left((\tp^{R}(h_i/\mathcal{U}^{\Omega}) \right) 
= \mathbb{E}[\varphi(\bar{x},b)]^{r_{\mu^{(k)}}},  
\end{gather*}
where the first and last equality follow from the fact that $F_{r_{\mu^{(\omega)}}}^{\mathbb{E}[\varphi(\bar{x},y)]}$ and $F_{r_{\mu^{(k)}}}^{\mathbb{E}[\varphi(\bar{x},y)]}$ are continuous maps,  by definability of $\mu$, hence of $\mu^{(\omega)}$ and $\mu^{(k)}$, and Fact \ref{fact:building1}(2)(b). 
\end{proof} 

\noindent The following fact is \cite[Lemma 3.13]{CGH2}. 
\begin{fact}\label{random:2} Suppose that $\mu \in \mathfrak{M}_{x}(\mathcal{U})$ and $(h_i)_{i \in I}$ is a net of elements such that $h_i \in (\mathcal{U}^{\Omega}_0)^{x}$ and $\lim_{i \in I} \tp^{R}(h_i/\mathcal{U}^{\Omega}) = p_{\mu}$. For each $i \in I$, let $\mathcal{A}_i$ be a finite measurable partition of $\Omega$ such that each element of $h_i$ is constant on each $A \in A_i$. Then 
\begin{equation*} 
\lim_{i \in I} \left( \sum_{A \in \mathcal{A}_i} \mathbb{P}(A)\delta_{(h_i|_{A})} \right) = \mu,
\end{equation*} 
where the limit is calculated in the space $\frak{M}_x(\cU)$.
\end{fact}

\begin{proposition}\label{Main:Proposition} Let $\mu \in \mathfrak{M}_{x}(\mathcal{U})$ be \fim\ over a small model $M \prec \cU$, $\nu \in \mathfrak{M}_{y}(\mathcal{U})$ and $\mathbf{x} = (x_i)_{i \in \omega}$ with $x_i$ of the same sort as $x$ for all $i$. Suppose that $\lambda \in \mathfrak{M}_{\mathbf{x}y}(\mathcal{U})$ such that $\lambda|_{\mathbf{x},M}  = \mu^{(\omega)}|_{M}$ and $\lambda|_{y} = \nu$. Then for any $\varphi(x,y) \in \mathcal{L}$ we have
\begin{equation*} 
\lim_{i \to \infty} \lambda(\varphi(x_i,y)) = \mu \otimes \nu(\varphi(x,y)). 
\end{equation*} 
\end{proposition} 

\begin{proof} The measures $\mu \in \mathfrak{M}_x(\cU), \lambda \in \mathfrak{M}_{\mathbf{x} y}(\mathcal{U})$, $\mu^{(\omega)} \in \mathfrak{M}_{\mathbf{x}}(\mathcal{U})$ and $\nu \in \mathfrak{M}_{y}(\mathcal{U})$  can be associated to  complete types $p_\mu \in S^R_{x}(\cU^{\Omega}), p_{\lambda} \in S^{R}_{\mathbf{x}y}(\mathcal{U}^{\Omega})$, $p_{\mu^{(\omega)}} \in S^{R}_{\mathbf{x}y}(\mathcal{U}^{\Omega})$ and $p_{\nu} \in S^{R}_{y}(\mathcal{U}^{\Omega})$, respectively,  by Fact \ref{fact:building1}(1). As $\mu$ is $M$-definable, $\mu^{(\omega)}$ is also $M$-definable, and so the types $r_{\mu} \in S_{x}(\mathcal{C})$ and $r_{\mu^{(\omega)}} \in S_{\mathbf{x}}(\mathcal{C})$ are well-defined and definable over $M^{\Omega}$ by Fact \ref{fact:building1}(2). We then have: 
\begin{enumerate}
\item $r_{\mu}$ is generically stable over $M^{\Omega}$ (by Fact \ref{CGH2:Facts}(3));
\item $r_{\mu^{(\omega)}}|_{M^{\Omega}} = (r_{\mu})^{(\omega)}|_{M^{\Omega}}$ (by Corollary \ref{Corollary:otimes});
\item $p_{\lambda}|_{\mathbf{x},M^{\Omega}} = p_{\mu^{(\omega)}}|_{M^{\Omega}} =  r_{\mu^{(\omega)}}|_{M^{\Omega}}$ and $p_{\lambda}|_y = p_{\nu}$ (by Corollary \ref{Corollary:restriction}).
\end{enumerate} 
Let $\mathbf{a} = (a_i)_{i \in \omega}$ and $b$ in $\mathcal{C}$ be so that $(\mathbf{a},b) \models p_{\lambda}$. Then $\mathbf{a}$ is a Morley sequence in $r_{\mu}$ over $M^{\Omega}$ by (2) and (3) from above. By Fact \ref{random:1}, choose a net $(h_j)_{j \in J}$ of tuples in $(\mathcal{U}_{0})^{y}$ such that $\lim_{j \in J} \tp^R(h_j/\mathcal{U}^{\Omega}) = \tp^R(b/\mathcal{U}^{\Omega}) = p_{\nu}$. For each $h_j$, choose a finite measurable partition $\mathcal{A}_j$ of $\Omega$ such that each function in $h_j$ is constant on each element of $\mathcal{A}_j$. Now, given any $\varphi(x,y) \in \mathcal{L}$, we then have the following computation: 
\begin{gather*}
	\lim_{i \to \infty} \lambda(\varphi(x_i,y)) = \lim_{i \to \infty} \mathbb{E}[\varphi(x_i,y)]^{p_{\lambda}} = \lim_{i \to \infty} \mathbb{E}[\varphi(a_i,b)] \overset{(a)}{=}\mathbb{E}[\varphi(x,b)]^{r_{\mu}} \\
	= F^{\varphi}_{r_{\mu}} \left( \tp^R(b/\mathcal{U}^{\Omega}) \right) \overset{(b)}{=} \lim_{j \in J} F^{\varphi}_{r_{\mu}} \left( \tp^R(h_j/\mathcal{U}^{\Omega}) \right)
	= \lim_{j \in J} \mathbb{E}[\varphi(x,h_j)]^{p_{\mu}} \\ \overset{(c)}{=} \lim_{j \in J} \sum_{A \in \mathcal{A}_j} \mathbb{P}(A) \mu(\varphi(x,h_{j}|_{A}))
	\overset{(d)}{=} \lim_{j \in J} \int_{S_{y}(\mathcal{U})} F_{\mu}^{\varphi} \  \textrm{d}\left( \sum_{A \in \mathcal{A}_j} \mathbb{P}(A) \delta_{\left( h_j|_{A} \right)}\right) \\
	\overset{(e)}{=} \int_{S_{y}(\mathcal{U})} F_{\mu}^{\varphi} \  \textrm{d} \left( \lim_{j \in J}  \sum_{A \in \mathcal{A}_j} \mathbb{P}(A) \delta_{\left( h_j|_{A} \right)}\right) 
	\overset{(f)}{=} \int_{S_{y}(\mathcal{U})} F_{\mu}^{\varphi} d\nu = \mu \otimes \nu(\varphi(x,y)),
\end{gather*}
where the corresponding equalities hold for the following reasons: 
\begin{enumerate}[$(a)$]
\item  since the type $r_{\mu} \in S_{x}^{R}(\mathcal{C})$ is  generically stable over $M^{\Omega}$ by (1) and 
$(\mathbf{a}_i)_{i \in \omega}$ is a Morley sequence in $r_{\mu}$ over $M^{\Omega}$;
\item by the choice of $(h_j)_{j \in I}$ and, as $r_\mu$ is definable over $\mathcal{U}^{\Omega}$ by (1), the map $F_{r_{\mu}}^{\varphi}: S^R_{y}(\mathcal{U}^{\Omega}) \to [0,1]$ is continuous;
\item by the definition of $p_{\mu}$ (Fact \ref{fact:building1});
\item for a fixed $j \in J$, the computations of the left hand side and the right hand side are the same; 
\item since $\mu$ is definable, the map $F_{\mu}^{\varphi}:S_{y}(\mathcal{U}) \to [0,1]$ is continuous, hence  the map $\gamma \in \mathfrak{M}_{y}(\mathcal{U}) \mapsto \int F_{\mu}^{\varphi} \textrm{d} \gamma \in [0,1]$ is continuous;
\item by Fact \ref{random:2}. \qedhere
\end{enumerate} 
\end{proof} 

 We will use the following general topological fact \cite[Lemma 2.3]{CGH2}: 

\begin{fact}\label{CGH:fact1} Let $f:X \to K$ be an arbitrary function from a compact Hausdorff space to a compact interval $K \subseteq \mathbb{R}$. Suppose there is a closed subset $C \subseteq K^{\omega} \times X$ satisfying the following properties: 
\begin{enumerate} 
\item the projection of $C$ onto $X$ is all of $X$;
\item if $(\alpha,x) \in C$ and $g: \omega \to \omega$ is strictly increasing, then $(\alpha \circ g, x) \in C$; 
\item for any $(\alpha,x) \in C$, $\lim_{i \to \infty} \alpha(i) = f(x)$. 
\end{enumerate} 
Then $f$ is continuous and, for any $\varepsilon > 0$, there is an $n_{\varepsilon} \in \mathbb{N}$ such that: for any $(\alpha,x) \in C$, $\{i \in \omega: \alpha(i) \not \approx_{\varepsilon} f(x)\} \leq n_{\varepsilon}$. 
\end{fact}

Finally, we can derive the main theorem of the section:
\begin{proof}[Proof of Theorem \ref{thm: unif gen stab meas}] 
The ``moreover'' clause of Theorem \ref{thm: unif gen stab meas} for formulas without parameters follows from Proposition \ref{Main:Proposition} by compactness (using Fact \ref{CGH:fact1}). 

Namely, first let $\varphi(x;y,z) \in \cL(\emptyset)$ be arbitrary. We let 
\begin{gather*}
	\mathfrak{M}_L(\cU) := \left\{ \eta \in \mathfrak{M}_{\mathbf{x}yz}(\cU) : \eta|_{\mathbf{x},M} = \mu^{(\omega)}|_{M}\right\}, X := \mathfrak{M}_{yz}\left( \cU \right), K = [0,1],\\
	f:X \to K \textrm{ defined by } \nu \in \mathfrak{M}_{yz}\left( \cU \right) \mapsto \mu \otimes \nu (\varphi(x,y,z)), \textrm{ and}\\
	C:= \left\{ \left(  \left( \eta(\varphi(x_i, y, z) ) : i \in \omega \right), \eta|_{yz} \right)  :  \eta \in \mathfrak{M}_L(\cU) \right\} \subseteq [0,1]^{\omega} \times X.
\end{gather*}
The assumptions of Fact \ref{CGH:fact1} are satisfied. Indeed,  
(1) holds since for every $\nu \in \mathfrak{M}_{yz}\left( \cU \right)$,  $\eta := \mu^{(\omega)} \otimes \nu$ gives an element in $C$ projecting onto it.
(2) For every strictly increasing $g: \omega \to \omega$ we have a continuous map  $g':S_{\mathbf{x}}(\cU) \to S_{\mathbf{x}}(\cU)$ defined by $\varphi(x_1, \ldots, x_n) \in g'(p) \iff \varphi(x_{g(1)}, \ldots, x_{g(n)}) \in p$. Now  
if $\eta|_{\mathbf{x},M} = \mu^{(\omega)}|_{M}$ then still $g_{\ast}(\eta)|_{\mathbf{x},M} = \mu^{(\omega)}|_{M}$, where $g_{\ast}(\eta)$ is the pushforward of the measure $\eta$ by $g$ (see Definition \ref{def: definable pushforward}). And (3) holds by Proposition \ref{Main:Proposition}. Then we obtain the required $n = n(\mu, \varphi, \varepsilon)$ applying Fact \ref{CGH:fact1}.

Now assume we are given $\mu \in \mathfrak{M}_x(\cU)$, $\varepsilon > 0$, and $\psi(x,y) \in \cL(\cU)$ is an arbitrary formula with parameters, say of the form $\varphi(x,y,b)$ for some $b \in \cU^z$ and $\varphi(x,y,z) \in \cL(\emptyset)$. Let $n = n(\mu, \varphi, \varepsilon) \in \omega$ be as given by the above for the formula $\varphi(x,y,z)$ without parameters. 
Given any $\lambda \in \mathfrak{M}_{\mathbf{x}y}(\mathcal{U})$ with $\lambda|_{\mathbf{x},M}  = \mu^{(\omega)}$ and $\lambda|_{y} = \nu$, consider the measures $\lambda_{b} \in \mathfrak{M}_{\mathbf{x}yz}(\mathcal{U})$ defined by $\lambda_{b}(\mathbf{x},y,z) := \lambda(\mathbf{x},y) \otimes \delta_{b}(z)$, and $\nu_{b} \in  \mathfrak{M}_{yz}(\mathcal{U})$ defined by $\nu_{b}(y,z) = \nu(y) \otimes \delta_{b}(z)$. Note that $\lambda_{b}|_{yz} = \nu_{b}$. By the choice of $n$ and the previous paragraph (with $\lambda_{b}$ and $\nu_{b}$ in place of $\lambda$ and $\nu$) we have
\begin{equation*} 
\lim_{i \to \infty} \lambda(\varphi(x_{i},y,b))  = \lim_{i \to \infty} \lambda_{b}(\varphi(x_{i},y,z)) = \mu \otimes \nu_{b}(\varphi(x,y,z)) = \mu \otimes \nu(\varphi(x,y,b)). \qedhere
\end{equation*} 
%[Consider the evaluation maps $E_{\varphi(x_i,y)}: \mathfrak{M}_{\mathbf{x}y}(\mathcal{U}) \to [0,1]$, the projection map $\pi_{\mathbf{x}}: \mathfrak{M}_{\mathbf{x}y}(\mathcal{U}) \to \mathfrak{M}_{\mathbf{x}}(\mathcal{U})$, and the restriction map $r_{M}: \mathfrak{M}_{\mathbf{x}}(\mathcal{U}) \to \mathfrak{M}_{\mathbf{x}}(M)$. Then $A : = \pi_{\mathbf{x}}(r_{M}^{-1}(\{\mu^{(\omega)}|_{M}\}))$ is a compact subset of $\mathfrak{M}_{\mathbf{x}y}(\mathcal{U})$ and $\mathbf{E}: \prod_{i < \omega} E_{\varphi(x_i,y)} \times \pi_{y}: \mathfrak{M}_{\mathbf{x}y}(\mathcal{U}) \to [0,1]^{\omega} \times \mathfrak{M}_{y}(\mathcal{U})$ is continuous. Therefore $\mathbf{E}(A)$ is a compact subset of $[0,1]^{\omega} \times \mathfrak{M}_{y}(\mathcal{U})$ and thus closed. Notice that $\mathbf{E}(A) = C$. Furthermore, 

%\begin{enumerate}
%\item For any $\lambda \in \mathfrak{M}_{\mathbf{x}}(\mathcal{U})$ and $\nu \in \mathfrak{M}_{y}(\mathcal{U})$, we can build an amalgam $\eta \in \mathfrak{M}_{\mathbf{x}y}(\mathcal{U})$ such that $\pi_{\mathbf{x}}(\eta) = \lambda$ and $\pi_{y}(\eta) = \nu$. 
%\item Clear. 
%\item Notice that 
%\begin{equation*} \lim_{i \to \infty} \alpha(i) = \lim_{i \to \infty} \lambda(\varphi(x_i,y)) = \mu \otimes \pi_{y}(\lambda)(\varphi(x,y)) = \mu \otimes \nu(\varphi(x,y)),
%\end{equation*}
%where the second equality holds by the first part. ] \qedhere
%\end{enumerate} 
\end{proof} 

\subsection{Definable pushforwards of Keisler measures}\label{sec: def pushforwards} We record some basic facts about definable pushforwards of Keisler measures. 

\begin{definition}\label{def: definable pushforward}Let $f:\mathcal{U}^{x} \to \mathcal{U}^{y}$ be a definable map. For $\mu \in \mathfrak{M}_{x}(\mathcal{U})$, we define the \emph{push-forward measure} $f_{*}(\mu)$ in $\mathfrak{M}_{y}(\mathcal{U})$, where for any formula $\varphi(y) \in \mathcal{L}_{y}(\mathcal{U}), f_*(\mu)(\varphi(y)) = \mu(\varphi(f(x)))$. 
\end{definition} 

\begin{proposition}\label{prop:push-forward} Let $\mu \in \mathfrak{M}_{x}(\mathcal{U})$. Let $A\subseteq \cU$ be a small set and let $f:\mathcal{U}^{x} \to \mathcal{U}^{y}$ be an $A$-definable map. Then we have the following:
\begin{enumerate} 
\item if $\mu$ is $A$-invariant, then $f_{\ast}(\mu)$ is $A$-invariant;
\item if $\mu$ is $A$-definable, then $f_{\ast}(\mu)$ is $A$-definable;
\item if $\mu$ is \fim\ over $A$, then $f_{\ast}(\mu)$ is \fim\ over $A$;
\item $\left\{ f_{\ast}(p) : p \in S(\mu) \right\} = S \left(f_{\ast}(\mu) \right)$. 
\end{enumerate} 
\end{proposition} 
\begin{proof}
\noindent	(1) Straightforward. 

	\noindent (2) Note that for any formula $\varphi(y,z) \in \mathcal{L}_{yz}$, we have $F_{f_*(\mu),A}^{\varphi}= F_{\mu,A}^{\varphi_{f}}$ where $\varphi_{f}(x,z) := \varphi(f(x),z)$. 
	
	~
	
\noindent 	(3) Let $\mu \in \frak{M}_{x}(\cU)$ be \fim\  over $A$.

\begin{clm*}
For any $\theta(y_1,\ldots ,y_n) \in \mathcal{L}(\mathcal{U})$, 
\begin{equation*} (f_*(\mu))^{(n)}(\theta(y_1,\ldots ,y_n)) = \mu^{(n)}(\theta(f(x_1),\ldots ,f(x_n))). 
\end{equation*} 	
\end{clm*}
\begin{clmproof}  We prove the claim by induction on $n$. The base case $n = 1$ is trivial. Assume the claim holds for $n$. For $1 \leq k < \omega$, let $G_k:(\mathcal{U}^{x})^{k} \to (\mathcal{U}^{y})^{k}$ be the $A$-definable map given by $G_k(x_1,\ldots ,x_k) = (f(x_1),\ldots ,f(x_k))$. Note that $G_k$ induces a pushforward from $\mathfrak{M}_{x_1, \ldots,  x_k}(\mathcal{U})$ to $\mathfrak{M}_{y_1, \ldots, y_k}(\mathcal{U})$. Then our induction hypothesis says  $(f_*(\mu))^{(n)} = \left( G_{n} \right)_{*} \left(\mu^{(n)} \right)$. Fix $\theta(y_1,\ldots ,y_{n+1} ) \in \mathcal{L}(\mathcal{U})$ and let 
\begin{gather*}
	\psi(y_{n+1};y_1,\ldots ,y_n):= \theta(y_1,\ldots ,y_{n+1} ),\\
	\psi_{G}(x_{n+1};x_1,\ldots ,x_n):= \theta(f(x_1),\ldots ,f(x_{n+1}) ).
\end{gather*}
 Let $N \prec \cU$ be a small model containing $A$ and all relevant parameters. Then
 \begin{gather*}
 	(f_*(\mu))^{(n+1)}(\theta(y_1,\ldots ,y_{n+1})) = \left( f_*(\mu)_{y_{n+1}} \otimes (f_*(\mu))^{(n)}_{y_1,\ldots ,y_n} \right) (\theta(y_1,\ldots ,y_{n+1})) \\
 =  \int_{S_{y_1,\ldots ,y_n}(N)} F_{f_*(\mu)}^{\psi} \  \textrm{d}\left( (f_*(\mu))^{(n)} \right)  
 = \int_{S_{y_1,\ldots ,y_n}(N)} F_{f_*(\mu)}^{\psi} \  \textrm{d}\left( \left( G_{n} \right)_{*} \left(\mu^{(n)} \right) \right)\\
  = \int_{S_{x_1,\ldots ,x_n}(N)} \left( F_{f_{\ast}(\mu)}^{\psi} \circ G_n  \right) \  \textrm{d} \mu^{(n)}
  = \int_{S_{x_1,\ldots ,x_n}(N)}  F_{\mu}^{\psi_{G}} \  \textrm{d} \mu^{(n)} \\
  = \mu^{(n+1)}(\theta(f(x_1),\ldots ,f(x_{n+1}))). \qedhere
 \end{gather*}
\end{clmproof}

We now show that $f_*(\mu)$ is \fim\ over $A$. Fix a formula $\varphi(y,z)$ in $\mathcal{L}_{yz}$. Since $\mu$ is \fim, let $(\theta_{n}(x_1, \ldots, x_n))_{1 \leq n <\omega}$ be a sequence of $\cL(A)$-formulas witnessing  this for the formula $\varphi(f(x),z)$ as in Definition \ref{def:fim}. To avoid ``scope-of-quantifiers'' confusion, we let $w_1,\ldots ,w_n$ be new variables with $w_i$ of the same sort as $x_i$ for each $i$. For each $1 \leq n < \omega$, we consider the $\cL(A)$-formula 
\begin{equation*} 
\gamma_{n}(y_1,\ldots ,y_n) := \exists w_1 \ldots  \exists w_n \left( \theta(w_1,\ldots ,w_n) \wedge \bigwedge_{1 \leq i \leq n} f(w_i) = y_i \right). 
\end{equation*} 
Note that 
\begin{equation*}
\theta(x_1,\ldots ,x_n) \vdash \gamma_{n}(f(x_1),\ldots ,f(x_n)) \textrm{ for every } 	n \in \omega. \tag{$\ast$}
\end{equation*}
We will show that the formulas $(\gamma_{n}(y_1, \ldots, y_n))_{1 \leq n < \omega}$ witness  that $f_{*}(\mu)$ is \fim\ over $A$ with respect to the formula $\varphi(y,z)$. 

Fix $\varepsilon > 0$. Then there exists some $n_{\varepsilon} \in \omega$ so that: for any $n_{\varepsilon} \leq k < \omega$ and any $\bar{d}$ with $\models \theta_{k}(\bar{d})$ we have
\begin{equation*} 
\sup_{b \in \mathcal{U}^{z}} |\Av(\dbar)(\varphi(f(x),b)) - \mu(\varphi(f(x),b))|<\varepsilon. \tag{$\ast \ast$}
\end{equation*} 
Now, suppose that $\bar{c}$ is such that  $\mathcal{U} \models \gamma_{k}(\bar{c})$. Then, by definition of $\gamma_k$, there exists some $\bar{e} \in (\mathcal{U}^{x})^{k}$ such that $\models \theta_{k}(\bar{e})$ and $(f(e_1),\ldots ,f(e_k)) = (c_1,\ldots ,c_k)$. Therefore, 
\begin{gather*}
	\sup_{b \in \mathcal{U}^{z}} |\Av(\overline{c})(\varphi(y,b)) - f_*(\mu)(\varphi(y,b))| \\  
= \sup_{b \in \mathcal{U}^{z}} |\Av(c_1,\ldots ,c_n)(\varphi(y,b)) - \mu(\varphi(f(x),b))|  \\  
= \sup_{b \in \mathcal{U}^{z}} |\Av(f(e_1),\ldots ,f(e_n))(\varphi(y,b)) - \mu(\varphi(f(x),b))|  \\  
= \sup_{b \in \mathcal{U}^{z}} |\Av(e_1,\ldots ,e_n)(\varphi(f(x),b)) - \mu(\varphi(f(x),b))|   
\overset{(\ast \ast)}{<} \varepsilon.
\end{gather*}

Finally, by the Claim  and $(\ast)$ we have 
\begin{gather*}
\lim_{n \to \infty} \left( (f_{*}(\mu))^{(n)}(\gamma_{n}(y_1, \ldots, y_n)) \right)= \lim_{n \to \infty} \mu^{(n)} \left (\gamma_{n}(f(x_1),\ldots ,f(x_n)) \right) \\
\geq \lim_{n \to \infty} \mu^{(n)}(\theta_n(x_1, \ldots, x_n)) = 1.
\end{gather*} 
We conclude $f_*(\mu)$ is \fim.

~

\noindent (4) $(\subseteq)$ is trivial. For the opposite inclusion, consider $q \in S(f_*(\mu))$. Let $\pi(x)$ be the closure of $\{\varphi(f(x)) : \varphi(y) \in q\}$ under logical consequence. Then, every formula in $\pi(x)$ has positive $\mu$-measure and $\pi(x)$ is closed under conjunction. So $\pi(x)$ can be extended to a type $p \in S(\mu)$. It is clear that $f_*(p) = q$.  
\end{proof}

\subsection{Fim and fsg groups}\label{sec: fim groups}

Throughout this section, $G(x)$ will be a $\emptyset$-type-definable group.  Let $\varphi_0(x) \in \cL(\emptyset)$ be chosen for $G(x)$ as in Section \ref{sec: setting types}.

\begin{definition}\cite[Definition 6.1]{NIP2}
	A ($\emptyset$-)type-definable group $G$ is \emph{fsg} (\emph{finitely satisfiable generics}) if there is some $p \in S_{G}(\cU)$ and small $M \prec \cU$ such that for every $g \in G(\cU)$, $g \cdot p$ is finitely satisfiable in $G(M)$. (We note that for definable $G$, we could equivalently just require finite satisfiability in $M$, but in the type-definable case finite satisfiability in $G(M)$ is a stronger condition.)
\end{definition}

We consider a natural generalization of generically stable groups (Section \ref{sec: gen stab groups}) to fim groups. 
\begin{definition}
	\begin{enumerate}
		\item We let $\frak{M}_{G}(\cU)$ denote the (closed) set of all measures $\mu \in \frak{M}_{x}(\cU)$ \emph{supported on $G$}, i.e.~with $S(\mu) \subseteq S_{G}(\cU)$. 
		\item For $\mu \in \mathfrak{M}_{G}(\cU)$ and $g \in G(\cU)$, we let  $g \cdot \mu := \left(g \cdot - \right)_{\ast}(\mu)$ and  $\mu^{^-1} := \left(^{-1} \right)_{\ast}(\mu)$ (where, as in Section \ref{sec: setting types}, we view $\cdot, ^{-1}$ as globally defined functions whose restrictions to $G$ give the group operations). As $g\cdot -$ and $^{-1}$ are definable bijections on $\varphi_0(\cU)$, we get $S(g \cdot \mu) = g \cdot S(\mu)$ and $S(\mu^{-1}) = \{ p^{-1} : p \in S(\mu) \}$ (by Proposition \ref{prop:push-forward}(4)). In particular, $g\cdot \mu, \mu^{-1} \in \mathfrak{M}_{G}(\cU)$.
		We define the right action of $G$ on $\mathfrak{M}_{G}(\cU)$ analogously.
			\end{enumerate}
\end{definition}

\begin{definition}\label{def: fim group}
	We will say that a ($\emptyset$-)type-definable group $G(x)$ is \emph{fim} if there exists a right $G$-invariant \fim\  measure $\mu \in \frak{M}_{G}(\cU)$, i.e.~$ \mu  \cdot g = \mu$ for all $g \in G(\cU)$.
\end{definition}

\begin{remark}
	\begin{enumerate}
		\item In any theory, if $G$ is fim then it is both definably amenable and   fsg.
		\item If $T$ is NIP and $G$ is fsg, then it is fim.
	\end{enumerate}
\end{remark}
\begin{proof}
	(1)  
	Assume $G$ is type-definable over a small model $M \prec \cU$ by a partial type $G(x) = \{\psi_\alpha(x) \in \cL(M) : \alpha < \kappa\}$ and $\mu \in \mathfrak{M}_{G}(\cU)$ is left-$G(\cU)$-invariant and fim over $M$. 	Let $M'$ with $M \prec M' \prec \cU$ be a small $|M|^{+}$-saturated model. We claim that $\mu$ is finitely satisfiable in $G(M')$ (then for any $p \in S(\mu)$ we have $g \cdot p$ is finitely satisfiable in $G(M')$ for all $g \in G(\cU)$).

	Indeed, assume $\varphi(x,y) \in \cL$ and $b \in \cU^{y}$ with $r := \mu(\varphi(x,b)) > 0$. 
	As $\mu$ is fim, there exists $n \in \omega$ and a formula $\theta_{n}(x_1,\ldots,x_n) \in \cL(M)$  such that $\mu^{(n)} \left( \theta_n \left(\bar{x} \right) \right) > 0$ and for any $\abar$ in $\cU$ with $\models \theta_{n}(\abar)$ we have 
	\begin{equation*} 
		\sup_{b \in \mathcal{U}^{y}} |\Av(\abar)(\varphi(x,b)) - \mu(\varphi(x,b))| < \frac{r}{2}.
		\end{equation*}

As $\mu \in \mathfrak{M}_{G}(\cU)$, we have $\mu(\bigwedge_{\alpha \in s}\psi_\alpha(x)) = 1$ for every finite $s \subseteq \kappa$. Hence $\mu^{(n)} \left( \bigwedge_{1 \leq i \leq n} \bigwedge_{\alpha \in s}\psi_\alpha(x_i) \right) = 1$. It follows that 
$$\mu^{(n)} \left( \theta_n(x_1, \ldots, x_n) \land \bigwedge_{1 \leq i \leq n} \bigwedge_{\alpha \in s}\psi_\alpha(x_i) \right) > 0$$
for every finite $s \subseteq \kappa$. Then, by $|M|^{+}$-saturation of $M'$, we can find some $\bar{a}' = (a'_1, \ldots, a'_n)$  with $a'_i \in G(M')$ so that $\models \theta_n(a'_1, \ldots, a'_n)$. By the choice of $\theta_n$ we must have $\models \varphi(a'_i,b)$ for at least one $i$.

	(2) The standard construction of a left-invariant generically stable measure in an fsg group goes through in the type-definable case (see  \cite[Proposition 6.2]{NIP1}, \cite[Remark 4.4]{NIP3}, and Remark \ref{rem: fim in NIP is gen stab}).
\end{proof}

\begin{problem}
	  Does fsg imply definable amenability without assuming NIP? Do there exist fsg (and definably amenable) groups that are not \fim?
\end{problem}

\begin{definition}\label{def: definable conv} Suppose that $\mu \in \mathfrak{M}_{G}(\cU)$ is Borel-definable. Then for any measure $\nu \in \mathfrak{M}_{G}(\cU)$, the (definable) \emph{convolution of $\mu$ and $\nu$}, denoted $\mu * \nu$, is the unique measure in $\mathfrak{M}_{G}(\cU)$ such that for any formula $\varphi(x) \in \mathcal{L}(\cU)$, 
\begin{equation*} 
(\mu * \nu)(\varphi(x)) = ( \mu \otimes \nu)(\varphi(x \cdot y)). 
\end{equation*} 
We say that $\mu$ is \emph{idempotent} if $\mu * \mu = \mu$. 
\end{definition} 
\noindent When $T$ is NIP, it suffices to assume that $\mu$ is invariant (under $\Aut(\mathbb{M}/M)$ for some small model $M$), as then $\mu$ is automatically Borel-definable (\cite{Guide}). We refer to \cite[Section 3]{chernikov2022definable} for a detailed discussion of when convolution is well-defined.

 The following is a simultaneous generalization of Fact \ref{fac: basic fsg} from types to measures in arbitrary theories, and of the previously known case for measures under the NIP assumption \cite[Theorem 4.3]{NIP3}. 

\begin{proposition}\label{prop:unique} Suppose that $G(x)$ is a $\emptyset$-type-definable \fim\ group, witnessed by a right-$G$-invariant fim measure $\mu \in \frak{M}_{G}(\cU)$. Then we have:
\begin{enumerate}
\item $\mu = \mu^{-1}$; 
\item $\mu$ is left $G$-invariant;
\item $\mu$ is the unique left $G$-invariant measure in $\mathfrak{M}_{G}(\cU)$; 
\item  $\mu$ is the unique right $G$-invariant measure in $\mathfrak{M}_{G}(\cU)$.
\end{enumerate} 
\end{proposition} 

\begin{proof} 
\textbf{(1)} 
 Fix a formula $\varphi(x) \in \mathcal{L}_{x}(\cU)$. Let $M \prec \cU$ be a small model such that $\mu$ is $M$-invariant and $M$ contains the parameters of $\varphi$. As $\mu$ is \fim\  over $M$,  $\mu^{-1}$ is also  \fim\ over $M$ by Proposition \ref{prop:push-forward}(3).

Suppose $q(x) \in S(\mu_x|_{M})$ (so $q(x) \vdash G(x)$). Then, for $(\varphi')^{*}(y,x) = \varphi(x \cdot y)$ (in the notation of Section \ref{sec: measures setting}) we have (for any $b \models q(x)$):
\begin{gather*}
	 F_{\mu_{y}^{-1}}^{(\varphi')^{*}}(q) = \mu_{y}^{-1}(\varphi(b \cdot y))
 = \mu_{y}(\varphi(b \cdot y^{-1})) 
 = \mu_{y}(\varphi((y \cdot b^{-1})^{-1})) \\
 \overset{(*)}{=} \mu_{y}(\varphi(y^{-1})) = \mu_{y}^{-1}(\varphi(y)),
\end{gather*}
\noindent where $(*)$ follows by right $G$-invariance of $\mu \in \mathfrak{M}_G(\cU)$ applied to the formula $\psi(y) := \varphi(y^{-1})$.
 
Using this and that  \fim\ measures commute with all Borel definable measures (Fact \ref{fac: fim commutes}) we have:
\begin{gather*} 
\mu * \mu^{-1}(\varphi(x)) = \mu_{x} \otimes \mu^{-1}_{y}(\varphi(x \cdot y)) = \mu^{-1}_{y} \otimes \mu_{x}(\varphi(x \cdot y)) \\
= \int_{S_{x}(M)} F_{\mu^{-1}_{y}}^{(\varphi')^*} \textrm{d}\mu_{x} = \int_{S(\mu_x|_M)} F_{\mu^{-1}_{y}}^{(\varphi')^*} \textrm{d}\mu_{x} 
= \int_{S_{x}(M)} \mu_{y}^{-1}(\varphi(y)) \textrm{d}\mu_{x} \\
= \mu_{y}^{-1}(\varphi(y)) = \mu_{x}^{-1}(\varphi(x)).
\end{gather*} 

Similarly, for $\varphi'(x,y) = \varphi(x \cdot y)$ and any $q(y) \in S(\mu^{-1}_y|_M)$ (so $q(y) \vdash G(y)$), by right $G$-invariance of $\mu$ we have $F_{\mu_{x}}^{\varphi'}(q) = \mu(\varphi(x))$. Then 
\begin{gather*}
	\mu * \mu^{-1}(\varphi(x)) = \mu_{x} \otimes \mu^{-1}_{y}(\varphi(x \cdot y)) 
=  \int_{S_{y}(M)} F_{\mu_x}^{\varphi'} \textrm{d} \left(\mu^{-1}_y\right) \\
=  \int_{S_{y}(M)} \mu(\varphi(x)) \textrm{d} \left( \mu^{-1}_y \right) 
=  \mu(\varphi(x)). 
\end{gather*}
We conclude that $\mu(\varphi(x)) = \mu^{-1}(\varphi(x))$.  

~

\noindent\textbf{(2)} For any $a \in G(\cU)$ and $\varphi(x) \in \cL(\cU)$, using (1) and right $G$-invariance of $\mu$ for the formula $\psi(x) := \varphi(x^{-1})$ we have: 
\begin{gather*}
	\mu(\varphi(a \cdot x)) = \mu^{-1}(\varphi(a \cdot x)) 
= \mu(\varphi(a \cdot x^{-1})) 
= \mu(\varphi((x \cdot a^{-1})^{-1})) \\ 
= \mu(\varphi(x^{-1})) = \mu^{-1}(\varphi(x)) 
= \mu(\varphi(x)) . 
\end{gather*}
%
%Likewise, if $a \in \stab_{l}(\mu)$, then 
%\begin{align*} 
%\mu(\varphi(x \cdot a )) &= \mu^{-1}(\varphi(x \cdot a)) \\ 
%&= \mu(\varphi(x^{-1} \cdot a )) \\ 
%&= \mu(\varphi((a^{-1} \cdot x)^{-1})) \\ 
%&= \mu(\varphi(x^{-1})) \\ 
%&= \mu(\varphi(x)) \\ 
%\end{align*} 
%Hence $\stab(\mu) = \stab_{l}(\mu)$. 

\

\noindent \textbf{(3)} Suppose that $\nu \in \mathfrak{M}_{G}(\cU)$ is left $G$-invariant, and let $\varphi(x) \in \cL(\cU)$ be arbitrary. Let $M \prec \cU$  be a small model such that $\mu$ is invariant over $M$ and $M$ contains the parameters  of $\varphi$. 

As in (1), for any $q(y) \in S(\nu_y|_M)$ (so $q(y) \vdash G(y)$), by right $G$-invariance of $\mu$ we have $F_{\mu_{x}}^{\varphi'}(q) = \mu(\varphi(x))$. Hence 
\begin{equation*} 
\int_{S_{y}(M)} F_{\mu_x}^{\varphi'} \textrm{d} \nu_y = \int_{S(\nu|_M)} F_{\mu_x}^{\varphi'} \textrm{d} \nu_y = \int_{S(\nu|_M)} \mu( \varphi(x))\textrm{d} \nu_y = \int_{S_{y}(M)} \mu(\varphi(x)) \textrm{d} \nu_y.  \tag{*}
\end{equation*} 
Second, for any $q \in S(\mu_{x}|_{M})$ (so $q(x) \vdash G(x)$), by left $G$-invariance of $\nu$ we have (for any $b \models q$):
\begin{equation*} 
F_{\nu_{y}}^{(\varphi')^*}(q) = \nu(\varphi(b \cdot y)) = \nu(\varphi(y)). \tag{\dag}
\end{equation*} 
So the map  $F_{\nu_{y}}^{(\varphi')^*}: S(\mu|_{M}) \to [0,1]$ is constant on the support of $\mu$, hence Borel. As $\mu$ is \fim, the proof of \cite[Proposition 5.15]{CGH} implies 
\begin{equation*} 
\int_{S(\mu|_M)} F_{\nu_y}^{(\varphi')^*} \textrm{d}\mu_x = \int_{S_{y}(M)} F_{\mu_x}^{\varphi'} \textrm{d}\nu_y, \tag{\ddag}
\end{equation*} 
where on the left we view $\mu$ as a regular Borel probability measure restricted to the compact set $S(\mu|_{M})$.
Then
\begin{gather*}
	\mu(\varphi(x)) = \int_{S_{y}(M)} \mu(\varphi(x)) \textrm{d}\nu_y 
\overset{(*)}{=} \int_{S_{y}(M)} F_{\mu_x}^{\varphi'} d\nu_y\\
\overset{(\ddag)}{=} \int_{S(\mu|_M)} F_{\nu_y}^{(\varphi')^*} \textrm{d}\mu_x 
\overset{(\dag)}{=} \int_{S(\mu|_M)} \nu(\varphi(y)) \textrm{d}\mu_x  =  \int_{S_{x}(M)} \nu(\varphi(y)) \textrm{d}\mu_x 
= \nu(\varphi(y)).
\end{gather*}

\ 

\noindent \textbf{(4)} Let $\nu \in \mathfrak{M}_{G}(\cU)$ be right $G$-invariant, $\varphi(x) \in \cL(\cU)$ and $M \prec \cU$ containing the parameters of $\varphi$ and such that $\mu$ is $M$-invariant.

As in (3), the map $F_{\nu_{x}}^{\varphi'}:S_{y}(M) \to [0,1]$ is $(\mu_{y}|_M)$-measurable since it is constant on the support by right $G$-invariance of $\nu$. As $\mu$ is \fim\ we can apply \cite[Proposition 5.15]{CGH} again to get
\begin{equation*}
 \int_{S_{y}(M)} F_{\nu_x}^{\varphi'} \textrm{d} \mu_y = \int_{S_{x}(M)} F_{\mu_y}^{(\varphi')^{*}} \textrm{d}\nu_x. 
\end{equation*} 
As $\mu$ is left $G$-invariant by (2), for any $q \in S(\nu_x|_{M})$, we have (for $b \models q$):
\begin{equation*}
F_{\mu_{y}}^{(\varphi')^*}(q) = \mu_{y}(\varphi(b \cdot y)) = \mu(\varphi(y)). 
\end{equation*} 
Combining we get
\begin{gather*}
	\nu(\varphi(x))= \int_{S_{y}(M)} \nu(\varphi(x)) \textrm{d} \mu_y 
= \int_{S_{y}(M)} F_{\nu_x}^{\varphi'} \textrm{d}\mu_y \\
= \int_{S_{x}(M)} F_{\mu_y}^{(\varphi')^{*}} \textrm{d}\nu_x 
= \int_{S_{x}(M)} \mu(\varphi(y)) \textrm{d}\nu_x 
= \mu(\varphi(y)). \qedhere
\end{gather*}
\end{proof} 

Proposition \ref{prop:unique} and its symmetric version with ``left'' swapped with ``right'' (obtained by a ``symmetric'' proof) yield:

\begin{corollary}
An ($\emptyset$)-type-definable group $G(x)$ is {\em fim} if and only if there exists a left $G$-invariant \fim\ measure $\mu \in \mathfrak{M}_G(\mathcal{U})$.\\
\end{corollary}

\subsection{Idempotent fim measures and generic transitivity}\label{sec: gen trans for meas}
Let $G(x)$ be a $\emptyset$-type-definable group, and we are in the same setting and notation as in Section \ref{sec: setting types}. For a measure $\mu \in \mathfrak{M}_{G}(\cU)$, we let 
$$\Stab(\mu) := \{g \in G(\cU) : \mu \cdot g = g \} $$ 
denote the right-stabilizer of $\mu$.

\begin{fact}\label{fac: stab of def meas type def}
	When $\mu \in \mathfrak{M}_{G}(\cU)$ is a measure definable over $M \prec \cU$, then $\Stab(\mu)$ is an $M$-type-definable subgroup of $G(\cU)$ (see e.g.~\cite[Proposition 5.3]{chernikov2022definable}). 
\end{fact}

We now consider the main question investigated in \cite{chernikov2022definable, chernikov2023definable} in the case of measures (generalizing from types in Section \ref{sec: gen trans}). 

\emph{Again, we let $H := \Stab(\mu)$.}

\begin{definition}
	For the rest of the section, we let $f: \left( \cU^{x} \right)^2 \to \left( \cU^{x} \right)^2$ be the ($\emptyset$-definable) map $f(x_1,x_0) = (x_1 \cdot x_0, x_0)$ (where $\cdot$ is viewed as a globally defined function whose restriction to $G$ defines the group operation, see Section \ref{sec: setting types}).
\end{definition}

\begin{proposition}\label{prop: gen trans meas}
Let $\mu \in \mathfrak{M}_{G}(\cU)$ be an idempotent \fim\ measure. Then the following are equivalent:
\begin{enumerate}
	\item $\mu \in \mathfrak{M}_{H}(\cU)$;
	\item $\mu^{(2)} = f_{\ast}\left(\mu^{(2)} \right)$;
	\item $\mu \otimes p = f_{\ast}(\mu \otimes p)$ for every $p \in S(\mu)$.
\end{enumerate}
\end{proposition}
\begin{proof}
	\noindent \textbf{(2) $\Leftrightarrow$ (3)} This equivalence only uses that $\mu$ is a definable measure.
	
	By the definition of $f$ we have $f_{\ast} \left( \mu^{(2)} \right) \left( \varphi(x_1, x_0) \right) = \mu^{(2)}\left( \tilde{\varphi}(x_1,x_0) \right)$ for all $\varphi(x_1, x_0) \in \mathcal{L}_{x_1, x_0}\left( \cU \right)$, where $\tilde{\varphi}(x_1, x_0) := \varphi(x_1 \cdot x_0, x_0)$.
	By definition of $\mu^{(2)}$  we have: $\mu^{(2)} \neq f_{\ast}\left(\mu^{(2)} \right)$ if and only if there exists a formula $\varphi(x_1,x_0) \in \mathcal{L}_{x_1,x_0}(\cU)$ and a small model $M \prec \cU$ so that $M$ contains the parameters of $\varphi$ and $\mu$ is $M$-definable, such that 
	$$\int_{S(\mu|_{M})} \left \lvert F_{\mu,M}^{\tilde{\varphi}}(p) - F^{\varphi}_{\mu,M}(p) \right \rvert \textrm{d} \mu_M > 0,$$ 
	where $\mu|_{M} \in \mathfrak{M}_{x_0}(M)$ is the restriction of $\mu$ to $M$, $\mu_{M}$ is the restriction $\mu|_M$ viewed as a regular Borel measure on $S_{x_0}(M)$. By definability of $\mu$, the function $ p \in S(\mu|_{M}) \mapsto \left \lvert F_{\mu,M}^{\tilde{\varphi}}(p) - F^{\varphi}_{\mu,M}(p) \right \rvert \in [0,1]$ is continuous (and non-negative), hence the integral is $>0$ if and only if $\left \lvert F_{\mu,M}^{\tilde{\varphi}}(p) - F^{\varphi}_{\mu,M}(p) \right \rvert >0$ for some $p \in S(\mu|_{M})$, that is $\mu\left( \varphi(x_1, b) \right) \neq \mu(\varphi(x_1 \cdot b,b))$ for some $p \in S(\mu)$ and $b \models p|_{M}$, that is $\left(\mu \otimes p \right)\left( \varphi(x_1, x_0) \right) \neq f_{\ast}\left(\mu \otimes p \right)\left( \varphi(x_1, x_0) \right)$ for some $p \in S(\mu)$.
	
	\ 
	
	\noindent \textbf{(3) $\Rightarrow$ (1)} Fix $p \in S(\mu)$, let $M \prec \cU$ be such that $\mu$ is $M$-invariant, and let $a \models p|_{M}$. 

\begin{clm}
	We have $\mu|_{Ma} = \left( \mu \cdot a \right)|_{M a}$. 
\end{clm}
\begin{clmproof}
	Any $\psi(x) \in \mathcal{L}_x(Ma)$ is of the form $\varphi(x,a)$ for some $\varphi(x,y) \in \mathcal{L}_{xy}(M)$. By (3) we have:
	\begin{gather*}
		\mu(\varphi(x,a)) = (\mu \otimes p)(\varphi(x,y)) 
= f_{\ast}(\mu \otimes p)(\varphi(x,y)) \\
= (\mu \otimes p)(\varphi(x \cdot y, y)) 
= \mu(\varphi(x \cdot a, a)).
\end{gather*}
\noindent Therefore, 
\begin{gather*}
\left( \mu \cdot a \right)(\psi(x)) = \mu(\psi(x \cdot a)) 
= \mu(\varphi(x \cdot a, a)) 
=\mu(\varphi(x,a)) 
= \mu(\psi(x)).
	\end{gather*}
\end{clmproof}
\begin{clm}
	The measure $\mu \cdot a$ is $(Ma)$-definable. 
\end{clm}
\begin{clmproof}
	Since $\mu$ is $M$-definable, it is also $Ma$-definable. Consider the $Ma$-definable map $g: x \mapsto x \cdot a$. Note  that $g_{\ast}(\mu)  = \mu \cdot a$ and so by Proposition \ref{prop:push-forward}, $\mu \cdot a $ is also $Ma$-definable.
\end{clmproof}

Hence, by Proposition \ref{prop: FIM unique inv ext} over $Ma$, we get $\mu = \mu \cdot a$. That is, $a \in \Stab(\mu)$, so $p|_{M}(x) \vdash H(x)$. 

\

\noindent \textbf{(1) $\Rightarrow$ (3)} Let $\varphi(x,y) \in \cL(\cU)$ be arbitrary, and $M \prec \cU$ contain its parameters such that $\mu$ is $M$-invariant. The measure $\mu$ is right $H$-invariant, and given any $p \in S(\mu)$ we have $p \in S_{H}(\cU)$ by (1). So for $a \models p|_{M}$ we have $a \in H$, hence $\mu \otimes p (\varphi(x,y)) = \mu (\varphi(x, a)) = \mu (\varphi(x \cdot a, a)) = f_{\ast}(\mu \otimes p) (\varphi(x,y))$.
\end{proof}

\begin{definition}
	We say that an idempotent \fim\ measure $\mu \in \mathfrak{M}_{G}(\cU)$ is \emph{generically transitive} if  it satisfies any of the equivalent conditions in Proposition \ref{prop: gen trans meas}.
\end{definition}

\noindent In particular, a type is generically transitive in the sense of Definition \ref{def: gen trans type} if and only if, viewed as a Keisler measure, it is generically transitive.

The following is a generalization of Remark \ref{rem: fsg follows}:

\begin{remark}\label{rem: gen trans implies fim}
Assume $\mu$ is \fim\  and $\mu \in \mathfrak{M}_{H}(\cU)$. Then:
\begin{enumerate}
	\item $H$ is a \fim\ group (Definition \ref{def: fim group}), hence $\mu$ is both the unique left-invariant and the unique right-invariant measure supported on $H$ (by Proposition \ref{prop:unique});
	\item $H$ is the smallest type-definable subgroup of $G$ with $\mu \in \mathfrak{M}_{H}(\cU)$;
	\item $H$ is both the left and the right stabilizer of $\mu$ in $G$.
\end{enumerate}
\end{remark}
\begin{proof}
	(1) $H$ is a \fim\ group, witnessed by the right $H$-invariant \fim\ measure $\mu \in \mathfrak{M}_{H}(\cU)$.
	
	(2) $H$ is type-definable by Fact \ref{fac: stab of def meas type def}.

	 For any type-definable $H' \leq G(\cU)$ with $S(\mu) \subseteq S_{H'}(\cU)$, the group $H'' := H \cap H'$ is type-definable with $S(\mu)  \subseteq  S_{H''}(\cU)$ and $H'' \leq H$. If the index is $\geq 2$, we have some $g \in H$ with $H'' \cap \left( H'' \cdot g \right) = \emptyset$, and $S(\mu)  \subseteq  S_{H''}(\cU), S \left( \mu \cdot g \right) \subseteq   S_{H'' \cdot g }(\cU)$, so $S(\mu) \cap S(\mu \cdot g) = \emptyset$, so $\mu \neq \mu \cdot g$ --- a contradiction. So $H'' = H$, and $H \subseteq H'$.
	
	(3) Let $H_{\ell}$ be the left stabilizer of $\mu$. Then $H_\ell$ is type-definable by a symmetric version of Fact \ref{fac: stab of def meas type def}. By (1), $\mu$ is  left $H$-invariant, so we have $H \subseteq H_\ell$, and so $S(\mu) \subseteq  S_{H_\ell}(\cU)$ is left  $H_\ell$-invariant. By a symmetric argument as in (2), we conclude $H = H_{\ell}$.
\end{proof}

\begin{example}
	If $G'$ is a \fim\ type-definable subgroup of $G$, witnessed by a $G'$-invariant \fim\ measure $\mu \in \mathfrak{M}_{G'}(\cU)$, then $\mu$ is obviously idempotent and generically transitive.
\end{example}

Analogously to the case of types (Section \ref{sec: gen trans}), the following is our main question in the case of measures:

\begin{problem}\label{conj: main measures}
Assume that $\mu \in \mathfrak{M}_{G}(\cU)$ is \fim\ and idempotent. Is it true that then $\mu$ is generically transitive? Assuming $T$ is NIP?
\end{problem}

\noindent A positive answer for stable $T$ is given in  \cite{chernikov2022definable}, see Section \ref{sec: stable measures} for a discussion.  

By symmetric versions of the above considerations, one easily gets:

\begin{remark}
Section 3.7 remains valid if one swaps ``left'' with ``right'' (including left and right stabilizers) and replaces $f$ by $f(x_1,x_0)=(x_0 \cdot x_1,x_0)$.\\
\end{remark}

\subsection{Idempotent fim measures in abelian groups}\label{sec: idemp meas in ab proof}
Finally, we have all of the ingredients to adapt the proof in Section \ref{sec: abelian for types} from types to measures and give a positive answer to Problem \ref{conj: main measures} for abelian groups in arbitrary theories. Let $G(x)$ and $f$ be as in the previous section.

First note the idempotency of a \fim\ measure $\mu$ can be iterated along a Morley sequence in $\mu$:
\begin{lemma}\label{lemma:fam idem} Let $\mu \in \mathfrak{M}_{G}(\cU)$ and suppose that $\mu$ is fim and idempotent. Then for any formula $\psi(x) \in \mathcal{L}(\cU)$ and $k \in \omega$ we have
\begin{equation*} 
\mu(\psi(x)) = \mu^{(k)}(\psi(x_1 \cdot \ldots  \cdot x_k)). 
\end{equation*} 
\end{lemma} 
\begin{proof}
By induction on $k$, the base case $k =2$ is the assumption on $\mu$. So assume that for any $\gamma(x) \in \mathcal{L}(\cU)$ and $\ell \leq k-1$ we have 
\begin{gather*}
	\mu^{(\ell)}_{x_1, \ldots, x_{\ell}}(\gamma(x_1 \cdot \ldots  \cdot x_{\ell})) = \mu(\gamma(x)). \tag{a}
\end{gather*}
  
Fix $\psi(x) \in \mathcal{L}(\cU)$, and choose a small $M \prec \cU$ such that $\mu$ is $M$-invariant and $M$ contains the parameters of $\psi$. Let 
 $\theta(x_1,\ldots ,x_{k-1}; x_k) := \psi(x_1 \cdot \ldots  \cdot x_k)$ and $\varphi(x_1; x_k) := \psi(x_1 \cdot x_k)$. For any $q \in S_{x_k}(M)$ (and any $b \models q$) we have 
 \begin{gather*}
 	F^{\theta}_{\mu^{(k-1)}_{x_1, \ldots, x_{k-1}}} (q) = \mu^{(k-1)}\left( \psi(x_1 \cdot \ldots  \cdot x_{k-1} \cdot b) \right) \overset{\textrm{(a)}}{=} \mu(\psi(x_1 \cdot b)) = F^{\varphi}_{\mu_{x_1}} (q). \tag{b}
 \end{gather*}

 Then, using that $\mu$ is \fim\ and \fim\ measures commute with Borel definable measures, 
 \begin{gather*}
 	\mu^{(k)}\left( \psi(x_1 \cdot \ldots  \cdot x_k) \right) = \mu_{x_k}\otimes \mu^{(k-1)}_{x_1, \ldots ,x_{k-1}}(\psi(x_1 \cdot \ldots \cdot x_k))\\
 	 = \mu^{(k-1)}_{x_1, \ldots, x_{k-1}} \otimes \mu_{x_{k}} (\psi(x_1 \cdot \ldots \cdot x_k)) = \int_{S_{x_k}(M)} F^{\theta}_{\mu^{(k-1)}_{x_1, \ldots, x_{k-1}}} \textrm{d} \mu_{x_k} \\
 	 \overset{\textrm{(b)}}{=} \int_{S_{x_k}(M)} F^{\varphi}_{\mu_{x_1}} \textrm{d} \mu_{x_k} 
 	= \mu_{x_1} \otimes \mu_{x_k} \left( \varphi(x_1, x_k) \right) = \mu_{x_1} \otimes \mu_{x_k} \left( \psi(x_1 \cdot x_k) \right) \overset{\textrm{(a)}}{=} \mu(\psi(x)).\qedhere
 \end{gather*}
\end{proof}

\begin{lemma} \label{lem: proof of the key prop meas ab}
	Assume that $\mu \in \mathfrak{M}_{G}(\cU)$ is \fim\ and idempotent. Let $k \geq 2$ be arbitrary and let $g: (\cU^x)^k \to (\cU^x)^{k+1}$ be the definable map 
	$$g(x_1, \ldots, x_k) = (x_1 \cdot \ldots \cdot x_k, x_1, \ldots, x_k),$$
where ``$\cdot$'' is viewed as a globally defined function whose restriction to $G$ defines the group operation (see Section \ref{sec: setting types}). Let $\lambda_k(y, x_1, \ldots,x_k) := g_{\ast} \left( \mu^{(k)} \right)$. Then:
	\begin{enumerate}
		\item $\lambda_k|_{x_1, \ldots, x_k} = \mu^{(k)}$;
		\item if $G$ is abelian, then $\lambda_k|_{y,x_i} = f_{\ast}\left(\mu^{(2)}(y,x_i) \right)$ for every $1 \leq i \leq k$.
	\end{enumerate}
\end{lemma}

\begin{proof}	
	\textbf{(1)} By definition of $g$.

	\noindent \textbf{(2)} 
	Let $\varphi(y,x) \in \mathcal{L}(\cU)$ be arbitrary, and let $M \prec \cU$ be a small model containing its parameters, and so that $\mu$ is $M$-invariant. We let 
	\begin{gather*}
		\theta(x_{1},\ldots ,x_{i-1},x_{i+1},\ldots ,x_{k};x_i) := \varphi(x_{1} \cdot \ldots  \cdot x_{k},x_i),  \  \psi(x;y) : = \varphi(x \cdot y,y),\\
		\mu^{(k-1)}_{\hat{x}_i} := \mu_{x_{k-1}} \otimes \ldots  \otimes  \mu_{x_{i+1}} \otimes \mu_{x_{i-1}} \otimes \ldots  \otimes \mu_{x_1}
	\end{gather*}
 Using that $G$ is abelian and Lemma \ref{lemma:fam idem}, for any $q \in S_{x_i}(M)$ and $b \models q$ we have (renaming the variables when necessary):  
\begin{gather*} 
F_{\mu^{(k-1)}_{\hat{x}_i}}^{\theta}(q) = \mu_{\hat{x}_i}^{(k-1)}(\varphi(x_1 \cdot \ldots  \cdot x_{i-1} \cdot b \cdot x_{i+1} \cdot  \ldots  \cdot x_k, b)) \\
= \mu_{\hat{x}_i}^{(k-1)}(\varphi(x_1 \cdot \ldots  \cdot x_{i-1} \cdot x_{i+1} \cdot \ldots  \cdot x_{k} \cdot b,b))\\
=  \mu_{x}(\varphi(x \cdot b,b)) 
=  F_{\mu_{x}}^{\psi}(q).
\end{gather*} 
Hence, since $G$ is abelian, $\mu$ is \fim, and \fim\  measures commute with Borel definable measures, we get 
\begin{gather*} 
\lambda_{k}(\varphi(y,x_i)) = \mu_{x_1,\ldots ,x_k}^{(k)}(\varphi(x_1 \cdot  \ldots  \cdot x_{k},x_i))\\
= \mu_{\hat{x}_i}^{(k-1)} \otimes \mu_{x_i}(\varphi(x_1 \cdot \ldots \cdot x_{i-1} \cdot x_{i+1} \ldots  \cdot x_{k} \cdot x_i,x_i))\\
= \int_{S_{x_i}(M)} F^{\theta}_{\mu_{\hat{x}_i}^{(k-1)}} \textrm{d} \mu_{x_i}
= \int_{S_{x_i}(M)} F^{\psi}_{\mu_{x}} \textrm{d} \mu_{x_i}\\
= \mu_{x} \otimes \mu_{x_i}(\varphi(x \cdot x_i,x_i))
= \mu_{x} \otimes \mu_{y}(\varphi(x \cdot y,y))
= f_{*}\left(\mu^{(2)} \right)(\varphi(x,y)) \qedhere. 
\end{gather*} 
\end{proof}

Finally, using results about the randomization (Theorem \ref{thm: unif gen stab meas}) and Lemma \ref{lem: proof of the key prop meas ab}, we can show generic transitivity in the abelian case:
\begin{theorem}\label{prop:main} Assume that $G(x)$ is an abelian type-definable group and $\mu \in \mathfrak{M}_{G}(\cU)$ is \fim\ and idempotent. Then $\mu$ is generically transitive.
\end{theorem}

\begin{proof}
By Proposition \ref{prop: gen trans meas}, it suffices to show that $\mu^{(2)} = f_{\ast}\left(\mu^{(2)} \right)$.
Assume not, say 
$$\left \lvert\mu^{(2)} (\varphi(x_1,x_2)) - f_{\ast}\left(\mu^{(2)}\right) (\varphi(x_1,x_2)) \right\rvert = \varepsilon_0$$
 for some $\varphi(x_1,x_2) \in \cL(\cU)$ and some $\varepsilon_0 > 0$. Let $M \prec \cU$ be a small model containing the parameters of $\varphi$, and so that $\mu$ is invariant over $M$. Let $n$ be as given by the moreover part of Theorem \ref{thm: unif gen stab meas}  for $\mu, \varphi, \varepsilon_0$. Fix any $k > n$, and consider the definable map $g: (\cU^{x})^{k} \to (\cU^{x})^{k+1}$ given by $g(x_1, \ldots, x_k) = (x_1 \cdot \ldots \cdot x_k, x_1, \ldots, x_k)$. Then $g$ induces a continuous map from $S_{\mathbf{x}}(\cU)$ to $S_{\mathbf{x}y}(\cU)$, where $\mathbf{x} = (x_i:i \in \omega)$ and we let $\lambda \in \mathfrak{M}_{\mathbf{x}y}(\cU)$ be defined by $\lambda := g_{\ast}(\mu^{(\omega)})$. That is, for every  $m \in \omega$ and every $\varphi(x_1, \ldots, x_m, y) \in \cL(\cU)$ we have 
 $$\lambda(\varphi(x_1, \ldots, x_m, y)) = \mu^{(\omega)}\left(\varphi(x_1, \ldots, x_m, x_1 \cdot \ldots \cdot x_k) \right).$$
  Then $\lambda|_{\mathbf{x}} = \mu^{(\omega)}$  and $\lambda|_{(x_1, \ldots, x_k) y} = \lambda_k$ from Lemma \ref{lem: proof of the key prop meas ab}, so by Lemma \ref{lem: proof of the key prop meas ab} we then have $| \lambda(\varphi(y,x_i)) - \mu^{(2)}(\varphi(y,x))| > \varepsilon_0$ for all $i=1, \ldots, k$ --- contradicting the choice of $n$.
\end{proof}

%\begin{theorem}\label{theorem:main} Suppose that $T$ is a first order theory expanding a group. Let $\mathcal{G}$ be a monster model of $T$ and $G$ be a small elementary submodel of $\mathcal{G}$. Let $\mu \in \mathfrak{M}_{x}^{\inva}(\mathcal{G},G)$ and suppose that $\mu$ is generically stable (i.e. \fim). Then the following are equivalent:
%\begin{enumerate} 
%\item $\mu$ is idempotent. 
%\item The stabilizer of $\mu$, denoted $\stab(\mu)$, is a $G$-type-definable subgroup of $\mathcal{G}$ and $\mu$ is the unique Keisler measure such that $\mu$ is $\stab(\mu)$-left-invariant and $\mu([\stab(\mu)]) = 1$. 
%\end{enumerate} 
%This gives a correspondence between generically stable, idempotent Keisler measures which are automorphism invariant over $G$ and $G$-type-definable fsg subgroups of $\mathcal{G}$ via 
%\begin{equation*} 
%\mu \to \stab(\mu) \text{ and } H \to \mu_H
%\end{equation*} 
%where $\mu_H$ is the unique measure which witnesses definable amenability of $H$. 
%\end{theorem} 
%
%\begin{proof} Follows directly from Proposition \ref{prop:main}, the definition of fsg, and the fact that all abelian groups are amenable (which implies they are definably amenable).
%\end{proof} 

\subsection{Support transitivity of idempotent measures}\label{sec: support trans}

A tempting strategy for generalizing the arguments in Sections \ref{subsection: stable theories}--\ref{sec: rosy types} with ranks from idempotent types to idempotent measures in $T$ is to apply (a continuous logic version of) the proof for types in the randomization $T^{R}$ of $T$, assuming that the randomization preserves the corresponding property. E.g, stability is preserved \cite{ben2009randomizations}, and (real) rosiness is known to be preserved in some special cases \cite{andrews2019independence} (e.g.~when $T$ is $o$-minimal). We note that simplicity of $T$ is not preserved, still one gets that $T^{R}$ is NSOP$_1$ assuming that $T$ is simple \cite{BYChRa}. When attempting to implement this strategy, one arrives at the following natural condition connecting the behavior of measures and types in their support:
 \begin{definition}
 	Assume that $G$ is a type-definable group and $\mu \in \mathfrak{M}_{G}(\cU)$. We say that $\mu$ is \emph{support transitive} if $\mu \ast p = \mu$ for every $p \in S(\mu)$.
 \end{definition}
 
 \begin{remark}
 	\begin{enumerate}
 		\item If $p \in S_{G}(\cU)$ is a generically stable idempotent type, then it is obviously support transitive (viewed as a Keisler measure).
 		\item If $\mu \in \mathfrak{M}_{G}(\cU)$ is generically transitive, then it is support transitive.
 		
 		Indeed, note that if $p \in S(\mu)$, $\theta(x,b) \in \mathcal{L}_{x}(\mathcal{U})$, $\mu$ is $M$-invariant for a small model $M \prec \cU$, and $c \models p|_{Mb}$, then  
\begin{equation*} 
(\mu * p)(\theta(x,b)) = (\mu_{x} \otimes p_{y})(\theta(x\cdot y,b)) = \mu(\theta(x \cdot c,b)) = \mu(\theta(x,b)),
\end{equation*} 
where the last equality follows as $p \vdash \Stab(\mu)$ by assumption, and $\Stab(\mu)$ is $M$-type-definable by Fact \ref{fac: stab of def meas type def}.
 	\end{enumerate}
 \end{remark}
 
 Thus, we view the following as an intermediate (and trivial in the case of types) version of our main Problem \ref{conj: main measures}:

\begin{problem}\label{intermediate:conjecture} 
Assume that $G$ is a type-definable group and $\mu \in \mathfrak{M}_{G}(\cU)$ is \fim\ and idempotent. Is $\mu$ support transitive?
\end{problem}

The following example (based on \cite[Example 4.5]{chernikov2023definable}) illustrates that the \fim\ assumption in Problem \ref{intermediate:conjecture}  cannot be relaxed to either \emph{definable} or \emph{Borel-definable and finitely satisfiable}, even in abelian NIP groups:
\begin{example}\label{example:intermediate} Consider $M := (\mathbb{R},<,+)$, $M \prec \mathcal{\cU}$, $G(\cU) = \cU$ and $\mu := \frac{1}{2}\delta_{p_{-\infty}} + \frac{1}{2}\delta_{p_{\infty}}$ and 
$\nu := \frac{1}{2}\delta_{p_{0^{-}}} + \frac{1}{2}\delta_{p_{0^{+}}}$, where $p_{\infty}, p_{-\infty}, p_{0^+}, p_{0^-}$  are the unique complete $1$-types satisfying:
\begin{itemize}
\item $p_\infty \supseteq  \{ x > a : a \in \mathbb{R}\} \cup \{ x < b : b \in \cU, b > \mathbb{R}\}$,
\item $p_{-\infty} \supseteq \{ x < a: a \in \mathbb{R}\} \cup \{x > b : b \in \cU,  b < \mathbb{R}\}$, 
\item $p_{0^{+}} \supseteq \{ x < a: a \in \cU, a > 0\} \cup \{x > 0 \}$, 
\item $p_{0^{-}} \supseteq \{ x <  0 \} \cup \{x > b : b \in \cU, b < 0\}$. 
\end{itemize} 
Then $\mu$ is finitely satisfiable in $M$ (hence also Borel-invariant over $M$ by NIP) and $\nu$ is definable over $M$, but neither is \fim. The following are easy to verify directly:
\begin{enumerate} 
\item $\mu * p_{\infty} = p_{\infty}$, $\mu * p_{-\infty} = p_{-\infty}$ and $\mu * \mu = \mu$ --- hence $\mu$ is idempotent, finitely satisfiable in $M$, but not support-transitive;
\item likewise, $\nu *p_{0^{+}} = p_{0^{+}}$, $\nu *p_{0^{-}} = p_{0^{-}}$, and $\nu * \nu = \nu$ --- hence $\nu$ is idempotent, definable over $M$, but not support transitive.
\end{enumerate} 
\end{example} 

Support transitivity is closely related to the algebraic properties of the semigroup induced by $\ast$ on the support of an idempotent measure, studied in \cite[Section 4]{chernikov2022definable}.

\begin{fact}\label{fac: supp is nice semigroup}
	\cite[Corollary 4.4]{chernikov2022definable}  Assume that $\mu \in \mathfrak{M}_{G}(\cU)$ is \fim\ and idempotent. Then $(S(\mu),*)$ is a compact Hausdorff semigroup which is left-continuous, i.e.~the map $p \in S(\mu) \mapsto p \ast q \in S(\mu)$ is continuous for each fixed $q \in S(\mu)$.
\end{fact}

\begin{proposition}\label{prop: supp trans equiv}
	Assume that $\mu \in \mathfrak{M}_{G}(\cU)$ is \fim\ and idempotent. Then the following are equivalent:
	\begin{enumerate}
		\item $\mu$ is support transitive, i.e.~$\mu \ast p = \mu$ for all $p \in S(\mu)$;
		\item for any $p,q \in S(\mu)$ there exists $r \in S(\mu)$ such that $r \ast q = p$;
		\item $I_{\mu} = S(\mu)$, where $I_{\mu}$ is a minimal (closed) left ideal of $\left( S(\mu), \ast \right)$.
	\end{enumerate}
\end{proposition}
\begin{proof}
	(2) $\Leftrightarrow$ (3). By \cite[Remark 4.17]{chernikov2022definable}.
	
	\noindent (3) $\Rightarrow$ (1). By \cite[Corollary 4.16]{chernikov2022definable}.
	
	\noindent (1) $\Rightarrow$ (2).  Let $p,q \in S(\mu)$ be given. By Fact \ref{fac: supp is nice semigroup}, the map $f_p: S(\mu) \to S(\mu)$ defined via $f_p(s) := s \ast p$ is continuous. We will show that it has a  dense image. Then by compactness of $S(\mu)$ and continuity of $f_p$, the image of $f_p$ is also closed, hence $f_p$ is surjective --- proving the claim.
	
	Indeed, fix some formula $\theta(x,c) \in \cL_x(\cU)$ such that $\mu(\theta(x,c))>0$ and choose a small $M \prec \cU$ such that $\mu$ is $M$-invariant. By (1), $(\mu \ast p)(\theta(x,c)) = \mu(\theta(x,c)) > 0$.
	Let $b \models p|_{M c}$, then $\mu(\theta(x \cdot b, c)) > 0$. Hence there exists some $r \in S(\mu)$ such that $\theta(x \cdot b, c) \in r$. But then by definition $\theta(x,c) \in r \ast p$, hence $f_p(r) \in [\theta(x,c)]$. As $\theta(x,c)$ was arbitrary, this shows that the image of $f_p$ is dense.
\end{proof}

We know that this property holds in specific examples like the circle group (e.g., see \cite[Example 4.2]{chernikov2022definable}).

\subsection{Idempotent measures in stable theories, revisited}\label{sec: stable measures}

It is shown in \cite[Theorem 5.8]{chernikov2022definable} that every idempotent measure on a type-definable group in a stable theory is generically transitive. The proof consists of two ingredients: an analysis of the convolution semigroup on the support of an idempotent Keisler measure, and an application of a variant of Hrushovski's group chunk theorem for partial types due to Newelski \cite{N3}.

In this section we provide an alternative argument, implementing the strategy outlined at the beginning of Section \ref{sec: support trans} of working in the randomization.  This replaces the use of Newelski's theorem by a direct generalization of the proof for types in stable theories from Section \ref{subsection: stable theories}, and the only fact about the supports of idempotent measures that we will need is that they are support transitive.

To simplify the notation, in this section we will assume that $T$ is an $\cL$-theory expanding a group.
%\begin{proposition}
%	Let $T$ be stable and let $\mu \in \mathfrak{M}_{G}(\cU)$ be idempotent. Then $\mu$ is support transitive.
%\end{proposition}
%\begin{proof}
%	By \cite[Theorem 4.7]{chernikov2022definable}, if $I \subseteq S(\mu)$ is a closed two-sided ideal, then $I = S(\mu)$. ***
%\end{proof}
 We first recall the basic results about local ranks in continuous logic, from \cite{yaacov2010stability, ben2010continuous}. The following facts are proved  under more general hypothesis in Sections 7 and 8 of \cite{ben2010continuous}. 

%\begin{definition}[Rank in continuous logic] Let $M$ be a continuous first order theory.  Let $X$ be a compact topometric space. For a fixed $\varepsilon > 0$, we define a decreasing sequence of closed subsets of $X_{\varepsilon,\alpha}$ by induction: 
%\begin{enumerate} 
%\item $X_{\varepsilon,0} = X$. 
%\item $X_{\varepsilon, \alpha} = \bigcap_{\beta < \alpha} X_{\varepsilon, \beta}$ when $\alpha$ is a limit ordinal. 
%\item $X_{\varepsilon,\alpha +1} = \bigcap \{F \subseteq X_{\varepsilon, \alpha}: F$ is closed an diam$(X_{\varepsilon,\alpha} \backslash F) \leq \varepsilon\}$. 
%\item $X_{\varepsilon, \infty} = \bigcap_{\alpha} X_{\varepsilon,\alpha}$
%\end{enumerate} 
%For a non-empty subset $C \subseteq X$ we define its $\varepsilon$-Cantor-Bendixson rank in $X$ as: 
%\begin{equation*} \CB_{X,\varepsilon}(C) = \sup\{\alpha: C \cap X_{\varepsilon,\alpha} \neq \emptyset\} \in Ord \cup \{\infty\}. 
%\end{equation*} 
%\end{definition} 

\begin{fact}\label{fact:BenU} Suppose that $T$ is a continuous stable theory. Let $M \prec \cU \models T$.
\begin{enumerate} 
\item For any $p \in S_{x}(M)$ there exists a unique $M$-definable extension $p' \in S_{x}(\mathcal{U})$.
\item For every $\varepsilon > 0$ and every partitioned $\mathcal{L}$-formula $\varphi(x,y)$, there exists a rank function $\CB_{\varphi,\varepsilon}$, which we call the \emph{$\varepsilon$-Cantor-Bendixson rank}. More specifically, for any subset $A \subseteq \mathcal{U}$ and $p \in S_{\varphi}(A)$, $CB_{\varphi,\varepsilon}(p)$ is an ordinal. 
\item For any $r \in S_{\varphi}(M)$ and $s \in S_{\varphi}(\mathcal{U})$ such that $s \supseteq r$, $s$ is the unique definable extension of $r$ if and only if for every $\varepsilon > 0$,  $\CB_{\varphi,\varepsilon}(r) = \CB_{\varphi,\varepsilon}(s)$. 
\end{enumerate} 
\end{fact} 

\noindent The following proposition is a standard exercise from the previous fact. 

\begin{proposition}\label{prop:rank} Let $T$ be a continuous stable theory expanding a group. Let $\mathcal{U}$ be a monster model of $T$ and $M \prec \cU$ a small submodel.  For any partitioned $\mathcal{L}$-formula $\varphi(x,y)$ we let $\Delta_{\varphi} := \varphi(x;y,z) = \varphi(x \cdot z,y)$. Then: 
\begin{enumerate}
\item for any $r  \in S_{x}(\mathcal{U})$ and $g \in \mathcal{U}$,  $\CB_{\Delta_{\varphi},\varepsilon}(r|_{\Delta_{\varphi}}) = \CB_{\Delta_{\varphi},\varepsilon}(r \cdot g|_{\Delta_{\varphi}})$; 
\item for any $p  \in S_{x}(M)$ and $q \in S_{x}(\mathcal{U})$ such that $q \supset p$, we have that $q$ is the unique definable extension of $p$ if and only if for every partitioned $\mathcal{L}$-formula $\varphi(x,y)$ and for all $\varepsilon > 0$, we have $\CB_{\Delta_{\varphi}, \varepsilon}(q|_{\Delta_{\varphi}}) = \CB_{\Delta_{\varphi}, \varepsilon}(p|_{\Delta_{\varphi}})$.
\end{enumerate}
\end{proposition}

\begin{proof}
\begin{enumerate} 
\item Given $g \in \cU$, the map $m_{g}: S_{\Delta_{\varphi}}(\mathcal{U}) \to S_{\Delta_{\varphi}}(\mathcal{U})$ defined via $\varphi(x\cdot a,b)^{m_{g}(r)} = \varphi(x \cdot g \cdot a,b)^{r}$ is a bijective isometry. In other words, it is an automorphism of $S_{\Delta_{\varphi}}(\mathcal{U})$ as a topometric space, and computing the rank is unaffected. 
\item Follows directly from (3) of Fact \ref{fact:BenU}. \qedhere
\end{enumerate} 
\end{proof}

\noindent We refer to Section \ref{sec: fim meas and randomization} for notation regarding Keisler randomizations.

\begin{fact}\cite[Theorem 5.14]{ben2009randomizations}\label{fact:stable} If $T$ is stable, then its Keisler randomization $T^{R}$ is stable. 
\end{fact}

\begin{proposition} Suppose that $T$ is stable, $\mu \in \mathfrak{M}_{x}(\mathcal{U})$ is idempotent and support transitive. Then $\mu$ is generically transitive. 
\end{proposition} 

\begin{proof}

Let $\mathcal{V} \succ \mathcal{U}$ be a bigger monster model of $T$. Fix an atomless probability algebra $(\Omega,\mathcal{B},\mathbb{P})$ and consider the randomizations $\mathcal{U}^{\Omega} \prec \mathcal{V}^{\Omega} \prec \mathcal{C}$, where $\mathcal{C}$ is a monster model of $T^{R}$.  Given $\mu \in \mathfrak{M}_{x}(\mathcal{U})$ (note that $\mu$ is definably by stability), we let $r_{\mu}^{\mathcal{U}} \in S_{x}^{R}(\mathcal{C})$ be as defined in Fact \ref{fact:building1}. Similarly, given $\mu \in \mathfrak{M}_{x}(\mathcal{V})$, we let $r_{\mu}^{\mathcal{V}} \in S_{x}^{R}(\mathcal{C})$ be as defined in Fact \ref{fact:building1}, but with respect to $\mathcal{V}$ in place of $\mathcal{U}$.
%
%We let $r_{-}^{\mathcal{U}}:\mathfrak{M}_{x}(\mathcal{U}) \to S_{x}^{R}(\mathcal{C})$ and $r_{-}^{\mathcal{V}}:\mathfrak{M}_{x}(\mathcal{V}) \to S_{x}^{R}(\mathcal{C})$ be the relative constructions of Ben Yaacov's transfer map (Fact \ref{fact:building1}).

Let now $\mu \in \mathfrak{M}_{x}(\mathcal{U})$ be idempotent and support transitive. Since $T$ is stable, there is some small model $M \prec \cU$ such that $\mu$ is $M$-definable.  
 Let $\mu' \in \mathfrak{M}_x(\mathcal{V})$ be the unique $M$-definable extension of $\mu$. To show that $\mu$ is generically transitive, it suffices to prove that for every $p \in S(\mu)$ and $a \in \mathcal{V}$ such that $a \models p$, we have that $\mu' = \mu' \cdot a$. We let $\mathbf{p} := r_{\mu}^{\mathcal{U}}|_{\mathcal{U}^{\Omega}}$. By construction, $r_{\mu'}^{\mathcal{V}} \supseteq \mathbf{p}$ and $r_{\mu'}^{\mathcal{V}}$ is $M^{\Omega}$-definable. Since $T^{R}$ is stable (Fact \ref{fact:stable}), $r_{\mu'}^{\mathcal{V}}$ is the unique global definable extension of $\mathbf{p}$.

 We claim that then $r_{\mu' \cdot a}^{\mathcal{V}} \supseteq \mathbf{p}$. Indeed, let $\varphi(x,y)$ be an $\mathcal{L}$-formula and $h \in \mathcal{U}^{\Omega}_0$. If $\mathcal{A}$ is a partition of $\Omega$ for $h$, using that $\mu$ is support transitive we have:
 \begin{gather*}
 	(\mathbb{E}[\varphi(x,h)])^{r_{\mu' \cdot a}^{\mathcal{V}}} = \sum_{A \in \mathcal{A}} \mathbb{P}(A)  (\mu' \cdot a)(\varphi(x,h|_{A}))  
=\sum_{A \in \mathcal{A}} \mathbb{P}(A)  (\mu * p)(\varphi(x,h|_{A})) \\ 
= \sum_{A \in \mathcal{A}} \mathbb{P}(A)  \mu (\varphi(x,h|_{A})) 
= \sum_{A \in \mathcal{A}} \mathbb{P}(A)  \mu' (\varphi(x,h|_{A})) 
=  (\mathbb{E}[\varphi(x,h)])^{r_{\mu' }^{\mathcal{V}}}.
 \end{gather*}

 Likewise, it is straightforward to check that $r_{\mu' \cdot a}^{\mathcal{V}} = r_{\mu'}^{\mathcal{V}} \cdot f_{a}$, where $f_{a} \in \cU_0^{\Omega}$ is the constant random variable (i.e., $f_{a}: \Omega \to \mathcal{U}$ via $f_{a}(t) =a $ for all $t \in \Omega$) and $\cdot$ is the randomization of the multiplication of the group in $T$. 

Since local rank in the stable theory $T^{R}$ is translation invariant (Proposition \ref{prop:rank}), we conclude that $r_{\mu' \cdot a}^{\mathcal{V}}$ is the unique definable extension of $\mathbf{p}$. This implies that $r_{\mu'}^{\mathcal{V}} = r_{\mu' \cdot a}^{\mathcal{V}}$ and in turn, $\mu' = \mu' \cdot a$. This completes the proof.
\end{proof}

\begin{remark}
	
We expect that this approach could be adapted for groups definable in $o$-minimal structures (as their randomizations are known to be real rosy \cite{andrews2019independence}), by developing a stratified local thorn rank  in continuous logic and generalizing the proof for types in rosy (discrete) first order theories from Section \ref{sec: rosy types}. When $T$ is a simple theory, the randomization $T^{R}$ is NSOP$_1$ (but not necessarily simple) by \cite{BYChRa}. A local rank for NSOP$_1$ theories is proposed in \cite[Section 5]{chernikov2023transitivity}, but a workable stratified rank is lacking at the moment. We do not pursue these directions here.
\end{remark}

\section{Topological dynamics of $\mathfrak{M}_{x}^{\fs}(\mathcal{G},G)$ and $S_{x}^{\fs}(\mathcal{G},G)$ in NIP groups}\label{section: top dyn}

In this section, we will use slightly different notation from the rest of the paper, in order to preserve continuity with the earlier work and setup in \cite{chernikov2022definable, chernikov2023definable}. We let $G$ be an expansion of a group, and $\mathcal{G} \succ G$ a monster model. Throughout this section, we assume that $T:=\Th(G)$ has NIP.

It was demonstrated in \cite[Proposition 6.4]{chernikov2022definable} that then the spaces of global $\Aut(\mathcal{G}/G)$-invariant Keisler measures,  and Keisler measures which are finitely satisfiable in $G$ (denoted $\mathfrak{M}_{x}^{\inva}(\mathcal{G},G)$ and $\mathfrak{M}_{x}^{\fs}(\mathcal{G},G)$,  respectively) form left-continuous compact Hausdorff semigroups with respect to definable convolution (Definition \ref{def: definable conv}). Note that $(\mathfrak{M}_{x}^{\fs}(\mathcal{G},G),*)$ is a submonoid of $(\mathfrak{M}_{x}^{\inva}(\mathcal{G},G),*)$. By $ S_{x}^{\fs}(\mathcal{G},G)$ we denote the submonoid of $(\mathfrak{M}_{x}^{\fs}(\mathcal{G},G),*)$ consisting of all global types finitely satisfiable in $G$ (viewed as $\{0,1\}$-measures). 

In \cite[Theorem 6.11]{chernikov2023definable}, the first two authors described a minimal left ideal of  $(\mathfrak{M}_{x}^{\fs}(\mathcal{G},G),*)$ [and $(\mathfrak{M}_{x}^{\inva}(\mathcal{G},G),*)$] in terms of the Haar measure on an ideal (or Ellis) group of $(S_{x}^{\fs}(\mathcal{G},G),*)$ [resp. $(S_{x}^{\inva}(\mathcal{G},G),*)$]. However, this required a rather specific assumption that this ideal group is a compact topological group with the topology induced from $(S_{x}^{\fs}(\mathcal{G},G),*)$ [resp. $(S_{x}^{\inva}(\mathcal{G},G),*)$]. In this section, we obtain the same description in the case of $(\mathfrak{M}_{x}^{\fs}(\mathcal{G},G),*)$, but under a more natural (from the point of view of topological dynamics) assumption that the so-called \emph{$\tau$-topology} on some (equivalently, every) ideal group is Hausdorff (equivalently, the ideal group with the $\tau$-topology is a compact topological group). In fact, the revised Newelski's conjecture formulated by Anand Pillay and the third author in \cite[Conjecture 5.3]{KrPi2} predicts that the $\tau$-topology is always Hausdorff under NIP. In Section \ref{sec: rev Newelski conj}, we confirm this conjecture in the case when $G$ is countable, which is an important result by its own rights. In particular, in the case when $G$ is countable, our description of a minimal left ideal of $(\mathfrak{M}_{x}^{\fs}(\mathcal{G},G),*)$ does not require any assumption on the ideal group.

As discussed in the introduction, the $\tau$-topology plays an essential role in many important structural results in abstract topological dynamics, including the recent theorem of Glasner on the structure of tame, metrizable, minimal flows \cite{Gla18}. In fact, our proof of the revised Newelski's conjecture for countable $G$ will be deduced  using this theorem of Glasner.

The reason why our proof of the revised Newelski's conjecture requires the countability of $G$ assumption is to guarantee that certain flows of types are metrizable in order to be able to apply the aforementioned theorem of Glasner. The reason why we focus only on $(\mathfrak{M}_{x}^{\fs}(\mathcal{G},G),*)$ and $(S_{x}^{\fs}(\mathcal{G},G),*)$ (and not on $(\mathfrak{M}_{x}^{\inva}(\mathcal{G},G),*)$ and $(S_{x}^{\inva}(\mathcal{G},G),*)$) is that $(S_{x}^{\fs}(\mathcal{G},G),*)$ is isomorphic to the Ellis semigroup of the $G$-flow $S_{x}^{\fs}(\mathcal{G},G)$ and so we have the $\tau$-topology on the ideal group of  $S_{x}^{\fs}(\mathcal{G},G)$ at our disposal.
%and we will be able to apply Glasner's theorem to some associated metrizable $G$-flows. 
For the revised Newelski's conjecture we will also use a well-known general principle that NIP implies tameness for various flows of types \cite{CS, ibarlucia2016dynamical, KrRz}.

\subsection{Preliminaries from topological dynamics}

\begin{definition}
A $G$-flow is a pair $(G,X)$, where $G$ is an abstract group acting (on the left) by homeomorphisms on a compact Hausdorff space $X$.
\end{definition}

\begin{definition}
If $(G,X)$ is a flow, then its \emph{Ellis semigroup}, denoted by $E(G,X)$ or just $E(X)$, is the (pointwise) closure in $X^X$ of the set of functions $\pi_g\colon x\mapsto g\cdot x$ for $g\in G$. 
%(We frequently slightly abuse the notation and write $g$ for $\pi_g$, treating $G$ as if it was a subset of $E(G,X)$.)
\end{definition}
	
\begin{fact}(See e.g.~\cite{Aus})
If $(G,X)$ is a flow, then $E(X)$ is a compact left topological semigroup (i.e.~it is a semigroup with the composition as its semigroup operation, and the composition is continuous on the left, i.e.~for any $f \in E(X)$ the map $-\circ f$ is continuous). It is also a $G$-flow with $g\cdot f:= \pi_g \circ f$.
\end{fact}

The next fact is folklore. Thanks to this fact $E(E(X))$ is always identified with $E(X)$.
\begin{fact}\label{fact: E(E(X)) cong E(X)}
The function $\Phi \colon E(X) \to E(E(X))$ given by $\Phi(\eta):= l_\eta$, where    $l_\eta \colon E(X) \to E(X)$ is defined by $l_\eta(\tau):=\eta \circ \tau$, is an isomorphism of semigroups and $G$-flows.
\end{fact}

The following is a fundamental theorem of Ellis on the basic structure of Ellis semigroups (see e.g.~\cite[Corollary 2.10 and Propositions 3.5 and 3.6]{Ell} or Proposition 2.3 of \cite[Section I.2]{Gla76}). We will use it freely without an explicit reference.
	
\begin{fact}[Ellis' Theorem]\label{fac: Ellis theorem}
		Suppose $S$ is a compact Hausdorff left topological semigroup (e.g.~the enveloping semigroup of a flow). Then $S$ has a minimal left ideal $\cM$. Furthermore, for any such ideal $\cM$:
\begin{enumerate}
			\item
			$\cM$ is closed;
			\item
			for any element $a\in \cM$ and idempotent $u \in \cM$ we have $au = a$, and  $\cM=Sa=\cM a$;
			\item
			$\cM=\bigsqcup_u u\cM$, where $u$ ranges over all idempotents in $\cM$;  in particular, $\cM$ contains an idempotent;
			\item
			for any idempotent $u\in \cM$, the set $u\cM$ is a subgroup of $S$ with the neutral element $u$.
		\end{enumerate}
		Moreover, all the groups $u\cM$ (where $\cM$ ranges over all minimal left ideals and $u$ over all idempotents in $\cM$) are isomorphic. In the model theory literature, the isomorphism type of all these groups (or any of these groups) is called the {\em ideal} (or {\em Ellis}) {\em group} of $S$; if $S=E(G,X)$, we call this group the \emph{ideal} (or \emph{Ellis}) {\em group} of the flow $(G,X)$.
\end{fact}

We will use the following fact, which gives us an explicit isomorphism from Fact \ref{fac: Ellis theorem} between any two ideal groups in a given minimal left ideal. The context is as in Fact \ref{fac: Ellis theorem}.

\begin{fact}\label{fact: explicit isomorphism}
If $u$ and $v$ are idempotents in $\cM$, then $p \mapsto up$ defines an isomorphism $l_u \colon v \cM \to u\cM$.
\end{fact}

\begin{proof}
The map $l_u$ is a homomorphism, because $l_u(pq) = u(pq)= u(pu)q=(up)(uq)=l_u(p)l_u(q)$. In the same way, $l_v \colon u\cM \to v \cM$ given by $l_v(p): =vp$ is a homomorphism. And it is clear that $l_v$ is the inverse of $l_u$.
\end{proof}

%	\begin{proof}
%		Classical. See e.g.\ Corollary 2.10 and Propositions 3.5 and 3.6 of \cite{Ell69}, or Proposition 2.3 of \cite[Section I.2]{Gl76}.
%	\end{proof}

We will also need the following observation which was Lemma 3.5 in the first arXiv version of \cite{KLM} (the section with this result was removed in the published version).

\begin{lemma}\label{lemma: basic lemma from KLM} Let $S$ be  a compact left topological semigroup, $\cM$ a minimal left ideal of $S$, and $u \in \cM$ an idempotent. Then the closure $\overline{u\cM}$ of $u\cM$ is a (disjoint) union of ideal groups. In particular,  $\overline{u\cM}$ is a subsemigroup of $S$.
\end{lemma}

\begin{proof}
Note that for every $\eta\in\cM$, $\eta\cM$ is an ideal group. Namely, $\eta\in v\cM$ for some idempotent $v\in \cM$. Thus, $\eta\cM\subseteq v\cM$, but also $v\cM\subseteq\eta\cM$, because $v=\eta\eta^{-1}$, where $\eta^{-1}$ is the inverse of $\eta$ in the ideal group $v\cM$.
%In particular, $\eta\in\eta\M$.

Let now $\eta_0\in\overline{u\cM} \subseteq \mathcal{M}$ (as $\cM$ is closed). By the first paragraph, it suffices to prove that $\eta_0\cM\subseteq\overline{u\cM}$. Since $\eta_0u = \eta_0$, we have $\eta_0u\cM=\eta_0\cM$.
 Take any $\eta\in \eta_0\cM$. Then $\eta=\eta_0\eta'$ for some $\eta'\in u\cM$. Since  $\eta_0\in\overline{u\cM}$, we have that $\eta_0$ is the limit point of a net $(\eta_0^i)_i\subseteq u\cM$. By left continuity, $\eta=\eta_0\eta'=\lim_i \left( \eta_0^i\eta' \right)$, so $\eta\in\overline{u\cM}$ as $\eta_0^i\eta'\in u\cM$ for all $i$'s.
\end{proof}

Most of the statements in the next fact are contained in \cite[Section IX.1]{Gla76}. There, the author considers the special case of $X=\beta G$ and defines $\circ$ in a slightly different (but equivalent) way. However, as pointed out in \cite[Section 2]{KrPi1} and \cite[Section 1.1]{KPR18}, many of the proofs from \cite[Section IX.1]{Gla76} go through in the general context. A very nice exposition of this material (with all the proofs) in the general context can be found in Appendix A of \cite{Rze18}.

\begin{fact}[The $\tau$-topology on the ideal group in an Ellis semigroup]\label{fact: tau-topology}
		Consider the Ellis semigroup $E(X)$ of a flow $(G,X)$, let $\cM$ be a minimal left ideal of $E(X)$ and $u\in \cM$ an idempotent.
		\begin{enumerate}
			\item
			For each $a\in E(X)$, $B\subseteq E(X)$, we write $a\circ B$ for the set of all limits of nets $(g_ib_i)_i$, where $g_i\in G$ are such that $\lim _i g_i = a$ and $b_i\in B$.
			\item
			The formula $\cl_\tau(A):=(u\cM)\cap (u\circ A)$ defines a closure operator on $u\cM$. It can also be (equivalently) defined as $\cl_\tau(A)=u(u\circ A)$. We call the topology on $u\cM$ induced by this operator the {\em $\tau$-topology}.
			\item
			\label{it:lem:tau_top:nearlycont}
			If $(f_i)_i$ (a net in $u\cM$) converges to $f\in \overline{u\cM}$ (the closure of $u\cM$ in $E(X)$), then $(f_i)_i$ converges to $uf$ in the $\tau$-topology.
			\item
			The $\tau$-topology on $u\cM$ coarsens the subspace topology inherited from $E(X)$.
			\item
			$u\cM$ with the $\tau$-topology is a quasi-compact, $T_1$ semitopological group (that is, the group operation is separately continuous) in which the inversion is continuous.
%Consequently, $u\cM/H(u\cM)$ is a compact, Hausdorff group (see Fact~\ref{fct:derived_quotient}).
\item All the groups $u\cM$ (where $\cM$ ranges over all minimal left ideals of $E(X)$ and $u$ over all idempotents in $\cM$) equipped with the $\tau$-topology are isomorphic as semitopological groups. In particular, the map from Fact \ref{fact: explicit isomorphism} is a topological isomorphism.
		\end{enumerate}
\end{fact}

By the Ellis joint continuity theorem \cite{Ell57} and Fact \ref{fact: tau-topology}(5), we get the following

\begin{corollary}\label{corollary: T2 implies topological}
If the $\tau$-topology on $u\cM$ is Hausdorff, then $u\cM$ is a compact topological group.
\end{corollary}

The following result is Lemma 3.1 in \cite{KPR18}.

\begin{fact}\label{fact: strange_cont}
Let $(G,X)$ be a flow. Let $\mathcal{M}$ be a minimal left ideal of $E(X)$ and $u\in \mathcal{M}$ an idempotent. Then the function $f \colon \overline {u\cM}\to u\cM$ (where $\overline{u\cM}$ is the closure of $u\cM$ in the topology of $E(X)$) defined by the formula $f(\eta):=u\eta$ has the property that for any continuous function $h\colon u\cM\to X$, where $X$ is a regular topological space and $u\cM$ is equipped with the $\tau$-topology, the composition $h\circ f \colon \overline {u\cM}\to X$ is continuous, where $\overline{u\cM}$ is equipped with subspace topology from $E(X)$. In particular, if $u\cM$ is Hausdorff with the $\tau$-topology, then $f$ is continuous.
\end{fact}

In the model-theoretic context of the $G$-flow $S_{x}^{\fs}(\mathcal{G},G)$ (with the action of $G$ by left translations), the following fact is folklore (see e.g.~\cite{N2, Pillay}).

\begin{fact}\label{fact: folklore on Ellis semigroups}
The function $\Phi \colon S_{x}^{\fs}(\mathcal{G},G) \to E( S_{x}^{\fs}(\mathcal{G},G))$ given by $\Phi(p): = l_p$, where $l_p \colon  S_{x}^{\fs}(\mathcal{G},G) \to  S_{x}^{\fs}(\mathcal{G},G)$ is defined by $l_p(q):=p *q$, is an isomorphism of  semigroups and $G$-flows.
\end{fact}

Thanks to this fact, Fact \ref{fact: tau-topology} can be applied directly to $S_{x}^{\fs}(\mathcal{G},G)$ in place of  $E( S_{x}^{\fs}(\mathcal{G},G))$, which we do without further explanations.

Note, however, that Fact \ref{fact: folklore on Ellis semigroups} does not hold for $S_{x}^{\inva}(\mathcal{G},G)$ in place of $S_{x}^{\fs}(\mathcal{G},G)$, because the $G$-orbit of $\tp(e/\mathcal{G})$ need not be dense in  $S_{x}^{\inva}(\mathcal{G},G)$.

\subsection{Minimal left ideal of  $\mathfrak{M}_{x}^{\fs}(\mathcal{G},G)$}\label{sec: minimal left ideal of measures}

Recall that $G$ is an expansion of a group. Fix $\mathcal{G} \succ G$ which is $|G|^+$-saturated.
Recall that in this section we assume that $T = \Th(G)$ is NIP. 

\begin{definition}
	For a formula $\varphi(x) \in \cL(\cG)$ and $n>0$, define a new formula 

$$\alt_n(x_0,\dots, x_{n-1}; y) := \bigwedge_{i<n-1} \neg(\varphi(x_i y) \leftrightarrow \varphi(x_{i+1} y)).$$
\end{definition}

In the next fact, $\bar x$ stands for $(x_0,\dots,x_{n-1})$. By the proof of \cite[Proposition 2.6]{NIP2} (where, if $\varphi(x) \in \cL(\cG)$ is of the form $\psi(x; c)$ for some $\psi(x,y) \in \cL(\emptyset)$ and $c \in \cG^{y}$, $N$ is chosen depending on $\psi(x,y)$), we have:

\begin{fact}\label{fact: Borel definability of types}
Let $p \in  S_{x}^{\inva}(\mathcal{G},G)$ and $\varphi(x) \in \cL(\cG)$ be any formula. Let $S:= \{ b \in \mathcal{G}: \varphi(xb) \in p\}$. Then, there exists a positive $N<\omega$ such that $S= \bigcup_{n<N} A_n \cap B_{n+1}^c$, where $-^c$ denotes the complement of a set and
$$
\begin{array}{ll}
A_n:=\{ b \in \mathcal{G}: (\exists \bar x) (p^{(n)}|_G(\bar x) \wedge \alt_n(\bar x;b) \wedge \varphi(x_{n-1}b))\},\\
B_n:=\{ b \in \mathcal{G}: (\exists \bar x) (p^{(n)}|_G(\bar x) \wedge \alt_n(\bar x;b) \wedge \neg \varphi(x_{n-1}b))\}.
\end{array}
$$
\end{fact}

The following lemma is the key new step needed to adapt the arguments from \cite[Section 6.2]{chernikov2023definable} to our general setting here. From now on, let $\cM$ be a minimal left ideal in $(S_{x}^{\fs}(\mathcal{G},G),*)$ and $u \in \cM$ an idempotent. For $p,q \in S_{x}^{\fs}(\mathcal{G},G)$, we will typically write $pq$ instead of $p \ast q$.

\begin{lemma}\label{lemma: constructibility}
Let $\varphi(x) \in \cL(\cG)$ be any formula. Assume that the $\tau$-topology on the ideal group $u\cM$ is Hausdorff. Then the subset $[\varphi(x)] \cap u\cM$ of $u\cM$ is constructible, and so Borel (in the $\tau$-topology). 
\end{lemma}

\begin{proof}
By Fact \ref{fact: strange_cont} and the assumption that $u\cM$ is Hausdorff, the function  $f \colon \overline{u\cM}\to u\cM$ (where $\overline{u\cM}$ is the closure of $u\cM$ in the topology of $E(X)$) given by $f(\eta): = u\eta$ is continuous. Note that $\ker(f):=\{p \in  \overline{u\cM}: f(p)=u\}$ is a subsemigroup of $\overline{u\cM}$ which coincides with the set $\mathcal{J}$ of all idempotents in $ \overline{u\cM}$. This follows from Lemma \ref{lemma: basic lemma from KLM} and the fact that for every $v \in \mathcal{J}$ the restriction $f |_{v\cM} \colon v\cM \to u\cM$ is a group isomorphism by Fact \ref{fact: explicit isomorphism}.

Put 
$$\widetilde{S}:=f^{-1}[[\varphi(x)]]= \{p \in \overline{u\cM}: up \in [\varphi(x)]\}.$$

Let $\fG \succ \mathcal{G}$ be a monster model in which $\mathcal{G}$ is small. Let $\bar u \in  S_{x}^{\fs}(\fG,G)$ be the unique extension of $u$ to a type in $S_x(\fG)$ which is finitely satisfiable in $G$. Pick $a \models \bar u$ (in a yet bigger monster model). Put
$$S:=\{b \in \fG: \models \varphi(ab)\} = \{b \in \fG: \varphi(xb) \in \bar u\}.$$
Then $\widetilde{S}= \{p \in \overline{u\cM}: \models \varphi(ab) \textrm{ for all/some } b \in p(\fG)\} = \{ \tp(b/\mathcal{G}) \in \overline{u\cM}:  b \in S\}$.

Take $N$, $A_n$, and $B_n$ from Fact \ref{fact: Borel definability of types} applied for $\fG$ in place of $\mathcal{G}$ and for $p=\bar u$. Then $S= \bigcup_{n<N} A_n \cap B_{n+1}^c$. Define:
$$
\begin{array}{ll}
\widetilde{A}_n:=\{\tp(b/\mathcal{G}): b \in A_n \textrm{ and } \tp(b/\mathcal{G}) \in \overline{u\cM}\},\\
\widetilde{B}_n:=\{\tp(b/\mathcal{G}): b \in B_n \textrm{ and } \tp(b/\mathcal{G}) \in \overline{u\cM}\},\\
\widetilde{A}_n':= \ker(f) * \widetilde{A}_n := \left\{rs : r \in  \ker(f), s \in \widetilde{A}_n \right\}, \widetilde{B}_n':= \ker(f) * \widetilde{B}_n, \textrm{ and }\\
\widetilde{S}':= \bigcup_{n<N} \widetilde{A}_n' \cap \widetilde{B}_{n+1}'^{c}.
\end{array}
$$

Note that all the sets $\widetilde{A}_n,\widetilde{B}_n, \widetilde{A}_n', \widetilde{B}_n', \widetilde{S}'$ are contained in $\overline{u\cM}$, as $\overline{u\cM}$ is a semigroup by Lemma  \ref{lemma: basic lemma from KLM}.

\begin{clm*}
\begin{enumerate}
\item $\widetilde{S}= \bigcup_{n<N} \widetilde{A}_n \cap \widetilde{B}_{n+1}^c$.
\item $\ker(f) * \widetilde{S} = \widetilde{S}$.
\item $\ker(f) * \widetilde{A}_n'=\widetilde{A}_n'$, $\ker(f) * \widetilde{B}_n'=\widetilde{B}_n'$, and $\ker(f) * \widetilde{S}'=\widetilde{S}'$.
\item $\widetilde{A}_n \subseteq \widetilde{A}_n'$ and $\widetilde{B}_n \subseteq \widetilde{B}_n'$.
\item $\widetilde{S}=\widetilde{S}'$.
\item $f[\widetilde{A}_n]$ and $f[\widetilde{B}_{n}]$ are closed.
\item $f[\widetilde{S}']= \bigcup_{n<N} f[\widetilde{A}_n'] \cap f[\widetilde{B}_{n+1}']^{c} =\bigcup_{n<N} f[\widetilde{A}_n] \cap f[\widetilde{B}_{n+1}]^c$ is constructible.
\end{enumerate}
\end{clm*}

\begin{clmproof}
(1) Let $\rho \colon \fG \to S_x(\mathcal{G})$ be given by $\rho(b):=\tp(b/\mathcal{G})$. Note that all of the sets $A_n, B_n$ are $\Aut(\fG/\cG)$-invariant (more precisely, invariant over $G$ and the parameters of $\varphi(x)$ in $\cG$), hence they are unions of sets of realizations in $\fG$ of complete types over $\mathcal{G}$. Then we have 
$$\rho[S]=\rho\left[\bigcup_{n<N} A_n \cap B_{n+1}^c\right]= \bigcup_{n<N} \rho[A_n] \cap \rho[B_{n+1}]^c.$$ 
So $\widetilde{S}= \overline{u\cM} \cap \rho[S] = \bigcup_{n<N}  (\overline{u\cM} \cap  \rho[A_n]) \cap (\overline{u\cM} \cap \rho[B_{n+1}])^c = \bigcup_{n<N} \widetilde{A}_n \cap \widetilde{B}_{n+1}^c$.

(2) $(\subseteq)$ Take $p \in \ker(f)$ and $s \in \widetilde{S}$. Then $f(ps)=u(ps) = (up) s = us=f(s) \in [\varphi(x)]$, which by definition implies  $ps \in \widetilde{S}$. 

$(\supseteq)$  Take $p \in \widetilde{S}$. Then, by  Lemma \ref{lemma: basic lemma from KLM}, $p \in v\cM$ for some $v \in \mathcal{J}$ (where $\mathcal{J}$ is the set of all idempotents in $ \overline{u\cM}$). Then $p=vp$ and $v \in \ker(f)$, so $p \in \ker(f) * \widetilde{S}$.

(3) The first two equalities follow from the definition of $A_n'$ and $B_n'$ and the fact that $\ker(f)=\mathcal{J}$ satisfies $\mathcal{J}*\mathcal{J}=\mathcal{J}$ (as $uv=u$ for $u,v \in \mathcal{J}$). To show $(\subseteq)$ in the third equality, take any $p \in  \widetilde{A}_n' \cap \widetilde{B}_{n+1}'^{c}$ and $v \in \ker(f) = \mathcal{J}$. Then $vp \in \widetilde{A}_n'$ by the first equality. Moreover, $vp \in \widetilde{B}_{n+1}'^c$, as otherwise $vp \in \widetilde{B}_{n+1}'$, so $p=v_p vp \in \widetilde{B}_{n+1}'$ by the second equality, where $v_p \in \mathcal{J}$ is such that $p \in v_p \cM$, a contradiction. To see $(\supseteq)$, take any $p \in \widetilde{S}'$. We have that $p \in v\cM$ for some $v \in \mathcal{J}=\ker(f)$.  So $p =v p \in \ker(f) * \widetilde{S}'$.

(4) follows as on the last line of (3).

(5) $(\supseteq)$ Take $g \in \widetilde{S}'$. Then $g \in  \widetilde{A}_n' \cap \widetilde{B}_{n+1}'^{c}$ for some $n<N$. Hence, $g \in \ker(f) * h$ for some $h \in \widetilde{A}_n$. Since $g \in \widetilde{B}_{n+1}'^{c}$, by (3), we get that $h \in \widetilde{B}_{n+1}'^{c}$, so, by (4), $h \in \widetilde{B}_{n+1}^c$. Thus, $h \in \widetilde{A}_n \cap \widetilde{B}_{n+1}^c$ which is contained in $\widetilde{S}$ by (1). Hence, by (2), we conclude that $g \in \widetilde{S}$.

$(\subseteq)$ Suppose for a contradiction that there is some $g \in \widetilde{S} \setminus \widetilde{S}'$. Using (1) and (2), let $n<N$ be maximal for which there is $h \in \widetilde{A}_n \cap \widetilde{B}_{n+1}^c$ such that $g \in \ker(f) * h$. Then $h \in  \widetilde{S} \setminus \widetilde{S}'$ by (1) and (3). So $h \notin  \widetilde{A}_n' \cap \widetilde{B}_{n+1}'^{c}$. This together with $h \in \widetilde{A}_n$ and (4) implies that $h \in \widetilde{B}_{n+1}'$, so $h \in \ker(f) * h'$ for some $h' \in \widetilde{B}_{n+1}$. Then $g \in \ker(f) * \ker(f) * h'=\ker(f) * h'$, so $g=vh'$ for some $v \in \ker(f)=\mathcal{J}$. Choose $v' \in \mathcal{J}$ with $h' \in v'\cM$. Then $v'g=v'vh'=v'h'=h'$, and so $h' \in \ker(f) * g$ which is contained in $\widetilde{S}$ by (2) (as $g \in \widetilde{S}$). Therefore, by (1), $h' \in \widetilde{A}_m \cap \widetilde{B}_{m+1}^c$ for some $m<N$. Since $h' \in \widetilde{B}_{n+1}$, we get $m+1>n+1$. As $g \in \ker(f) * h'$, we get a contradiction with the maximality of $n$.

(6) follows as $f$ is continuous by Fact  \ref{fact: strange_cont} (this is the only place in the proof where we use the assumption that $u \cM$ is Hausdorff) and $\widetilde{A}_n, \widetilde{B}_n$ are closed.

(7) To show the first equality, it is enough to prove that $\widetilde{A}_n'$ and $\widetilde{B}_n'$ are unions of fibers of $f$. Let us prove it for $\widetilde{A}_n'$; the case of $\widetilde{B}_n'$ is the same. So consider any $p \in \widetilde{A}_n'$ and $q \in  \overline{u\cM}$ with $f(p)=f(q)$. We have $p \in v_p \cM$ and $q \in v_q \cM$ for some $v_p,v_q \in \mathcal{J}$. Then $f(v_q p) = u v_q p = up = f(p) = f(q)$ and  $v_q p,q \in v_q \cM$. So, since $f |_{v_q\cM}$ is an isomorphism, we get that $q=v_q p \in \ker(f) * p \subseteq \ker(f) * \widetilde{A}_n' = \widetilde{A}_n'$ by (3).

The second equality follows since:
$$
\begin{array}{ll}
f[\widetilde{A}_n']= f[\ker(f) * \widetilde{A}_n] = u * \ker(f) * \widetilde{A}_n = u * \widetilde{A}_n =f[\widetilde{A}_n],\\
f[\widetilde{B}_n']= f[\ker(f) * \widetilde{B}_n] = u * \ker(f) * \widetilde{B}_n = u * \widetilde{B}_n =f[\widetilde{B}_n].
\end{array}
$$
Hence, $f[\widetilde{S}']$ is constructible by (6).
\end{clmproof}

By item (5) of the claim, we get $[\varphi(x)] \cap u\cM = f[f^{-1}[[\varphi(x)]]] = f[\widetilde{S}]=f[\widetilde{S}']$, which is a constructible set by item (7) of the claim.
\end{proof}

\begin{proposition}\label{prop: def of inv meas on ideal group}
	Assume that the $\tau$-topology on $u\cM$ is Hausdorff. By Corollary \ref{corollary: T2 implies topological}, $u\cM$  is a compact topological group, so we have the unique normalized  Haar measure $h_{u\cM}$ on Borel subsets of $u\cM$. By Lemma \ref{lemma: constructibility}, the formula
$$\mu_{u\cM}(\varphi(x)):=h_{u\cM}([\varphi(x)] \cap u\cM)$$
yields a well-defined Keisler measure in $\mathfrak{M}_{x}(\mathcal{G})$ which is concentrated on $\overline{u\cM} \subseteq \cM$ (i.e.~with the support contained in  $\overline{u\cM}$).
\end{proposition}

Using this, the material from Lemma 6.9 to Corollary 6.12 of \cite{chernikov2023definable} goes through word for word.  In particular, we get the following lemma and the main theorem describing a minimal left ideal of the semigroup $\left(\mathfrak{M}_{x}^{\fs}(\mathcal{G},G), \ast \right)$, under a more natural assumption than the property \emph{CIG1} (requiring that $u \cM$ is compact with respect to the induced topology instead of the $\tau$-topology) assumed in \cite{chernikov2023definable}.

\begin{lemma}
Assume that the $\tau$-topology on $u\cM$ is Hausdorff. Then $\mu_{u\cM} * \delta_p= \mu_{u\cM}$ for all $p\in \cM$, where $\delta_p$ is the Dirac measure at $p$.
\end{lemma}

Let $\mathfrak{M}(\cM):=\{ \mu \in \mathfrak{M}_x(\mathcal{G}): S(\mu) \subseteq \mathcal{M}\}$.

\begin{theorem}\label{thm: min ideal of meas Hausd}
Assume that the $\tau$-topology on $u\cM$ is Hausdorff. Then $\mathfrak{M}(\cM) * \mu_{u\cM}$  is a minimal left ideal of $\mathfrak{M}_{x}^{\fs}(\mathcal{G},G)$, and  $\mu_{u\cM}$ is an idempotent which belongs to $\mathfrak{M}(\cM) * \mu_{u\cM}$.
\end{theorem}

In the next section we will see that the assumption that the $\tau$-topology on $u\cM$ is Hausdorff is always satisfied when $G$ is a countable NIP group.

\section{Revised Newelski's conjecture for countable NIP groups}\label{sec: rev Newelski conj}

The goal of this section is to prove the revised Newelski's conjecture (see \cite[Conjecture 5.3]{KrPi2}) working over a countable model. We consider here a standard context for this conjecture (originally from \cite[Section 4]{N1}, but see also \cite{chernikov2014external}) which is slightly more general than in the previous Section \ref{section: top dyn}. Let $M$ be a model of a NIP theory, $G$ a 0-definable group in $M$, and $N \succ M$ an $|M|^{+}$-saturated elementary extension. By $S_{G,\ext}(M)$ we denote the space of all complete external types over $M$ concentrated on $G$, i.e.~the space of ultrafilters of externally definable subsets of $G$. It is a $G$-flow with the action given by left translation, which is naturally isomorphic as a $G$-flow to the $G$-flow $S_G^{\fs}(N,M)$ of all complete types over $N$ concentrated on $G$ and finitely satisfiable in $M$. Note that the previous context $S_x^{\fs}(\mathcal{G},G)$ of Section \ref{section: top dyn} is a special case when $M=G$ and $N=\mathcal{G}$. On $S_G^{\fs}(N,M)$ we have the left continuous semigroup operation defined in the same way as for $S_x^{\fs}(\mathcal{G},G)$ (i.e.~$p*q:=\tp(ab/N)$ for any $a \models p$, $b \models q$ such that $\tp(a/N,b)$ is finitely satisfiable in $M$), and Fact \ref{fact: folklore on Ellis semigroups} still holds for it. 

The following revised Newelski's conjecture was stated in \cite[Conjecture 5.3]{KrPi2}.

\begin{conjecture}\label{conjecture: revised Newelski's conjecture}
Assume that $T$ is NIP. Let $\cM$ be a minimal left ideal of  $S_G^{\fs}(N,M)$ and $u \in \cM$ an idempotent. Then the $\tau$-topology on $u\cM$  is Hausdorff.
\end{conjecture}
A background  around this conjecture, including an explanation that it is a weakening of Newelski's conjecture, is given in the introduction, and in more details in the long paragraph preceding Conjecture 5.3 in \cite{KrPi2} and short paragraph following it.

In order to prove Conjecture \ref{conjecture: revised Newelski's conjecture} for countable $M$, first we will deduce from the main theorem of \cite{Gla18} on the structure of tame,  metrizable, minimal flows that each such flow has Hausdorff ideal group. Then Conjecture \ref{conjecture: revised Newelski's conjecture} for countable $M$ will follow using this fact and a presentation of $S_G^{\fs}(N,M)$ as an inverse limit of certain metrizable flows. The topological dynamical material below is rather standard, but it requires recalling quite a few notions and basic facts about them, and making some observations.\\

From now on, until we say otherwise, we are in the general abstract context of $G$-flows and homomorphisms between them, where $G$ is an arbitrary abstract (not necessarily definable in a NIP theory) group. We let $(G,X)$, $(G,Y)$, etc.~be $G$-flows, which we will sometimes denote simply as $X,Y$, etc.

%Recall from Fact \ref{fact: on pi-tilde} that every $G$-flow epimorphism $\pi \colon X \to Y$ induces the $G$-flow and semigroup epimorphism $\tilde{\pi} \colon E(X) \to E(Y)$ given by $\tilde{\pi}(\eta)(\pi(x)):=\pi(\eta(x))$ with some good properties which we will be using.

For a proof of the following fact see \cite[Proposition 5.41]{Rze18}.

\begin{fact}\label{fact: on pi-tilde}
Let $\pi \colon X \to Y$ be an epimorphism of $G$-flows. Then $\tilde{\pi} \colon E(X) \to E(Y)$ given by $\tilde{\pi}(\eta)(\pi(x))=\pi(\eta(x))$ is a well-defined semigroup and $G$-flow epimorphism. If $\cM$ is a minimal left ideal of $E(X)$ and $u \in \cM$ an idempotent, then $\tilde{\pi}[\cM]$ is a minimal left ideal of $E(Y)$ and $\tilde{\pi}(u)$ is an idempotent in $\tilde{\pi}[\cM]$. Moreover, $\tilde{\pi}|_{u\cM} \colon u\cM \to \tilde{\pi}(u)\tilde{\pi}[\cM]$ is a group epimorphism and topological quotient map with respect to the $\tau$-topologies.
\end{fact}

\begin{remark}\label{remark: uniqueness of epimorphism}
If $(G,X)$ and $(G,Y)$ are flows for which there exists a semigroup and flow epimorphism $\Phi \colon E(X) \to E(Y)$, then it is unique. In particular, the epimorphism $\tilde{\pi}$ in Fact \ref{fact: on pi-tilde} does not depend on the choice of the epimorphism $\pi$.
\end{remark}

\begin{proof}
 Since $\Phi$ is a semigroup epimorphism, it satisfies $\Phi(\id_X)=\id_Y$. Let $g_X \in E(X)$ be the left translation by $g$, and similarly $g_Y \in E(Y)$. Then $\Phi(g_X)=\Phi(g\id_X)=g\Phi(\id_X)=g\id_Y=g_Y$. By definition, $E(X)$ is the closure of $\{g_X: g \in G\}$, so $\Phi$ is unique (as it is continuous).
\end{proof}

Let $\beta G$ be the Stone-\v{C}ech compactification of $G$ and $\mathcal{U}_e$ the principal ultrafilter at $e$. As e.g.~explained on page 9 of \cite{Gla76}, the $G$-ambit $(\beta G,\mathcal{U}_e)$ (where by a {\em $G$-ambit} we mean a $G$-flow with a distinguished point with dense orbit) is universal, and so there is a unique  left continuous semigroup operation on $\beta G$  extending the action of $G$ by left translation. (In fact, it is precisely the $*$ operation on $S_{G,\ext}(M)$ for $M:=G$ expanded by predicates for all subsets of $G$). For any flow $(G,X)$, universality of $\beta G$ also yields a unique action of the semigroup $(\beta G,*)$ on $X$ which is left-continuous and extends the action of $G$, and which we will denote by $\cdot$. Fix a minimal left ideal $\cM$ of $\beta G$ and an idempotent $u \in \cM$. Using this action, for any $G$-flow $(G,X)$ with a distinguished point $x_0 \in X$ such that $u\cdot x_0=x_0$ (note that such an $x_0$ always exists), the {\em Galois group} of $(X,x_0)$ is defined as:
$$\Gal(X,x_0) := \{ p \in u\cM: p\cdot x_0=x_0\} \leq u\cM,$$
it is a $\tau$-closed subgroup of $u \cM$ (see \cite[Page 13 and 120]{Gla76}). In the topological dynamics literature, this group is sometimes called the {\em Ellis group} of $(X,x_0)$, e.g.~see \cite[Page 13]{Gla76}.

There is an obvious semigroup and $G$-flow epimorphism $\Phi \colon \beta G \to E(X)$ given by $\Phi(p)(x):= p \cdot x$. It is unique by Remark \ref{remark: uniqueness of epimorphism}.
As  in Fact \ref{fact: folklore on Ellis semigroups}, $\beta G$  is naturally isomorphic to  $E(\beta G)$ via $p \mapsto l_p$, where $l_p(q)=p*q$. Using this identification, for any $G$-flow epimorphism $f \colon \beta G \to X$, the induced map $\tilde{f} \colon \beta G \to E(X)$  from Fact \ref{fact: on pi-tilde}  coincides with $\Phi$. 
%In particular, $\Phi |_{u\cM} \colon u\cM \to \Phi(u)\Phi[\cM]$ is a group epimorphism and topological quotient map.
%is given by $\tilde{f}(p)(f(q))=f(p*q)$.

\begin{remark}\label{remark: basic properties of Phi}
Let $\cM$ and $u \in \cM$ be as above. 
%Let $(G,X)$ be a flow and $x_0 \in X$ satisfy $u \cdot x_0=x_0$. 
Let $(G,X)$ be a flow and $\Phi \colon \beta G \to E(X)$ the unique epimorphism defined above.
Let $\mathcal{N}:=\Phi[\cM]$ and $v:=\Phi(u)$.
\begin{enumerate}
\item  For every $x \in X$ with $u \cdot x=x$, $\ker(\Phi |_{u\cM}) \subseteq \Gal(X,x)$.
\item For every $x \in \ima(v)$, $u \cdot x= x$.
\item  $\bigcap_{x \in \ima(v)} \Gal(X,x) = \ker(\Phi |_{u\mathcal{M}})$.
\item  $\Phi |_{u\cM} \colon u\cM \to v\mathcal{N}$ is a group epimorphism and topological quotient map with respect to the $\tau$-topologies.
\end{enumerate}
\end{remark}

\begin{proof}
It is clear that $\mathcal{N}$ is a minimal left ideal of $E(X)$, $v \in \mathcal{N}$ an idempotent, and $\Phi |_{u\cM} \colon u\cM \to v \mathcal{N}$ a group epimorphism.

(1) Take $p \in \ker(\Phi |_{u\cM})$, i.e.~$\Phi(p)=v$. Then $p \cdot x=\Phi(p)(x)= v(x) = \Phi(u)(x)=u \cdot x= x$. Hence, $p \in \Gal(X,x)$.

(2) Take $x \in \ima(v)$, i.e.~$x = v(y)$ for some $y \in X$. Then $u \cdot x = v(v(y)) = (v \circ v)(y)=v(y) = x$.

(3) The inclusion $(\supseteq)$ follows from (1) and (2). For the opposite inclusion, consider any $p \in \bigcap_{x \in \ima(v)} \Gal(X,x)$. In order to show that $p \in \ker(\Phi |_{u\cM})$, it is enough to check that $\Phi(p)|_{\ima(v)}= \id_{\ima(v)}$ (because for any $\eta \in \mathcal{N}$ we have $\eta=\eta v$, and so the map $v\mathcal{N} \to \Sym(\ima(v))$ given by $\eta \mapsto \eta |_{\ima(v)}$ is injective, in fact a group monomorphism, and $v|_{\ima(v)} = \id_{\ima(v)}$ as $v$ is an idempotent). But this is trivial by the choice of $p$: $\Phi(p)(v(y))=p \cdot (v(y))=v(y)$.

(4) In the situation when $(G,X)$ has a dense orbit (which is for example the case when $(G,X)$ is minimal), this follows from Fact \ref{fact: on pi-tilde}, the existence of an epimorphism $f \colon \beta G \to X$ (as $\beta G$ is a universal $G$-ambit), and the observation that $\Phi=\tilde{f}$ made just before Remark \ref{remark: basic properties of Phi}. In general, it follows from the straightforward generalization of  Fact \ref{fact: on pi-tilde} stated in \cite[Fact 2.3]{KLM}.
\end{proof}

\begin{definition}
A $G$-flow epimorphism $\pi \colon X \to Y$ is {\em almost 1-1} if the set $X_0:=\{x \in X : \pi^{-1}[\pi(x)] =\{x\}\}$ is dense in $X$.
\end{definition}

\begin{remark}\label{remark: preservation of minimality under almost 1-1}
If $(G,Y)$ is minimal and $\pi \colon X \to Y$ is almost 1-1, then $(G,X)$ is also minimal.
\end{remark}

\begin{proof}
We will show that for every $x \in X$ we have $X_0 \subseteq E(X)x := \{ \eta(x): \eta \in E(X)\}$. This implies that $E(X)x=X$ for every $x \in X$ (because $X_0$ is dense in $X$ and $E(X)x$ is closed in $X$), which means that $(G,X)$ is minimal.

So fix $x \in X$, and consider any $x_0 \in X_0$. Since $(G,Y)$ is minimal, we can find $\tau \in E(Y)$ such that $\tau(\pi(x))=\pi(x_0)$. Pick $\eta \in E(X)$ satisfying $\tilde{\pi}(\eta)=\tau$. Then $\pi(x_0) = \tau(\pi(x))=\tilde{\pi}(\eta)(\pi(x))=\pi(\eta(x))$. Since $x_0 \in X_0$, we conclude that $\eta(x)=x_0$. 
\end{proof}

\begin{lemma}\label{lemma: almost 1-1}
If $(G,Y)$ is minimal and $\pi \colon X \to Y$ is almost 1-1, then the group homomorphism $\tilde{\pi} |_{u\cM} \colon u\cM \to \tilde{\pi}(u)\tilde{\pi}[\cM]$ is a topological isomorphism (in the $\tau$-topologies), where $\cM$ is a minimal left ideal of $E(X)$ and $u \in \cM$ an idempotent.
\end{lemma}

\begin{proof}
By Remark \ref{remark: preservation of minimality under almost 1-1}, $(G,X)$ is minimal, so $X=\cM x$ for every $x \in X$. Pick $x_0 \in X_0$; then there is $\eta_0 \in \cM$ such that $x_0 \in \ima(\eta_0)$.
Choose an idempotent $v \in \cM$ so that $\eta_0 \in v\cM$. Then $\ima(\eta_0)=\ima(v)$, so $x_0 \in \ima(v)$, and hence $v(x_0)=x_0$ by idempotence of $v$.

Since the diagram
\begin{center}
			\begin{tikzcd}
				v\cM\ar[d, two heads, "\tilde{\pi}"]\ar[r,  "u\circ"]&u\cM\ar[d, two heads, "\tilde{\pi}"] \\
				\tilde{\pi}(v)\tilde{\pi}[\cM]\ar[r, "\tilde{\pi}(u)\circ"]&\tilde{\pi}(u)\tilde{\pi}[\cM]
			\end{tikzcd}
		\end{center}
commutes (by definition of $\tilde{\pi}$) and, by Fact \ref{fact: tau-topology}(6), the horizontal arrows are isomorphisms of semitopological groups, it is enough to show that $\tilde{\pi} |_{v\cM} \colon v\cM \to \tilde{\pi}(v) \tilde{\pi}[\cM]$ is an isomorphism of semitopological groups. By Fact \ref{fact: on pi-tilde}, it is a group epimorphism and topological quotient map, so it remains to show that it is injective.

Suppose for a contradiction that $\ker(\tilde{\pi} |_{v\cM})$ is non-trivial, i.e.~there is $\eta \in v\cM \setminus \{v\}$ such that $\tilde{\pi}(\eta)=\tilde{\pi}(v)$. Then $\eta v=\eta \ne v = vv$, so $\eta |_{\ima(v)} \ne v |_{\ima(v)}$. On the other hand, by the first sentence of the proof, $\ima(v) =v\cM (x_0)$. So there is $\eta' \in v\cM$ with $\eta \eta'(x_0) \ne v\eta'(x_0)=\eta'(x_0)$, and hence $(\eta')^{-1} \eta \eta'(x_0) \ne x_0$ (where $(\eta')^{-1}$ is the inverse of $\eta'$ computed in $v\cM$). As $x_0 \in X_0$, $\eta \in \ker(\tilde{\pi} |_{v\cM})$ and $\tilde{\pi}|_{v\cM}$ is a group morphism, we get $\pi(x_0) \ne \pi((\eta')^{-1} \eta \eta'(x_0))=\tilde{\pi}((\eta')^{-1} \eta \eta')(\pi(x_0)) 
= \tilde{\pi}(v)(\pi(x_0)) = \pi(v(x_0)) = \pi(x_0)$ (where the last equality follows from the first paragraph) --- a contradiction.
\end{proof}

A pair of points $(x_1,x_2)$ of a flow $(G,X)$ is called {\em proximal} if there is $\eta \in E(X)$ such that $\eta(x_1)=\eta(x_2)$; it is called {\em distal} if $x_1=x_2$ or $(x_1,x_2)$ is not proximal. Let $P$ denote the collection of all proximal pairs of points in $X$. The flow $(G,X)$ is said to be {\em proximal} when $P=X \times X$, and {\em distal} when $P =\Delta_X:=  \{(x, x) : x\in  X\}$.

\begin{fact}\label{fact: ideal groups of proximal flows}
The ideal group of every proximal flow is trivial.
\end{fact}

\begin{proof}
Let $\cM$ be a minimal left ideal of $E(X)$ and $u\in \cM$ an idempotent.
From \cite[Chapter I, Proposition 3.2(3)]{Gla76}, it follows that each pair of points in $\ima(u):=u[X]$ is distal, so, by proximality, $\ima(u)$ is a singleton, say $\ima(u)= \{x_0\}$. Then for any $p$ in $u\cM$, say $p = uh$ with $h \in \cM$, and any $x \in X$, we have $p(x) = u (h(x)) = x_0$. So $p = u$, hence $u\cM = \{u\}$. 
\end{proof}

Whenever $\pi \colon X \to Y$ is a $G$-flow epimorphism, let 
$$R_\pi:= \{ (x_1,x_2) \in X^2: \pi(x_1)=\pi(x_2)\}.$$

\begin{definition} A $G$-flow epimorphism $\pi \colon X \to Y$ is said to be:
\begin{enumerate}
\item {\em equicontinuous} (or {\em almost periodic}) if for every $\varepsilon$ which is an open neighborhood
of the diagonal $\Delta_X: = \{(x, x) : x\in  X\} \subseteq X \times X$, there exists a neighborhood $\delta$ of
$\Delta_X$ such that $g(\delta \cap R_\pi)\subseteq \varepsilon$ for every $g \in G$;
\item {\em distal} if $P\cap R_\pi=\Delta_X$.
\end{enumerate}
\end{definition}

It is sometimes assumed that the flows in the definition of equicontinuous extensions are minimal.
On page 100 of \cite{Gla76}, Glasner defines almost periodic extensions of minimal flows in a different way. A proof that both definitions are equivalent for minimal flows can be found in \cite[Chapter 14, Theorem 1]{Aus}.

The following remark is well-known and follows from an argument on page 4 of \cite{Gla76}.

\begin{remark}
An equicontinuous epimorphism of flows is distal. 
\end{remark}

\begin{proposition}\label{proposition: equicontinuity and Hausdorfness}
If $\pi \colon X \to Y$ is an equicontinuous epimorphism of minimal $G$-flows and $(G,Y)$ has a trivial ideal group, then the ideal group of $(G,X)$ is Hausdorff (with respect to the $\tau$-topology). 
\end{proposition}

\begin{proof}

Choose a minimal left ideal $\mathcal{M}$ in $\beta G$, an idempotent $u \in \mathcal{M}$, and an element $x \in X$ with $u\cdot x=x$. Let $\Phi \colon \beta G \to E(X)$ be the unique semigroup and $G$-flow epimorphism considered before and in Remark \ref{remark: basic properties of Phi}. Put $\mathcal{N}_X:=\Phi[\cM]$ and $v_X:=\Phi(u)$. So $v_X\mathcal{N}_X$ is the ideal group of $(G,X)$.
Put $y:=\pi(x)$. Then $u \cdot y =y$ (as $u\cdot y=\lim g_i\pi(x)=\pi(\lim(g_ix)) = \pi(u \cdot x)=\pi(x)=y$, where $(g_i)_i$ is a net in $G$ converging to $u$ in $\beta G$). Let  $\mathcal{N}_Y:=\tilde{\pi}[\mathcal{N}_X]$ and $v_Y:=\tilde{\pi}(v_X)$. So $v_Y\mathcal{N}_Y$ is the ideal group of $(G,Y)$ which is trivial by assumption.

Clearly $\Phi':=\tilde{\pi} \circ  \Phi \colon \beta G \to E(Y)$ is the unique semigroup and $G$-flow epimorphism from $\beta G$ to $E(Y)$, and $\Phi' |_{u\cM} \colon u\cM \to v_Y\mathcal{N}_Y$ is a group epimorphism.  Hence, as $v_{Y}\mathcal{N}_{Y}$ is trivial, $\ker(\Phi'|_{u\cM}) = u\cM$. As by Remark  \ref{remark: basic properties of Phi}(1) $\ker(\Phi' |_{u\cM})\subseteq \Gal(Y,y)$, we conclude that $\Gal(Y,y) =u\cM$.

 Let $F$ be a $\tau$-closed subgroup of $u \cM$, and let 
 $$H(F) := \bigcap\left\{\cl_\tau(V \cap F) : V \textrm{ is a } \tau \textrm{-neighborhood of }u \textrm{ in } u\cM \right\}.$$
 As $\pi$ is almost periodic, \cite[Chapter IX, Theorem 2.1(4)]{Gla76} yields $H\left( \Gal(Y,y) \right) \subseteq \Gal(X,x)$, and 
 together with the conclusion of the last paragraph this implies $H(u\cM) \subseteq \Gal(X,x)$. Note that we proved it for any $x \in X$ with $u \cdot x =x$, in particular for any $x \in \ima(v_X)$ (by  Remark  \ref{remark: basic properties of Phi}(2)). On the other hand, by Remark \ref{remark: basic properties of Phi}(3),  $\bigcap_{x \in \ima(v_X)} \Gal(X,x) = \ker(\Phi |_{u\mathcal{M}})$. Hence, $H(u\cM) \subseteq \ker(\Phi |_{u\mathcal{M}})$.

By Remark \ref{remark: basic properties of Phi}(4), $\ker(\Phi |_{u\mathcal{M}})$ is $\tau$-closed. So, by \cite[Chapter IX, Theorem 1.9(3)]{Gla76} and the conclusion of the last paragraph, we get that $u\cM/\ker(\Phi |_{u\mathcal{M}})$ is Hausdorff with the quotient topology of the $\tau$-topology on $u\mathcal{M}$.  Since by Remark \ref{remark: basic properties of Phi}(4) we know that $u \mathcal{M}/\ker(\Phi|_{u\mathcal{M}}) \cong v_X\mathcal{N}_X$ as semitopological groups, we conclude that $v_X\mathcal{N}_X$ is Hausdorff.
\end{proof}

There are many equivalent definitions of tame flows (see Theorems 2.4, 3.2 and Definition 3.1 in \cite{Gla23}). We give the one which immediately points to a strong connection with the NIP property in model theory. 

A sequence $(f_n)_{n<\omega}$  of real valued functions on a set $X$ is said to be {\em independent}
if there exist real numbers $a < b$ such that
$$\bigcap_{n\in P} f_n^{-1}[(-\infty,a)] \cap \bigcap_{n \in Q} f_n^{-1}[(b,\infty)] \ne \emptyset$$
for all finite disjoint subsets $P, Q$ of $\omega$.

\begin{definition}\label{def: tame}
%A flow $(G,X)$ is {\em tame} if for every $f \in C(X)$ (i.e. a continuous real valued function on $X$) the family of translates $\{gf: g \in G\}$ does not contain an independent sequence.
Let $(G,X)$ be a flow.
A function $f \in C(X)$ (i.e.~a continuous real valued function on $X$) is {\em tame} if the  family of translates $\{gf: g \in G\}$ does not contain an infinite independent sequence (where $(gf)(x):=f(g^{-1}x)$). The flow $(G,X)$ is {\em tame} if all functions in $C(X)$ are tame.
\end{definition}

The following fact is a part of the information contained in the main theorem (Theorem 5.3) of \cite{Gla18} on the structure of tame, metrizable, minimal flows.

\begin{fact}\label{fact: Glasner's theorem}
Let $(G,X)$ be a tame, metrizable, minimal flow. Then there exists the following commutative diagram of $G$-flow epimorphisms
\begin{center}
			\begin{tikzcd}
				& \tilde{X}\ar[ld, two heads, "\eta"]\ar[dd,  two heads, "\pi"] & X^*\ar[l, two heads, "\theta^*"]\ar[d, two heads, "\iota"]\\
				X & & Z\ar[d,two heads, "\sigma"]\\
				& Y & Y^* \ar[l, two heads, "\theta"],
			\end{tikzcd}
		\end{center}
where: 
\begin{enumerate}
\item $(G,\tilde{X})$ is minimal;
\item $(G,Y)$ is proximal;
\item $\theta, \theta^*, \iota$ are almost 1-1;
\item $\sigma$ is equicontinuous.
\end{enumerate}
\end{fact}

We discussed all of the notions and facts above in order to deduce the following corollary.

\begin{corollary}\label{corollary: main corollary of Glasner's result}
The $\tau$-topology on the ideal group of any tame, metrizable, minimal flow is Hausdorff.
\end{corollary}

\begin{proof}
We will be referring to items (1)--(4) in Fact \ref{fact: Glasner's theorem}. By (1), (3) and Remark \ref{remark: preservation of minimality under almost 1-1}, $(G,X^*)$ is minimal, and so are $(G,Z)$, $(G,Y^*)$, and $(G,Y)$ as homomorphic images of  $(G,X^*)$. By (2) and Fact \ref{fact: ideal groups of proximal flows}, the ideal group of $(G,Y)$ is trivial, and so is the ideal group of $(G,Y^*)$ by (3) and Lemma \ref{lemma: almost 1-1}. Hence, using (4) and Proposition \ref{proposition: equicontinuity and Hausdorfness}, we get that the ideal group of $(G,Z)$ is Hausdorff, and so is the ideal group of $(G,X^*)$ by (3) and Lemma \ref{lemma: almost 1-1}. Therefore, the ideal groups of $(G,\tilde{X})$ and $(G,X)$ are both Hausdorff by Fact  \ref{fact: on pi-tilde}, because they are quotients of a compact topological group (namely, the ideal group of $(G,X^*)$) by closed subgroups.
\end{proof}

To apply this general corollary to our model-theoretic context, we need one more general observation, namely Lemma \ref{lemma: ideal groups of X and M}. To prove it, we have to recall a description of the $\tau$-closure that was stated as Lemma 3.11 in the first arXiv version of \cite{KLM} (the relevant section of \cite{KLM} was removed in the published version).

\begin{fact}\label{fact: description of the tau-closure}
Let $(G,X)$ be a flow, $\cM$ a minimal left ideal of $E(X)$, and $u \in \cM$ an idempotent. Then for every $A \subseteq u\cM$, the $\tau$-closure $\cl_\tau(A)$ can be described as the set of all limits contained in $u\cM$ of nets $(\eta_ia_i)_i$ such that $\eta_i \in \cM$, $a_i \in A$, and $\lim_i \eta_i = u$.
\end{fact}
\begin{proof}
Consider $a\in \cl_{\tau}(A)$. Then, by the definition of the $\tau$-topology, there are nets $(g_i)_i\subseteq G$ and $(a_i)_i\subseteq A$ such that $\lim_i g_i=u$ and $\lim_i g_ia_i=a$. Note that $ua_i=a_i$, as $a_i\in A\subseteq u\cM$. Put $\eta_i :=g_iu\in \cM$ for all $i$. By left continuity, we have that $\lim_i \eta_i=\lim_i g_iu=(\lim_i g_i)u=uu=u$. Furthermore, $\lim_i \eta_i a_i=\lim_i g_i u a_i=\lim_i g_i a_i=a$.

Conversely, consider any $a \in u\cM$ for which there are nets $(\eta_i)_i\subseteq \cM$ and $(a_i)_i\subseteq A$ such that $\lim_i \eta_i=u$ and $\lim_i \eta_i a_i = a$. Since each $\eta_i$ can be approximated by elements of $G$ and the semigroup operation is left continuous, one can find a subnet $(a_j')_j$ of $(a_i)_i$ and a net $(g_j)_j \subseteq G$ such that $\lim_j g_j =u$ and $\lim_j g_ja_j' =a$, which means that $a \in \cl_\tau(A)$.
\end{proof}

\begin{lemma}\label{lemma: ideal groups of X and M}
Let $(G,X)$ be a flow, $\cM$ a minimal left ideal in $E(X)$. Then there exists a minimal left ideal $\mathcal{N}$ of $E(\cM) $ and a semigroup and $G$-flow isomorphism from $\cM$ to $\mathcal{N}$. In particular, the ideal groups of the $G$-flows $X$ and $\cM$ are isomorphic as semitopological groups (with the $\tau$-topologies).
\end{lemma}

\begin{proof}
Denote the semigroup operation on $E(X)$ by $*$ and on $E(\cM)$ by $\circ$ (although both are compositions of functions, but on different sets).
Let $f \colon \cM \to E(\cM)$ be given by $f(p):=l_p$, where $l_p(q) := p * q$. It is easy to check that $f$ is a semigroup and $G$-flow monomorphism. Put $\mathcal{N}:=f[\cM]\subseteq E(\cM)$. Thus, $f \colon \cM \to \mathcal{N}$ is a semigroup and $G$-flow isomorphism. We need to check that $\mathcal{N}$ is a minimal left ideal in $E(\mathcal{M})$. For that, first note that $\cM$ is a minimal left ideal in itself. Indeed, if $\cM'$ is a left ideal in $\cM$, then for any $p \in \cM'$ and $q \in E(X)$, taking an idempotent $u \in \cM$ such that $ p \in u\cM$, we have $p=u*p$, and so $q * p = q*(u*p)=(q*u)*p \in \cM'$ as $q*u \in \cM$ ($\cM$ is a left ideal in $E(X)$). Hence, $\cM'$ is a left ideal in $E(X)$ which is contained in $\cM$, so it must be equal to $\cM$ by minimality of $\cM$. Since $\cM$ is a minimal left ideal in itself, we get that
\begin{enumerate}
\item $\mathcal{N}$ is a minimal left ideal in itself.
\end{enumerate}
On the other hand, 
\begin{enumerate}
\item[(2)] $\mathcal{N}$ is a left ideal in $E(\cM)$.
\end{enumerate}
To prove (2), consider any $\eta \in E(\cM)$ and $p \in \mathcal{M}$, and we need to show that $\eta \circ l_p \in \mathcal{N}$. Using limits of nets, we easily see that there is $\tilde{\eta} \in E(E(X))$ such that $\tilde{\eta} |_{\cM}=\eta$. For any $q \in \cM$ we have: $(\eta \circ l_p)(q) = \eta(l_p(q))=\eta(p * q) = \tilde{\eta}(p * q) = \eta' * (p * q) =(\eta' * p) * q = l_{\eta' * p}(q) =f(\eta' * p)(q)$, where $\eta' \in E(X)$ satisfies $\tilde{\eta}(\tau) = \eta' * \tau$ for all $\tau \in E(X)$ (such $\eta'$ exists by Fact \ref{fact: E(E(X)) cong E(X)}). Hence, $\eta \circ l_p =f(\eta' *p) \in \mathcal{N}$ (and  $\eta'*p \in \cM$ as $\cM$ is a left ideal).

By (1) and (2), we get that $\mathcal{N}$ is a minimal left ideal of $E(\mathcal{M})$.

Thus, $f(u)\mathcal{N}$ is the ideal group of $\cM$, and $f |_{u\cM} \colon u\cM \to f(u)\mathcal{N}$ is a group isomorphism, where $u \in \cM$ is an idempotent. This isomorphism is topological (with respect to the $\tau$-topologies) by Fact \ref{fact: description of the tau-closure} (which expresses the $\tau$-closure in terms of the semigroup operation and convergence within the minimal left ideal in question) and the above observation that $f\colon \cM \to \mathcal{N}$ is an isomorphism of left topological semigroups.
\end{proof}

We can finally prove Conjecture \ref{conjecture: revised Newelski's conjecture} for countable $M$.

\begin{theorem}\label{thm: revised New conj}
Let $M$ be a countable model of a theory with NIP and $N \succ M$ be $\aleph_1$-saturated. Let $\cM$ be a minimal left ideal of  $S_G^{\fs}(N,M)$ and $u \in \cM$ an idempotent. Then the $\tau$-topology on $u\cM$  is Hausdorff.
\end{theorem}

\begin{proof}
Let $M^{\ext}$ be the Shelah expansion of $M$ obtained by adding predicates for all externally definable subsets of $M^n$ for all $n<\omega$. By Shelah's theorem \cite{She} (see also \cite{chernikov2013externally}), we know that $\Th(M^{\ext})$ has quantifier elimination, NIP, and all types in $S(M^{\ext})$ are definable (i.e.~all externally definable subsets of $M^{\ext}$ are definable). It follows that the Boolean algebra of externally definable subset of $G$ with respect to the original language coincides with the Boolean algebra of definable subsets of $G$ in the sense of the expanded language. Hence, $S_{G,\ext}(M)=S_G(M^{\ext})$. Thus, without loss of generality, we may assume that $M$ is a countable model of an NIP theory such that all types in $S(M)$ are definable. Then $S_{G,\ext}(M)=S_G(M)$, and the semigroup operation on $S_G(M)$ is given by $p*q = \tp(ab/M)$, where $a \models p$, $b \models q$, and $\tp(a/M,b)$ is finitely satisfiable in $M$.

It is well-known, and observed first time in the introduction of \cite{CS}, that NIP implies that $(G,S_G(M))$ is a tame flow. (A standard way to see it is to note that, by NIP, the characteristic functions of all the clopens in $S_G(M)$ are tame (in the sense of Definition \ref{def: tame}) and separate points, and so, by Stone-Weierstrass theorem, they generate a dense subalgebra of $C(S_G(M))$; then use the fact that tame functions on $S_G(M)$ form a closed subalgebra of $C(S_{G}(M))$ to conclude that all functions in $C(S_G(M))$ are tame.) However, $(G,S_G(M))$ is neither metrizable (even when the original language of $M$ was countable, the expanded language of $M^{\ext}$ is usually uncountable) nor minimal, so we cannot apply Corollary \ref{corollary: main corollary of Glasner's result} directly to $(G,S_G(M))$.

Let $\Delta$ range over all finite collections of definable subsets of $G$. For any such $\Delta$, let $\mathcal{B}_G(\Delta)$ be the Boolean $G$-algebra (so closed under left translations by the elements of $G$) of subsets of $G$ generated by $\Delta$, and denote by $S_{G,\Delta}(M)$ the space of all  ultrafilters of $\mathcal{B}_G(\Delta)$. Note that $S_{G,\Delta}(M)$ is naturally a $G$-flow (the action is by left translations), and $S_G(M) \cong \varprojlim_\Delta S_{G,\Delta}(M)$ as $G$-flows. Let $\cM$ be any minimal left ideal (and so minimal subflow) of $S_{G}(M)$. Let $\cM_\Delta \subseteq S_{G,\Delta}(M)$ be the image of $\cM$ under the restriction map. It is clearly a minimal subflow of $S_{G,\Delta}(M)$, and the above isomorphism induces a $G$-flow isomorphism $\cM \cong    \varprojlim_\Delta \cM_\Delta$ (see \cite[Lemma 6.42]{Rze18}).

Since $(G,S_G(M))$ is tame, so is $(G,S_{G,\Delta}(M))$ as a quotient of $S_G(M)$, and so is $(G,\cM_\Delta)$ as a subflow of  $S_{G,\Delta}(M)$ (using Definition \ref{def: tame} and Tietze's extension theorem, or see e.g.~\cite[Fact 4.20]{KrRz}). Moreover, $S_{G,\Delta}(M)$ is metrizable since $\mathcal{B}_G(\Delta)$ is countable by finiteness of $\Delta$ and countability of $G \subseteq M$ (and this is the only place where we assume that $G$ is countable). Hence, $\cM_\Delta$ is metrizable. 

Summarizing, $(G,\cM_\Delta)$ is a tame, metrizable, minimal flow, and hence its ideal group is Hausdorff by Corollary \ref{corollary: main corollary of Glasner's result}.

On the other hand, since $\cM \cong    \varprojlim_\Delta \cM_\Delta$, \cite[Lemma 6.42]{Rze18} implies that the ideal group of $\mathcal{M}$ equipped with  the $\tau$-topology is topologically isomorphic to an inverse limit of the ideal groups (with the $\tau$-topologies) of the flows $\cM_\Delta$.

By the last two paragraphs, we conclude that the ideal group of $\cM$ is Hausdorff.

Finally, by Fact \ref{fact: folklore on Ellis semigroups}, $E(S_G(M)) \cong S_G(M)$ as semigroups and $G$-flows, so $\cM$ can be identified as a $G$-flow with a minimal left ideal of $E(S_G(M))$.
Then, by Lemma \ref{lemma: ideal groups of X and M}, the ideal groups of $S_G(M)$ and $\cM$ are topologically isomorphic. Therefore, by the last paragraph, the ideal group of $S_G(M)$ is Hausdorff.
\end{proof}

Combining Theorem \ref{thm: min ideal of meas Hausd} with Theorem \ref{thm: revised New conj} we thus get (using the notation in Theorem \ref{thm: min ideal of meas Hausd}):
\begin{corollary}\label{cor: min ideal ctbl group}
	Assume that $G$ is countable and $\Th(G)$ is NIP, and  let $\cM$ be a minimal left ideal in $(S_{x}^{\fs}(\mathcal{G},G),*)$ and $u \in \cM$ an idempotent. 
	Then $\mathfrak{M}(\cM) * \mu_{u\cM}$  is a minimal left ideal of $\left(\mathfrak{M}_{x}^{\fs}(\mathcal{G},G), \ast \right)$, and  $\mu_{u\cM}$ is an idempotent which belongs to $\mathfrak{M}(\cM) * \mu_{u\cM}$.
\end{corollary}

\section{Acknowledgements}

We are very grateful to the referee for suggestions on improving the paper, and especially for providing a proof of a stronger version of the chain condition result in Proposition \ref{prop: new chain cond gs} in Section \ref{sec: chain conds} (in particular answering our question from the previous version of the paper on sufficiency of generic stability of $p$ as opposed to $p^{(n)}$ for all $n$). 
We thank Aaron Anderson, Martin Hils, Anand Pillay, Sergei Starchenko, Atticus Stonestrom and Mariana Vicaria for helpful conversations. Chernikov was partially supported by the NSF CAREER grant DMS-1651321 and by the NSF Research Grant DMS-2246598. Gannon was supported by the Fundamental Research Funds for the Central Universities, Peking University, grant no.~7100604835 and by the National Natural Science Fund of China, grant no.~12501001. Krupi\'{n}ski was supported by the Narodowe Centrum Nauki grant no.~2016/22/E/ST1/00450.
\printbibliography

\end{document}